\documentclass[a4paper,12pt]{amsart}

\usepackage{hyperref, tikz, amssymb}

\newtheorem{thm}{Theorem}
\newtheorem{prp}[thm]{Proposition}
\newtheorem{cor}[thm]{Corollary}
\newtheorem{lem}[thm]{Lemma}
\theoremstyle{definition}
\newtheorem{dfn}[thm]{Definition}
\newtheorem{ntn}[thm]{Notation}
\theoremstyle{remark}
\newtheorem{rmk}[thm]{Remark}

\newcommand{\Bb}{\mathcal{B}}

\newcommand{\Kk}{\mathcal{K}}
\newcommand{\Ll}{\mathcal{L}}

\newcommand{\Nn}{N}
\newcommand{\Uu}{\mathcal{U}}
\newcommand{\Hh}{\mathcal{H}}

\newcommand{\G}{G}

\newcommand{\R}{\mathcal{R}}

\newcommand{\CC}{\mathbb{C}}
\newcommand{\NN}{\mathbb{N}}
\newcommand{\TT}{\mathbb{T}}
\newcommand{\ZZ}{\mathbb{Z}}

\newcommand{\go}{G^{(0)}}
\newcommand{\ho}{H^{(0)}}

\newcommand{\Io}[1]{\operatorname{Iso}(#1)^\circ}

\newcommand{\supp}{\operatorname{supp}}
\newcommand{\osupp}{\operatorname{supp^\circ}}
\newcommand{\lsp}{\operatorname{span}}
\newcommand{\clsp}{\overline{\lsp}}
\newcommand{\id}{\operatorname{id}}
\newcommand{\dom}{\operatorname{dom}}
\newcommand{\ran}{\operatorname{ran}}

\newcommand{\Aut}{\operatorname{Aut}}
\newcommand{\Ad}{\operatorname{Ad}}
\newcommand{\op}{\operatorname{op}}
\newcommand{\Stab}{\operatorname{Stab}}
\newcommand{\ess}{\operatorname{ess}}

\newcommand{\TS}[1]{\overline{#1}}

\numberwithin{equation}{section}

\title[Reconstruction of groupoids and $C^*$-rigidity]
{Reconstruction of groupoids and $C^*$-rigidity of dynamical systems}

\author[T.M. Carlsen]{Toke Meier Carlsen}
\address[T.M. Carlsen]{Department of Science and Technology\\University of the Faroe Islands
N\'oat\'un 3\\ FO-100 T\'orshavn\\the Faroe Islands}
\email{toke.carlsen@gmail.com}

\author[E. Ruiz]{Efren Ruiz}
\address[E. Ruiz]{Department of Mathematics\\University of Hawaii,
Hilo\\200 W. Kawili St.\\
Hilo, Hawaii\\
96720-4091 USA}
\email{ruize@hawaii.edu}

\author[A. Sims]{Aidan Sims}
\address[A. Sims]{School of Mathematics and Applied Statistics\\
University of Wollongong\\
NSW 2522\\
Australia}
\email{asims@uow.edu.au}

\author[M. Tomforde]{Mark Tomforde}
\address[M. Tomforde]{Mark Tomforde\\
	Department of Mathematics\\
	University of Colorado\\
	Colorado Springs\\
	CO 80918-3733 USA}
\email{mtomford@uccs.edu}

\thanks{This research was initiated while the authors were participating in the research
program \emph{Classification of operator algebras: complexity, rigidity, and dynamics}
at the Mittag-Leffler Institute, January--April 2016. It also benefitted significantly
from the support of the MATRIX@Melbourne Research Program \emph{Refining $C^*$-algebraic
invariants for dynamics using $KK$-theory, July 18--29, 2016} and from the Intensive Research
Program \emph{Operator algebras: dynamics and interactions} at the Centre de Recerca Matem\`atica
in Barcelona. This research was supported by
Australian Research Council grant DP150101598 and by grants from the Simons Foundation
(\#279369 to Efren Ruiz and \#527708 to Mark Tomforde).}
\begin{document}

\begin{abstract}
We show how to construct a graded locally compact Hausdorff \'etale groupoid from a
$C^*$-algebra carrying a coaction of a discrete group, together with a suitable abelian
subalgebra. We call this groupoid the extended Weyl groupoid. When the coaction is
trivial and the subalgebra is Cartan, our groupoid agrees with Renault's Weyl groupoid.
We prove that if $\G$ is a second-countable locally compact \'etale groupoid carrying a
grading of a discrete group, and if the interior of the trivially graded isotropy is
abelian and torsion free, then the extended Weyl groupoid of its reduced $C^*$-algebra is
isomorphic as a graded groupoid to $\G$. In particular, two such groupoids are isomorphic
as graded groupoids if and only if there is an equivariant diagonal-preserving
isomorphism of their reduced $C^*$-algebras. We introduce graded equivalence of
groupoids, and establish that two graded groupoids in which the trivially graded isotropy
has torsion-free abelian interior are graded equivalent if and only if there is an
equivariant diagonal-preserving Morita equivalence between their reduced $C^*$-algebras.
We use these results to establish rigidity results for a number of classes of dynamical
systems, including all actions of the natural numbers by local homeomorphisms of locally
compact Hausdorff spaces.
\end{abstract}

\maketitle

\setcounter{tocdepth}{1}


\renewcommand*{\thethm}{\Alph{thm}}

\section*{Introduction}
%

\subsection*{Background} The use of operator algebras to encode dynamics goes all the way back to the foundational
results of Murray and von Neumann on the group von Neumann algebra construction
\cite{MvN4}. Crossed-product algebras and their generalisations have played a crucial
role in both von Neumann algebra theory and $C^*$-algebra theory ever since. Recently,
particularly since the work of Cuntz and Krieger \cite{CK} on operator-algebraic
representations of shifts of finite type, and connections with Bowen--Franks theory
\cite{BF}, significant strides have been made in the direction of $C^*$-rigidity of
dynamical systems. In broad terms this is the principle that dynamical systems can be
recovered, up to a suitable notion of equivalence, from associated $C^*$-algebraic data.

A seminal result in this direction was Krieger's celebrated theorem \cite{Krieger} showing that
nonsingular ergodic actions of $\ZZ$ are classified up to orbit equivalence by isomorphism of the
associated von Neumann factors. This was soon followed by Cuntz and Krieger's construction
\cite{CK} of $C^*$-algebras from irreducible shifts of finite type and R{\o}rdam's proof
\cite{Rordam} that stable isomorphism of Cuntz--Krieger algebras classifies irreducible shifts of
finite type up to equivalence via the combination of flow equivalence and the so-called Cuntz
splice on directed graphs. Later, building on work of Boyle \cite{Boyle}, Giordano--Putnam--Skau
\cite{GPS} proved the remarkable result that for minimal homeomorphisms of the Cantor set, flip
conjugacy, continuous orbit equivalence, and diagonal-preserving\footnote{Throughout this paper,
we often use the phrase ``diagonal-preserving" to describe homomorphisms interwining distinguished
abelian sulgebras of $C^*$-algebras; we do not necessarily mean $C^*$-diagonals in the technical
sense of Kumjian \cite{Kumjian:CJM86}.} isomorphism of the associated crossed-product
$C^*$-algebras are equivalent. Tomiyama \cite{Tomiyama} and Boyle--Tomiyama \cite{BT} subsequently
proved that topologically free homeomorphisms of the Cantor set (minimal or not), are continuously
orbit equivalent if and only if they each decompose as a disjoint union of two subsystems one pair
of which are conjugate and the other pair of which are flip-conjugate. These results introduced,
in particular, \emph{diagonal-preserving} $C^*$-isomorphisms as a key ingredient in
$C^*$-algebraic rigidity of topological dynamics.

The importance of diagonal-preserving isomorphism in operator algebras associated to
groupoids goes back further. Feldman and Moore \cite{FM1, FM2, FM3} proved that a Borel
equivalence relation $R$ can be reconstructed from the pair consisting of its associated
von Neumann algebra $M$ and the canonical Cartan subalgebra $D \subseteq M$, and that
every Cartan pair of von Neumann algebras arises from such an $R$ and a Borel 2-cocycle
on $R$. In his thesis \cite{Renault80}, Renault introduced a notion of a Cartan
subalgebra of a $C^*$-algebra, and proved that a topologically principal \'etale groupoid
$\G$ and a continuous cocycle $c$ on $\G$ can be recovered from the Cartan pair
$(C^*(\G), C_0(\G^{(0)}))$, and that every Cartan pair arises this way. Subsequently
Kumjian \cite{Kumjian:CJM86} refined Renault's notion of a twisted groupoid $C^*$-algebra
and showed that Renault's theorem extended to these more-general twists, in the setting
of principal \'etale groupoids. Later in \cite{Ren2008}, Renault further extended
Kumjian's results to topologically principal groupoids. Renault's machinery, and
techniques from groupoid homology, underpinned Matsumoto and Matui's remarkable recent
results \cite{MM} that irreducible (two-sided) shifts of finite type are flow equivalent
if and only if there is a diagonal-preserving isomorphism of the stabilisations of the
associated Cuntz--Krieger algebras, and that the corresponding one-sided shifts are
continuously orbit equivalent if and only if the Cuntz--Krieger algebras are isomorphic
in a diagonal-preserving way.

These results all require topologically freeness of actions, or topologically principal
groupoids. While seemingly fairly natural, these conditions are not generic. General
$C^*$-algebraic rigidity theorems for homeomorphisms, local homeomorphisms, and more
general group actions, require a version of Renault's theory for
non-topologically-principal groupoids. Ad hoc results in this direction have been
achieved recently for graph $C^*$-algebras \cite{BCW, CR}, but there is no general
theory available.

\subsection*{Our results}
A version of Renault's theory is impossible for general groupoids: for example, there is
no way to distinguish the groupoids $\ZZ_4$ and $\ZZ_2 \oplus \ZZ_2$ using their
$C^*$-algebras. Our key observation, inspired by techniques developed in \cite{BCW, CR2}
is that this obstruction disappears if we insist that the fibres of the interior of the
isotropy bundle of the groupoid are torsion-free and abelian; in the 1-unit case, we can
then use \cite[Theorem~8.57]{HofmannMorris} to recover the groupoid as the quotient of the unitary group of its $C^*$-algebra by
the connected component of the identity. Groupoids of this sort include all groupoids
arising from actions of $\ZZ^k$ by homeomorphisms or of $\NN^k$ by local homeomorphisms;
this includes all Cantor systems, all groupoids associated to graphs and $k$-graphs and
their topological analogues, and many other natural examples.

Our key results, following the idea introduced in \cite{ABHS}, deal with groupoids $\G$
graded by cocycles $c$ into discrete groups $\Gamma$. The reduced $C^*$-algebra
$C^*_r(G)$ then carries a natural coaction $\delta_c$ of $\Gamma$. We prove that if the
interior of the isotropy in $c^{-1}(\id_\Gamma)$ (where $\id_\Gamma$ is the identity element of $\Gamma$) is torsion-free and abelian, then $\G$ and $c$
can be reconstructed from $C^*_r(\G)$, the subalgebra $C_0(\G^{(0)})$ and the coaction
$\delta_c$. We also obtain a $C^*$-algebraic characterisation of groupoid equivalences
that respect gradings in an appropriate sense.

Incorporating gradings and coactions has significant advantages \cite{CR}, but the reader
may, in the first instance, wish to keep in mind the case where $\Gamma$ is the trivial
group, so $c$ is trivial: our results in this situation are special cases of our general
theorems, but the statements are simpler, and still have substantial new content (we
summarise our results in this setting in Section~\ref{sec:summary ungraded}).

We detail the consequences of our results for dynamical systems, demonstrating the breadth of
their applicability: we develop $C^*$-algebraic characterisations of appropriate notions of
stabiliser-preserving orbit equivalence and topological conjugacy for group actions whose
essential stabilisers are torsion-free and abelian, generalising Li's
continuous-orbit-equivalence-rigidity theorem \cite[Theorem 1.2]{Li}; we characterise
stabiliser-preserving continuous orbit equivalence and eventual conjugacy both of local
homeomorphisms and of the associated stabilised systems in the same terms, generalising
\cite[Theorem~5.1]{BCW}, \cite[Corollary~6.3]{CEOR}, \cite[Theorem 4.1 and Theorem~5.1]{CR},
\cite[Theorem~2.3]{MM}, and \cite[Theorem~2]{Tomiyama}; and we generalise Boyle and Tomiyama's
theorem \cite[Theorem 3.6]{BT} to arbitrary homeomorphisms of second-countable compact Hausdorff
spaces. Our results have many other potential applications, particularly to topological graphs, to
$k$-graphs and their topological analogues, to actions of $\ZZ^k$ on locally compact spaces, and
to actions of $\NN^k$ by local homeomorphisms. In fact, our results have recently been applied to
shift spaces in \cite{BC}, to $k$-graphs in \cite{CR3}, and to self-similar groups in \cite{Yi}.

\subsection*{Pr\'ecis} The paper is laid out as follows. We give some very brief background in
Section~\ref{sec:background}, and recall some facts about normalisers in Section~\ref{sec:weakly cartan pairs}. Our main results are
about $C^*$-algebras $A$ carrying coactions $\delta$ of discrete groups and an abelian $C^*$-subalgebra $D$ containing an approximate unit of the generalised fixed-point algebra $A^\delta$,
but in Section~\ref{sec:summary ungraded}, we state these results as they apply to
ungraded groupoids and $C^*$-algebras. In Section~\ref{sec:sdcs} we show how to construct
a locally compact Hausdorff \'etale groupoid $\Hh(A, D, \delta)$ from a separable
$C^*$-algebra $A$, a coaction $\delta$ of a discrete group on $A$ and an abelian
$C^*$-subalgebra $D$ of the generalised fixed-point algebra $A^\delta$. Our construction builds
on those of Kumjian \cite{Kumjian:CJM86} and Renault \cite{Renault80, Ren2008}, but
extends them by incorporating the structure of the unitary groups of the fibres of the
relative commutant $D_{A^\delta}'$ in $A^\delta$. Our main application is to groupoid
$C^*$-algebras, but  this general construction yields an interesting invariant for
general systems $(A, D, \delta)$. In Section~\ref{sec:intiso} we prove that if $\G$ is an
\'etale groupoid in which the interior of the isotropy is abelian, then the $C^*$-algebra
of the interior of the isotropy is a maximal abelian subalgebra of $C^*_r(\G)$; this is
needed in Section~\ref{sec:reconstruction}, but also answers a question left open in
\cite{BNRSW}.

In Section~\ref{sec:reconstruction} we prove our main theorem: if $(\G, c)$ is a graded second-countable locally compact Hausdorff étale
groupoid and the interior of the isotropy in $c^{-1}(\id)$ is torsion-free and abelian,
then $\Hh(C^*_r(\G), C_0(\G^{(0)}), \delta_c) \cong \G$ via an isomorphism that
intertwines gradings. Consequently, two such graded groupoids $(\G_1, c_1)$ and $(\G_2, c_2)$ are isomorphic if and only if there is an isomorphism between the corresponding triples $(C^*_r(\G_1), C_0(\G_1^{(0)}), \delta_{c_1})$ and $(C^*_r(\G_2), C_0(\G_2^{(0)}), \delta_{c_2})$. Sections \ref{sec:group
actions}--\ref{sec:homeomorphisms} detail the consequences of
Section~\ref{sec:reconstruction} for group actions, for local homeomorphisms, and for
homeomorphisms, including extensions of Li's rigidity theorem, Matsumoto and Matui's
theorem about continuous orbit equivalence, and Boyle and Tomiyama's theorem.

The final two sections deal with Morita equivalence. In Section~\ref{sec:Eq Me} we
introduce equivariant Morita equivalence of triples $(A_i, D_i, \delta_i)$ as above, and
prove that such a Morita equivalence induces a graded equivalence of extended Weyl
groupoids. In Section~\ref{sec:E vs Me}, we apply this result to triples $(C^*_r(\G_i),
C_0(\G_i^{(0)}), \delta_{c_i})$ corresponding to graded groupoids $(G_i, c_i)$ such that
the interior of the isotropy in each $c_i^{-1}(\id)$ is torsion-free and abelian. We prove
that $(G_1, c_1)$ and $(G_2, c_2)$ are graded equivalent if and only if the associated
triples $(C^*_r(\G_i), C_0(\G_i^{(0)}), \delta_{c_i})$ are equivariantly Morita
equivalent. Restricting attention to ample groupoids, we use the results of \cite{CRS} to
relate these notions to versions of graded Kakutani equivalence and to graded stable
isomorphism of groupoids. We finish by detailing our applications to surjective local
homeomorphisms and their dilations.

\subsection*{Acknowledgments}
In a preliminary version of this paper, we introduced the notion of a weakly Cartan subalgebra
(which was different from the notion introduced in \cite{EP}). In this version, we have added
Lemma~\ref{lem:semidiag} which shows that what we had previously called a weakly Cartan subalgebra
of $A^\delta$ is simply an abelian $C^*$-subalgebra of $A^\delta$ containing an approximate unit
for of $A^\delta$. So we have eliminated the phrase ``weakly Cartan subalgebra" in this version.
We thank Bartosz Kwa\'sniewski and Ralf Meyer for showing us (1)~and~(2) in
Lemma~\ref{lem:semidiag}. We would also like to thank the two referees for valuable comments and
suggestions.

\renewcommand*{\thethm}{\arabic{thm}}
\numberwithin{thm}{section}

\section{Background}\label{sec:background}

We establish some brief background and notational conventions for \'etale groupoids and
their reduced $C^*$-algebras, $C_0(X)$-algebras, coactions on $C^*$-algebras, and
$C^*$-algebraic Morita equivalence. For more details see \cite{SimsNotes},
\cite[Appendix~C]{Williams:cp}, \cite{EKQR}, and \cite{tfb}.

When $\Gamma$ is a group, then we use $\id_\Gamma$ to denote its identity element, and when $X$ is a set, then we let $(e_x)_{x\in X}$ be an orthonormal basis for the Hilbert space $l^2(X)$.

\subsection{\'Etale groupoids}

A groupoid $G$ is the set of morphisms of a small category with inverses; the space of
identity morphisms is called the \emph{unit space} and denoted $\G^{(0)}$, and the set of
composable pairs of morphisms is denoted $G^{(2)}$. A locally compact Hausdorff groupoid
is a groupoid $\G$ with a locally compact Hausdorff topology under which the inverse and
multiplication maps are continuous. A map $c$ from $\G$ to a discrete group $\Gamma$ is a
\emph{cocycle} if $c(\eta_1\eta_2)=c(\eta_1)c(\eta_2)$ for $(\eta_1,\eta_2)\in \G^{(2)}$
(this forces $c(\G^{(0)}) = \{\id_\Gamma\}$ and $c(\eta^{-1}) = c(\eta)^{-1}$). A \emph{graded groupoid} is a pair $(G,c)$ consisting of a groupoid $G$ and a cocycle $c:G\to\Gamma$. We say that $(G,c)$ is \emph{trivially graded} if $\Gamma$ is the trivial group.

We write $r,s$ for the range and source maps $r(\eta) = \eta\eta^{-1}$ and $s(\eta) =
\eta^{-1}\eta$ from $\G$ to $\G^{(0)}$. A subset $X \subseteq \go$ is \emph{full} if
$\{r(\eta) : s(\eta) \in X\} = \go$. The \emph{isotropy} of $\G$ is
$\operatorname{Iso}(\G) := \{\eta \in \G : r(\eta) = s(\eta)\}$. We say that $\G$ is
\emph{\'etale} if $r$ (equivalently $s$) is a local homeomorphism from $\G$ to $\go$, and
that it is \emph{ample} if it is \'etale and $\G^{(0)}$ is totally disconnected. For $u
\in \G^{(0)}$, we write $\G_u := s^{-1}(u)$ and $\G^u := r^{-1}(u)$. Using that $r,s$ are
local homeomorphisms and that $\G$ is Hausdorff, one can check that $\G^{(0)}$ is clopen
in $C^*_r(\G)$, and also that $\G$ has a basis of open sets $U$ such that $r|_U$ is a
homeomorphism of $U$ onto $r(U)$ and $s|_U$ is a homeomorphism of $U$ onto $s(U)$; we
call such sets \emph{bisections}.

Given a locally compact Hausdorff \'etale groupoid $\G$, the space $C_c(\G)$ is a
$^*$-algebra under the operations $f^*(\gamma) = \overline{f(\gamma)}$ and $(f *
g)(\gamma) = \sum_{\alpha \in \G^{r(\gamma)}} f(\alpha)g(\alpha^{-1}\gamma)$. For each $u
\in \G^{(0)}$, there is a $^*$-representation $\rho_u$ of $C_c(\G)$ on $\ell^2(\G_u)$
defined by $\rho_u(f)e_\gamma = \sum_{\alpha \in G_{r(\gamma)}} f(\alpha)
e_{\alpha\gamma}$. We call $\rho_u$ the \emph{regular representation} of
$C_c(\G)$ at $u$. The \emph{reduced $C^*$-algebra} $C^*_r(\G)$ is the completion of
$C_c(\G)$ with respect to the norm $\|f\| = \sup_{u \in \G^{(0)}} \|\rho_u(f)\|$. This
norm agrees with the supremum norm on functions $f$ supported on bisections. Since
$\G^{(0)}$ is clopen, $C_c(\G^{(0)})$ includes in $C_c(\G)$ in the canonical way, and
this extends to an injection $C_0(\G^{(0)}) \hookrightarrow C^*_r(\G)$.  The
representations $\rho_u$ extend to representations $\rho_u : C^*_r(\G) \to
\Bb(\ell^2(\G_u))$. So there is a norm-decreasing map $a \mapsto f_a$ from $C^*_r(\G)$ to
$C_0(\G)$ given by
\begin{equation}\label{eq:Jean's jmap}
    f_a(\gamma) = \big(\rho_{s(\gamma)}(a)e_{s(\gamma)} \mid e_\gamma\big)\quad\text{ for all $a \in C^*_r(\G)$ and $\gamma \in \G$,}
\end{equation}
and $a \mapsto f_a$ restricts to the identity map on $C_c(\G)$.

\subsection{\texorpdfstring{$C_0(X)$}{C(X)}-algebras}

Let $X$ be a locally compact Hausdorff space. A $C_0(X)$-algebra is a $C^*$-algebra $A$
together with a nondegenerate inclusion $\iota : C_0(X) \to ZM(A)$ of $C_0(X)$ into the
centre of the multiplier algebra of $A$. We obtain a family of ideals $I_x :=
\overline{\iota(\{f \in C_0(X) : f(x) = 0\}) A}$ of $A$ (these subsets are automatically
linear subspaces), and then a bundle of $C^*$-algebras $\{A_x : x \in X\}$ over $X$ given
by $A_x := A/I_x$. Each $a \in A$ determines a section $f_a : X \to \mathcal{A} =
\bigsqcup_{x \in X} A_x$ such that $f_a(x) = a + I_x \in A_x$. There is a unique topology
on $\mathcal{A}$ under which these sections are all continuous. With respect to this
topology, $\mathcal{A}$ is an upper-semicontinuous bundle of $C^*$-algebras in the sense
that $b \mapsto \|b\|$ is upper semicontinuous from $\mathcal{A}$ to $[0,\infty)$.

Given any upper-semicontinuous bundle $\mathcal{A}$ of $C^*$-algebras over $X$, the space
$\Gamma_0(X, \mathcal{A})$ of continuous sections of $\mathcal{A}$ that vanish at
infinity is a $C^*$-algebra under pointwise operations and the supremum norm, and it
becomes a $C_0(X)$-algebra with respect to the map $\iota : C_0(X) \to ZM(\Gamma_0(X,
\mathcal{A}))$ given by $(\iota(f)\xi)(x) = f(x)\xi(x)$ for $f \in C_0(X)$ and $\xi \in
\Gamma_0(X, \mathcal{A})$. If $\mathcal{A}$ is the bundle coming from a $C_0(X)$-algebra
$A$ as above, the map $a \mapsto f_a$ is an isomorphism  $A \cong \Gamma_0(X,
\mathcal{A})$.

\subsection{Coactions}

Given a discrete group $\Gamma$, we write $\lambda_g$ for the image of $g \in \Gamma$ in
the left regular representation of $\Gamma$ on $\ell^2(\Gamma)$. We write $\delta_\Gamma
: C^*_r(\Gamma) \to C^*_r(\Gamma) \otimes C^*_r(\Gamma)$ (we use the minimal tensor
product) for the comultiplication such that $\delta_\Gamma(\lambda_g) = \lambda_g \otimes
\lambda_g$ for $g \in \Gamma$. Given a $C^*$-algebra $A$, a \emph{coaction} of $\Gamma$
on $A$ is a nondegenerate homomorphism $\delta : A \to A \otimes C^*_r(\Gamma)$
satisfying the coaction identity $(\delta \otimes 1) \circ \delta = (1 \otimes
\delta_\Gamma) \circ \delta$. The spectral subspaces of $A$ are the spaces $A_g := \{a
\in A : \delta(a) = a \otimes \lambda_g\}$; we write $A^\delta$ for the neutral spectral
subspace $A_{\id_\Gamma}$, and call it the \emph{generalised fixed-point algebra} for $\delta$.
Since we are dealing with reduced coactions, they automatically satisfy $A = \clsp
\big(\bigcup_g A_g\big)$.

For each $g \in \Gamma$ there is a norm-decreasing linear map $\Phi_g : A \to A_g$ that
fixes $A_g$ pointwise and annihilates $A_h$ for $h \not=g$. Specifically, writing
$\operatorname{Tr}$ for the canonical trace on $C^*_r(\Gamma)$, the map $\Phi_g$ is given
by $\Phi_g(a) = (\id_A \otimes \operatorname{Tr})(\delta(a)(1_A \otimes
\lambda_{g^{-1}}))$ for $a \in A$. In particular $\Phi^\delta := \Phi_{\id_\Gamma} : A \to A^\delta$
is a conditional expectation.

We say that the coaction $\delta$ is \emph{trivial} if $\Gamma$ is the trivial group.

\subsection{Morita equivalence}

Throughout the paper, we say that an element or a subset of a $C^*$-algebra $A$ is
\emph{$A$-full} (or just \emph{full}) if it generates $A$ as an ideal. Given
$C^*$-algebras $A$ and $B$, an $A$--$B$-imprimitivity bimodule is an $A$--$B$-bimodule
$X$ carrying a left $A$-linear $A$-valued inner-product ${_A\langle \cdot, \cdot
\rangle}$ and a right $B$-linear $B$-valued inner-product $\langle\cdot, \cdot\rangle_B$
such that ${_A\langle x, y\cdot b \rangle} = {_A\langle x\cdot b^*, y \rangle}$ and
$\langle a\cdot x, y\rangle_b = \langle x, a^* \cdot y\rangle_B$ for all $x,y \in X$, $a
\in A$ and $b \in B$, and such that $x \cdot \langle y, z\rangle_B = {_A\langle x,
y\rangle}\cdot z$ for all $x,y,z \in X$. We say that $C^*$-algebras $A$ and $B$ are
\emph{Morita equivalent} if there exists an $A$--$B$-imprimitivity bimodule. For any
$C^*$-algebra $A$ and any positive element (in particular, any projection) $a \in M(A)$,
the space $Aa$ is an imprimitivity bimodule from the ideal $A a A := \clsp\{b a c : b,c
\in A\}$ generated by $a$ to the hereditary subalgebra $\overline{a A a}$ under
${_{AaA}\langle x, y\rangle} = xy^*$ and $\langle x,y\rangle_{aAa} = x^*y$. So if $a$ is
$A$-full, then $A$ is Morita equivalent to $\overline{aAa}$. If $X$ is an
$A$--$B$-imprimitivity bimodule, then the conjugate module $X^*$ is a
$B$--$A$-imprimitivity bimodule, and $L := A \oplus X \oplus X^* \oplus B$ becomes a
$C^*$-algebra called the \emph{linking algebra}, under the natural operations obtained by
regarding elements $(a, x, y^*, b)$ as $2 \times 2$-matrices $\big(\begin{smallmatrix} a
& x \\ y^* & b\end{smallmatrix}\big)$: products of the form $xy^*$ and $y^*x$ are given
by ${_A\langle x, y\rangle}$ and $\langle y,x\rangle_B$ respectively. The projections $P
= 1_{M(A)}$ and $Q = 1_{M(B)}$ are complementary full multiplier projections such that $P
L P \cong A$, $Q L Q \cong B$, and $PLQ \cong X$. The Brown--Green--Rieffel theorem
implies that if $A$ and $B$ admit countable approximate units, then they are Morita
equivalent if and only if they are stably isomorphic.

\section{Normalisers}\label{sec:weakly cartan pairs}

In this section we
recall the notion of
\emph{normalisers}, and prove some fundamental results that we will need later (mainly in
Section~\ref{sec:sdcs}).

\begin{dfn}[See \cite{Kumjian:CJM86, Renault80, Ren2008}]\label{dfn:normaliser}
Given a $C^*$-algebra $A$ and a $C^*$-subalgebra $D$ of $A$ containing an approximate unit for
$A$, a \emph{normaliser} $n$ of $D$ in $A$ is an element $n \in A$ such that $nDn^* \cup
n^*Dn \subseteq D$. We write $\Nn_A(D)$, or just $\Nn(D)$ if $A$ is clear from context,
for the set of normalisers of $D$ in $A$.
\end{dfn}

\begin{ntn} For the remainder of the section, $A$ is a $C^*$-algebra and $D \subseteq A$ is an abelian $C^*$-subalgebra
containing an approximate unit for $A$. We write $D'_A$ for the
relative commutant $D'_A = \{a \in A : ad = da\text{ for all }d \in D\}$ and
$\widehat{D}$ for the set of characters of $D$.

For $d \in D$, let $\osupp(d) := \{\phi \in
\widehat{D} : \phi(d) \not= 0\}$ and $I(d):=\{d'\in D : \osupp(d')\subseteq\osupp(d)\}$.
So $\osupp(d)$ is an open subset of $\widehat{D}$, and $I(d)$ is an ideal of $D$.
\end{ntn}

We establish some basic properties of normalisers, several of which also appear in
\cite{Kumjian:CJM86}.

\begin{lem}[Kumjian, Renault]\label{lem:alpha-n}
Let $A$ be a separable $C^*$-algebra and $D$ an abelian $C^*$-subalgebra containing an
approximate unit for $A$. For $m,n \in N(D)$,
\begin{enumerate}
\item\label{it:nn*,n*n in D} $n^*n, nn^* \in D$;
\item\label{it:alphan} there is a unique homeomorphism $\alpha_n : \osupp(n^*n) \to
    \osupp(nn^*)$ such that\linebreak $\phi(n^*n) \alpha_n(\phi)(d) = \phi(n^*dn)$
    for $\phi\in \osupp(n^*n)$ and  $d \in D$;
\item\label{it:alphand} there is a unique isomorphism $\alpha_n^\#:I(nn^*)\to
    I(n^*n)$ such that $\phi(\alpha_n^\#(d))=\alpha_n(\phi)(d)$ for $\phi\in \osupp (n^*n)$ and $d\in I(nn^*)$;
\item\label{it:d<->n} if $d\in I(nn^*)$, then $dn=n\alpha_n^\#(d)$;
\item\label{it:alphanm} $mn \in \Nn(D)$, $\osupp((mn)^*(mn)) =
    \alpha_n^{-1}(\osupp(m^*m)\cap\osupp(nn^*))$, and on this domain, $\alpha_m \circ
    \alpha_n = \alpha_{mn}$;
\item\label{it:alphan*} $\alpha_{n^*} = \alpha_n^{-1}$; and
\item\label{it:commute} if $U\subseteq \osupp(n^*n)\cap\osupp (m^*m) \subseteq
    \widehat{D}$ is open and satisfies $\alpha_n \vert_U = \alpha_m \vert_U$, then
    $dn^*m=n^*md$ for all $d\in D$ with $\osupp(d)\subseteq U$, and $\phi(m^*nn^*m) =
    \phi(m^*m)\phi(n^*n)$ for $\phi\in U$.
\end{enumerate}
\end{lem}

\begin{proof}
\eqref{it:nn*,n*n in D} Let $(u_j)_j$ be an approximate unit
for $A$ in $D$. Then each $n^*u_j n, nu_j n^* \in D$ because $n \in \Nn(D)$.
So $n^*n = \lim_j n^* u_j n \in D$ and $nn^* = \lim_j nu_j n^*
\in D$.

\eqref{it:alphan} and \eqref{it:alphand} This is proved in
\cite[Proposition~1.6]{Kumjian:CJM86}. We summarise the points of the proof that we need
for the remaining statements. Let $n = v |n|$ be the polar decomposition of $n$ in
$A^{**}$. Then $vv^*d=dvv^*=d$ for $d\in I(nn^*)$, and $v^*vd=dv^*v=d$ for $d\in
I(n^*n)$. Kumjian shows that $v I(n^*n) v^* \subseteq I(nn^*)$ and $v^* I(nn^*) v
\subseteq I(n^*n)$. So $d\mapsto v^* dv$ defines an isomorphism $\alpha_n^\#:I(nn^*)\to
I(n^*n)$, and there is a homeomorphism $\alpha_n : \osupp(n^*n) \to \osupp(nn^*)$ such
that $\alpha_n(\phi)(d) = \phi(\alpha_n^\#(d))$ for all $d \in I(nn^*)$. If $d\in D$,
then
\begin{equation*}
\phi(n^*n)\alpha_n(\phi)(d)=\phi(|n|)\phi(\alpha_n^\#(d))\phi(|n|)=
\phi(|n|v^*dv|n|)=\phi(n^*dn).
\end{equation*}
That each of $\widehat{D}$ and $D$ separates elements of the other gives uniqueness of
$\alpha_n$ and $\alpha_n^\#$.

\eqref{it:d<->n} As above, let $n = v|n|$ be the polar decomposition of $n$.
By~(\ref{it:nn*,n*n in D}), we have $|n|=(n^*n)^{1/2}\in D$. It follows that if $d\in
I(nn^*)$, then
\[
    dn = d v|n| = vv^* d v |n| = v \alpha_n^\#(d) |n| = v|n| \alpha_n^\#(d) = n \alpha_n^\#(d).
\]

\eqref{it:alphanm} Fix $d \in D$ and $m,n \in N(D)$. Then $(mn)^*d(mn) = n^*(m^*dm)n^*
\in n^*Dn \subseteq D$; likewise $(mn)d(mn)^* \in D$, so $mn\in N(D)$. We have
$n^*m^*mnn^*n=n^*n\alpha_n^\#(m^*mnn^*)$ by~\eqref{it:d<->n}. So $\osupp((mn)^*(mn)) =
\alpha_n^{-1}(\osupp(m^*m)\cap\osupp(nn^*))$. For $\phi\in \osupp((mn)^*(mn))$,
\begin{align*}
\phi(n^*m^*mn)\alpha_{mn}(\phi)(d) &=
\phi(n^*m^*dmn)=
\phi(n^*n)\alpha_n(\phi)(m^*dm)\\ &=
\phi(n^*n)\alpha_n(m^*m)\alpha_m(\alpha_n(\phi))(d)=
\phi(n^*m^*mn)\alpha_m(\alpha_n(\phi))(d).
\end{align*}
This shows that $\alpha_{mn}(\phi)=\alpha_m(\alpha_n(\phi))$.

\eqref{it:alphan*} Suppose $\phi\in\osupp ((n^*n)^*(n^*n))$ and $d\in D$. Then
\begin{equation*}
\phi((n^*n)^*(n^*n))\alpha_{n^*n}(\phi)(d)=
\phi((n^*n)^*d(n^*n))=
\phi(n^*n)\phi(d)\phi(n^*n)=
\phi((n^*n)^*(n^*n))\phi(d).
\end{equation*}
So $\alpha_{n^*n}=\id_{\osupp(n^*n)}$, and then~\eqref{it:alphanm} gives $\alpha_{n^*}
\circ \alpha_n = \id_{\osupp(n^*n)}$. Thus, $\alpha_{n^*} = \alpha_n^{-1}$.

\eqref{it:commute} Fix $d\in D$ with $\osupp(d)\subseteq U$. As $\alpha_n \vert_U =
\alpha_m \vert_U$, statement \eqref{it:alphand}~and~\eqref{it:alphan*} give
$\alpha_{n^*}^\#(d)=\alpha_{m^*}^\#(d)$. Two applications of \eqref{it:d<->n} then give
$dn^*m=n^*\alpha_{n^*}(d)m=n^*md$. The definition of $\alpha\#_n$ shows that
$\alpha^\#_n(nn^*) = n^*n$, so $\alpha_n(\phi)(nn^*) = \phi(n^*n)$ for all $\phi$. So for
$\phi \in U$, we have $\phi(m^*nn^*m) = \phi(m^*m)\alpha_m(\phi)(nn^*) =
\phi(m^*m)\alpha_n(\phi) = \phi(m^*m)\phi(n^*n)$.
\end{proof}

\section{Results for ungraded groupoids and \texorpdfstring{$C^*$}{C*}-algebras}\label{sec:summary ungraded}

To maximise their generality, we formulate our key results later in the paper for graded
groupoids and $C^*$-algebras carrying coactions, under suitable hypotheses on the
trivially-graded subgroupoids and generalised fixed-point subalgebras. In this section,
we summarise the consequences of our main results for trivial gradings and coactions.

Lemma~\ref{lem:equivalence-construction} below applied to a separable $C^*$-algebra $A$, an
abelian $C^*$-subalgebra $D\subseteq A$ containing an approximate unit for $A$, and the trivial
coaction on $A$ yields an equivalence relation $\sim$ on $\{(n,\phi) : n \in N_A(D),\
\phi\in\osupp(n^*n)\}$ as follows: given $\phi \in \widehat{D}$, we write $J_\phi$ for the ideal
of $D_A'$ generated by $\ker(\phi)$; and then $(n,\phi) \sim (m,\psi)$ if and only if $\phi =
\psi$, $\alpha_n$ and $\alpha_m$ agree on a neighbourhood $U$ of $\phi$, and there exists $d \in
D$ with $\osupp(d) \subseteq U$ and $\phi(d) = 1$ such that
$\phi(m^*m)^{-\frac{1}{2}}\phi(n^*n)^{-\frac{1}{2}} d n^*m d + J_\phi$ is homotopic to the
identity in the unitary group of $D_A'/J_\phi$. We write $[n,\phi]$ for the equivalence class of
$(n,\phi)$ with respect to this equivalence relation.

Our first main theorem says that the quotient space can be made into an \'etale groupoid.

\begin{thm}[see Theorem~\ref{thm:lche gpd}]\label{prp:no-coaction H}
Let $A$ be a separable $C^*$-algebra and $D$ an abelian $C^*$-subalgebra of $A$ that contains an approximate unit for $A$. Then
$$\Hh(A, D) := \{[n,\phi] : n \in \Nn_A(D),\ \phi\in\osupp(n^*n)\}$$
is a second-countable locally compact locally Hausdorff \'etale groupoid with composable
pairs $\Hh(A, D)^{(2)} = \{([m,\phi], [n,\psi]) : \phi =\alpha_n(\psi)\}$, multiplication
and inverses given by
\[
[m,\alpha_n(\psi)][n,\psi] = [mn,\psi]
    \quad\text{ and }\quad
[n,\psi]^{-1}=[n^*,\alpha_n(\psi)],
\]
and topology with basic open sets
\[
Z(n, U) = \{[n,\phi] : \phi \in U\}
\]
indexed by $n \in \Nn(D)$ and open $U \subseteq \osupp(n^*n) \subseteq \widehat{D}$.
\end{thm}

If $D \subseteq A$ is Cartan as in \cite{Ren2008}, then $\Hh(A,D)$ is the Weyl groupoid
of \cite{Ren2008}. If $A=C^*(E)$ and $D=\mathcal{D}(E)$ where $E$ is a countable graph,
then $\Hh(A, D)$ is the extended Weyl groupoid $\mathcal{G}_{(C^*(E),\mathcal{D}(E))}$ of
\cite{BCW}.
If $A=C^*(\Gamma)$ where $\Gamma$ is a torsion free and abelian discrete group, and $D=\mathbb{C} 1_A$, then $N_A(D)$ consists of the nonzero scalar multiples of unitaries in $C^*(\Gamma)$, and it then follows from \cite[Theorem~8.57]{HofmannMorris} that $\Hh(A, D)$ is isomorphic to $\Gamma$ by an isomorphism that takes $[n,\phi]$ to the class of the unitary $(n^*n)^{-\frac{1}{2}}n$ in $K_1(C^*(\Gamma))\cong \Gamma$. The following generalises this.

\begin{prp}[see Lemma~\ref{lem:gpd sds} and Proposition~\ref{prp:isomorphism}]\label{prp:isomorphism-noaction}
Let $\G$ be a second-countable locally compact Hausdorff \'etale groupoid. Then $C_0(\go)$ is an abelian $C^*$-subalgebra of $C^*_r(\G)$ and contains an approximate unit for $C^*_r(\G)$. If
$\Io{\G}$ is torsion free and abelian, then there is an isomorphism $\G \cong \Hh(C^*_r(\G),
C_0(\go))$ that carries $\gamma \in \G$ to $[f, \widehat{s(\gamma)}]$ for any $f \in
C_c(\G)$ supported on a bisection with $f(\gamma) \not= 0$.
\end{prp}

As an almost immediate consequence of Proposition~\ref{prp:isomorphism-noaction}, we
obtain the following.

\begin{thm}[see Theorem~\ref{thm:1}]\label{thm:no-coaction 1}
Let $\G_1, \G_2$ be second-countable locally compact Hausdorff \'etale groupoids such
that each $\Io{\G_i}$ is torsion-free and abelian.
\begin{enumerate}
\item Any isomorphism $\kappa : \G_2 \to \G_1$ induces an isomorphism $\phi :
    C^*_r(\G_1)\to C^*_r(\G_2)$ such that $\phi(f) = f \circ \kappa$ for $f \in
    C_c(\G_1)$, and in particular $\phi(C_0(\G_1^{(0)})) = C_0(\G_2^{(0)})$.
\item Any isomorphism $\phi:C^*_r(\G_1)\to C^*_r(\G_2)$ satisfying
    $\phi(C_0(\G_1^{(0)}))=C_0(\G_2^{(0)})$ induces an isomorphism $\kappa : \G_2 \to
    \G_1$ such that $f \circ \kappa = \phi(f)$ for $f \in C_0(\G_1^{(0)})$.
\end{enumerate}
\end{thm}

If $D_i \subseteq A_i$ is a nested pair of $C^*$-algebras for $i = 1,2$, we say that
$(A_1, D_1)$ and $(A_2, D_2)$ are Morita equivalent if there is an
$A_1$--$A_2$-imprimitivity bimodule $X$ such that
\[
X = \clsp\{x \in X : \langle x, D_1 \cdot x\rangle_{A_2} \subseteq D_2\text{ and } {_{A_1}\langle x\cdot D_2, x\rangle} \subseteq D_1\}.
\]

\begin{thm}[see Theorem~\ref{thm:eqiv systems->eqiv groupoids}]\label{thm:no-coaction eqiv systems->eqiv groupoids}
Let $A_1, A_2$ be separable $C^*$-algebras. Suppose, for $i = 1,2$, that $D_i$ is an abelian $C^*$-subalgebra of $A_i$ containing an approximate unit for $A_i$. Suppose that $X$ is a Morita equivalence between
$(A_1, D_1)$ and $(A_2, D_2)$. Let $A$ be the linking algebra of $X$ and let $D := D_1
\oplus D_2 \subseteq A$. Then $D$ is a $C^*$-algebra of $A$ that contains an approximate unit for $A$. The groupoid $\Hh := \Hh(A,
D)$ contains $\widehat{D}_1$ and $\widehat{D}_2$ as complementary full clopen subsets of
$\ho$, $\widehat{D}_i\Hh\widehat{D}_i \cong \Hh(A_i, D_i)$ for $i = 1,2$, and
$\widehat{D}_1 \Hh \widehat{D}_2$ is an equivalence from $\Hh(A_1, D_1)$ to $\Hh(A_2,
D_2)$.
\end{thm}

\begin{thm}[see Theorem~\ref{thm:2}]\label{thm:no-coaction 2}
Let $\G_1, \G_2$ be second-countable locally compact Hausdorff \'etale groupoids such
that each $\Io{\G_i}$ is torsion-free and abelian. The following are equivalent:
\begin{enumerate}
\item $\G_1$ and $G_2$ are equivalent;
\item there exist a second-countable locally compact Hausdorff \'etale groupoid $G$
    such that $\Io{\G}$ is torsion-free and abelian, and a pair of complementary
    $\G$-full open subsets $K_1, K_2\subseteq \go$ such that $K_i \G K_i \cong \G_i$
    for $i = 1,2$.
\item $(C^*_r(\G_1), C_0(\G_1^{(0)}))$ and $(C^*_r(\G_2), C_0(\G_2^{(0)}))$ are
    Morita equivalent; and
\item there are a separable $C^*$-algebra $A$, an abelian $C^*$-subalgebra $D \subseteq A$ containg an approximate unit for $A$ and
    satisfying $\clsp \Nn(D) = A$, a pair of complementary $A$-full projections $P_1,
    P_2\in M(D)$, and isomorphisms $\phi_i : P_i A P_i \to C^*_r(\G_i)$ such that
    $\phi_i(P_i D P_i) = D_i$.
\end{enumerate}
\end{thm}

As in \cite{CRS}, we let $\R = \NN \times \NN$ regarded as a discrete groupoid; we have
$C^*_r(\R) \cong \Kk$, the compact operators on $\ell^2(\NN)$, with canonical diagonal
subalgebra $\mathcal{C}$. As in \cite{CRS}, we say $\G_1$ and $\G_2$ are \emph{weakly Kakutani
equivalent} if there are full open subsets $U_i \subseteq \G_i^{(0)}$ such that $U_1 \G_1
U_1 \cong U_2\G_2 U_2$. If $U_1$ and $U_2$ are \emph{compact} open, $\G_1$ and $\G_2$ are
\emph{Kakutani equivalent} as defined in \cite{Matui}. Combining our results with
\cite{CRS}, we obtain the following.

\begin{cor}[see Corollary~\ref{cor:strong equiv ample}]\label{cor:no-coaction strong equiv ample}
Let $\G_1, \G_2$ be second-countable ample Hausdorff groupoids such that each $\Io{\G}$
is torsion-free and abelian. Then the equivalent conditions of
Theorem~\ref{thm:no-coaction 2} are also equivalent to the following:
\begin{enumerate}
\item $\G_1\times\R \cong \G_2\times\R$;
\item $\G_1$ and $\G_2$ are Kakutani equivalent;
\item $\G_1$ and $\G_2$ are weakly Kakutani equivalent;
\item there is an isomorphism $\phi:C^*_r(\G_1) \otimes \Kk \to
    C^*_r(\G_2)\otimes\Kk$ satisfying $\phi(C_0(\G_1^{(0)}) \otimes \mathcal{C}) =
    C_0(\G_2^{(0)})\otimes \mathcal{C}$;
\item there exist $C^*_r(\G_i)$-full projections $p_i\in M(C_0(\G_1^{(0)}))$ and an
    isomorphism $\phi$ of $p_1C^*_r(\G_1)p_1$ onto $p_2C^*_r(\G_2)p_2$ such that
    $\phi(p_1C_0(\G_1^{(0)})) = p_2C_0(\G_1^{(0)})$; and
\item there are ideals $I_i \subseteq C_0(\G_i^{(0)})$ that are full in
    $C^*_r(\G_i)$, and an isomorphism $\phi : I_1C^*_r(\G_1)I_1 \to
    I_2C^*_r(\G_2)I_2$ such that $\phi(I_1)=I_2$.
\end{enumerate}
\end{cor}

\section{The extended Weyl groupoid}\label{sec:sdcs}

We construct a graded groupoid $(\Hh(A, D, \delta),c_\delta)$ from a separable
$C^*$-algebra $A$, a coaction of a discrete group $\Gamma$ on $A$, and an abelian
$C^*$-subalgebra $D \subseteq A^\delta$ contaning an approximate unit of $A^\delta$. The reader may wish to keep in mind the case where
$\delta$ is trivial, so $A^\delta=A$; in this case, if $D$ is a Cartan subalgebra of $A$,
then $\Hh(A, D, \delta)$ agrees with Renault's Weyl groupoid \cite{Ren2008}. Given a
countable directed graph, then $\Hh(C^*(E), C_0(\partial E),\delta)$ is the extended Weyl
groupoid $\mathcal{G}_{(C^*(E),\mathcal{D}(E))}$ of \cite{BCW}, both when $\delta$ is the
trivial coaction, and when $\delta$ is the coaction of $\ZZ$ dual to the gauge action.

Given a $C^*$-algebra $A$ and an abelian subalgebra $D$ of $A$, we write $D'_A$ for the
relative commutant $D'_A = \{a \in A : ad = da\text{ for all }d \in D\}$ and
$\widehat{D}$ for the set of characters of $D$.

\begin{lem}\label{lem:semidiag}
Let $A$ be a separable $C^*$-algebra and $D$ an abelian $C^*$-subalgebra of $A$ containing an approximate unit for $A$. Then
\begin{enumerate}
\item\label{it:sd unital fibres} for each $\phi \in \widehat{D}$, the quotient
    $D_A'/J_\phi$ by the ideal $J_\phi := \overline{\ker(\phi)D_A'}$ of $D_A'$ is unital,
\item\label{it:sd unital sections} for each $\phi \in \widehat{D}$, there exist $d \in D$ and an open neighbourhood $U$ of $\phi$ such that $d + J_\psi =
    1_{D_A'/J_\psi}$ for all $\psi \in U$, and
\item  for each $\phi \in \widehat{D}$ and for each $d \in D$, we have $d +
    J_\phi = \phi ( d ) 1_{ D_A'/J_\phi }$.
\end{enumerate}
\end{lem}

\begin{proof}
We first prove~(2), and then~(1) follows directly. So fix $\phi \in \widehat{D}$. Since $D$ is abelian, we have $D \cong C_0(\widehat{D})$, so there exists an open neighbourhood $U$ of $\phi$ and an element $d \in D$ such that $\psi(d) = 1$ for all $\psi \in U$. We claim that $d + J_\psi$ is a unit for $D_A'/J_\psi$ for each $\psi \in U$. To establish the claim, fix $\psi \in U$. It suffices to show that $d + J_\psi$ is a unit for $D_A'/J_\psi$. For this, fix $a \in D_A'$ so that $a + J_\psi$ is a typical element of $D_A'/J_\psi$. Let $(u_j)_j$ be an approximate unit for $A$ such that each $u_j\in D$. Since $\psi(d) = 1$, we have $\psi(du_j - u_j) = 0$ for all $j$, giving $du_j - u_j \in J_\psi$ for all $j$. Hence $du_j a - u_j a \in J_\psi$ for all $j$. Since $u_j a \to a$, we deduce that $da - a \in J_\psi$. Hence $(d + J_\psi)(a + J_\psi) = a + J_\psi$. Thus $d + J_\psi$ is a unit for $D_A'/J_\psi$ as required.

To see (3), fix $\phi \in \widehat{D}$ and $d
\in D$. Choose $d_0 \in D$ and an open $U \owns \phi$ such that $d_0 + J_\psi =
1_{D_A'/J_\psi}$ for all $\psi$ in $U$.  Since $( d_0^2 + J_\phi )^2 = (1_{D_A'/J_\phi})^2 =
1_{D_A'/J_\phi} = d_0 + J_\phi$, we have $\phi ( d_0^2 ) = \phi ( d_0 )$. Thus, $\phi ( d_0
) = 1$ and $\phi ( \phi ( d ) d_0 ) = \phi (d)$. Consequently, $d - \phi (d) d_0 \in
J_\phi$, giving
\[
d + J_\phi = \phi (d) d_0 + J_\phi = \phi (d) ( d_0 + J_\phi ) = \phi (d) 1_{D_A'/J_\phi}.\qedhere
\]
\end{proof}

Throughout the remaining part of this section, $A$ is a separable $C^*$-algebra, $\Gamma$ is a discrete group,
$\delta$ is a coaction of $\Gamma$ on $A$, and $D$ is a an abelian $C^*$-subalgebra of
$A^\delta$ that contains an approximate unit for $A^\delta$. We write
\[
D'_{A^\delta} := \{a \in A^\delta : ad = da\text{ for all } d \in D\},
\]
for the relative commutant of $D$ in $A^\delta$, $J_\phi$ for the ideal $\overline{\ker(\phi)D'_{A^\delta}}$, and
\[
\pi_\phi : D'_{A^\delta} \to D'_{A^\delta}/J_\phi
\]
for the canonical quotient map.


Following \cite{ABHS}, we say that a normaliser $n$ of $D$ is a \emph{homogeneous
normaliser} if $n \in A_g$ for some $g \in \Gamma$. We write $\Nn_g(D) := \Nn(D) \cap
A_g$, and we write $\Nn_\star(D) := \bigcup_{g \in \Gamma} \Nn_g(D)$. The groupoid
$\Hh(A, D, \delta)$ consists of equivalence classes $[n,\phi]$ where $n\in\Nn_\star(D)$
and $\phi \in \osupp(n^*n)$. To define the appropriate equivalence relation, we first
need two lemmas.

\begin{lem}\label{lem:unitary-normalizers}
Let $A$ be a separable $C^*$-algebra, $\delta$ a coaction of a discrete group $\Gamma$ on
$A$, and $D$ an abelian $C^*$-subalgebra of $A^\delta$ containing an approximate unit for $A^\delta$. Let $n, m \in N(D)$ and $\phi \in
\osupp(n^*n) \cap \osupp(m^*m)$, and suppose that there is an open neighbourhood $U$ of
$\phi$ such that $U \subseteq \osupp(n^*n) \cap \osupp(m^*m)$ and $\alpha_n \vert_U =
\alpha_m \vert_U$. Fix $d \in D$ with $\osupp(d) \subseteq U$ and $\phi(d) = 1$, and let
\begin{equation}\label{eq:wdef}
w := \phi(m^*m)^{-1/2}\phi(n^*n)^{-1/2} d n^*m d.
\end{equation}
Then $w\in D'_{A^\delta}$ and $\pi_\phi(w)$ is unitary in $D'_{A^\delta}/J_\phi$. We have

\[
\pi_\phi\big(\phi(n^*n)^{-1/2}\phi(m^*m)^{-1/2} d m^*n d\big) = \pi_\phi(w^*),
\]
and $\pi_\phi(w)$ is
independent of the choices of $U$ and $d$.
\end{lem}
\begin{proof}
We have $w\in D'_{A^\delta}$ by Lemma~\ref{lem:alpha-n}\eqref{it:commute}. By
Lemma~\ref{lem:semidiag} and Lemma~\ref{lem:alpha-n}, and since $\osupp(d)\subseteq U$,
$\phi(d) = 1$, and $\alpha_m|_U = \alpha_n|_U$, we have
\begin{align*}
\phi(m^*m)\phi(n^*n) \pi_\phi(w^*w)
&=\pi_\phi ((d n^*m d)^*(d n^*m d))
=\pi_\phi(d^*d)^2  \pi_\phi(m^*nn^*m)\\
&= \phi(m^*nn^*m) 1_{D'_{A^\delta}/J_\phi}
=\phi(m^*m)\phi(n^*n)1_{D'_{A^\delta}/J_\phi},
\end{align*}
so $\pi_\phi(w^*w) = 1_{D'_{A^\delta}/J_\phi}$. Switching the roles of $m,n$ gives
$\pi(ww^*) = 1_{D'_{A^\delta}/J_\phi}$. We clearly have
$\pi_\phi(\phi(m^*m)^{-1/2}\phi(n^*n)^{-1/2} d m^*n d) = \pi_\phi(w^*)$.

Now fix open $U_1, U_2 \subseteq \widehat{D}$ and $d_1,d_2 \in D$ with $\phi \in U_1 \cap
U_2$, $\alpha_n|_{U_i} = \alpha_m|_{U_i}$, $\osupp(d_i) \subseteq U_i$ and $\phi(d_i) =
1$. It follows from Lemma~\ref{lem:alpha-n}\eqref{it:commute} that $n^*m d_1d_2 = d_1d_2
n^*m$. Thus
\[
\pi_\phi(d_1 n^*m d_1 d_2m^*n d_2 )
    = \pi_\phi (d_1^2d_2 n^*m m^* n d_2 )
    = \pi_\phi (n^*m m^* n)
    = \phi(n^*m m^*n)1_{D'_{A^\delta}/J_\phi},
\]
and we deduce that
\[
\pi_\phi\big(\phi(m^*m)^{-1/2}\phi(n^*n)^{-1/2} d_1 n^*m d_1\big) =
\pi_\phi\big(\phi(m^*m)^{-1/2}\phi(n^*n)^{-1/2} d_2 n^*m d_2\big).\qedhere
\]
\end{proof}

\begin{ntn}\label{ntn:Un*m}
Let $A$ be a separable $C^*$-algebra, $\delta$ a coaction of a discrete group $\Gamma$ on
$A$, and $D$ an abelian $C^*$-subalgebra of $A^\delta$ containing an approximate unit for $A^\delta$. Suppose that $n,m \in N(D)$,
$\phi\in\osupp (n^*n)\cap\osupp (m^*m)$, and that there is an open $U \owns \phi$ such
that $\alpha_m|_U = \alpha_n|_U$. We write
\[
U^\phi_{n^*m} := \pi_\phi(w)
\]
for any $w$ of the form~\eqref{eq:wdef}. If $\phi$ is clear from context, we just write
$U_{n^*m}$ for $U^\phi_{n^*m}$
\end{ntn}

\begin{lem}\label{lem:unitary}
Let $A$ be a separable $C^*$-algebra, $\delta$ a coaction of a discrete group $\Gamma$ on
$A$, and $D$ an abelian $C^*$-subalgebra of $A^\delta$ containing an approximate unit for $A^\delta$. Take $n_1, n_2,m \in N(D)$,
$\phi\in\osupp (n_1^*n_1)\cap \osupp(n_2^*n_2) \cap\osupp (m^*m)$, and open $U_i \owns
\phi$ with $\alpha_m|_{U_i} = \alpha_{n_i}|_{U_i}$. Then
\begin{enumerate}
	\item\label{it:u1} $U_{n_i^*n_i}=1_{D'_{A^\delta}/J_\phi}$,
	\item\label{it:u2} $U_{n_i^*m}^*=U_{mn_i^*}$, and
	\item\label{it:u3} $U_{n_1^*m}U_{m^*n_2}=U_{n_1^*n_2}$.
\end{enumerate}
\end{lem}
\begin{proof}
By normalising, we can assume that $\phi(n_i^*n_i) = \phi(m^*m) = 1$. For~(1), just
calculate: $\pi_\phi(dn_i^*n_id) = \phi(d)^2\phi(n_i^*n_i) \cdot 1_{D'_{A^\delta}/J_\phi}
= 1_{D'_{A^\delta}/J_\phi}$. Statement~(2) follows from
Lemma~\ref{lem:unitary-normalizers} because $\pi_\phi$ is a *-homomorphism. For~(3), take
$d_i$ supported on $U_i$ with $\phi(d_i) = 1$. A quick calculation using that
$\alpha_{n_i}(\phi)(mm^*) = \phi(m^*m)$ and that $\alpha_{n_i^*m}(\phi) = \phi$ gives
$U_{n_1^*m}U_{m^*n_2} = \phi(m^*m)\phi(d_1d_2)\pi_\phi(d_1 n_1^*n_2d_2) = U_{n_1^*n_2}$.
\end{proof}

We are now ready to describe the elements of $\Hh(A,D,\delta)$.

\begin{lem}\label{lem:equivalence-construction}
Let $A$ be a separable $C^*$-algebra, $\delta$ a coaction of a discrete group $\Gamma$ on
$A$, and $D$ an abelian $C^*$-subalgebra of $A^\delta$ containing an approximate unit for $A^\delta$. Define $\sim$ on $\{(n,\phi) : n
\in N_\star(D), \phi\in\osupp(n^*n)\}$ by $(n, \phi) \sim (m,\psi)$ if and only if
\begin{itemize}
\item[(R1)] $\phi = \psi$,
\item[(R2)] $n^*m \in A^\delta$,
\item[(R3)] there exists an open neighbourhood $U$ of $\phi$ in $\widehat{D}$ such
    that $U\subseteq \osupp(n^*n)\cap\osupp (m^*m)$ and $\alpha_m \vert_U = \alpha_n
    \vert_U$, and
\item[(R4)] the unitary $U^\phi_{n^*m}$ of Notation~\ref{ntn:Un*m} belongs to the
    connected component $\mathcal{U}_0(D'_{A^\delta}/J_\phi)$ of the identity in the
    unitary group of $D'_{A^\delta}/J_\phi$.
\end{itemize}
Then $\sim$ is an equivalence relation.
\end{lem}
\begin{proof}
Reflexivity follows from Lemma~\ref{lem:unitary}\eqref{it:u1}. Symmetry follows from
Lemma~\ref{lem:unitary}\eqref{it:u2} and that if $U^\phi_{n^*m}\in
\mathcal{U}_0(D'_{A^\delta}/J_\phi)$, then $(U^\phi_{n^*m})^*\in
\mathcal{U}_0(D'_{A^\delta}/J_\phi)$.

For transitivity, suppose $(n_1, \phi) \sim (m, \psi) \sim (n_2, \omega)$. Then $\phi = \psi =
\omega$, $n_1^*m, m^*n_2 \in A^\delta$, there are open sets $\phi \in U_i \subseteq \osupp
(n_i^*n_i) \cap \osupp (m^*m)$ such that $\alpha_{n_i} \vert_{U_i} = \alpha_m \vert_{U_i}$, and
$U_{n_1^*m}, U_{m^*n_2} \in \mathcal{U}(D'_{A^\delta}/J_\phi)_0$.

The set $V := U_1\cap U_2\subseteq \osupp (n_1^*n_1) \cap \osupp (n_2^* n_2)$ is open,
contains $\phi$ and satisfies $\alpha_{n_1} \vert_V = \alpha_{n_2} \vert_V$. We have $n_1
\in A_{g_1}$, $m \in A_{h}$ and $n_2 \in A_{g_2}$ for some $g_1, h, g_2 \in \Gamma$.
Since $\alpha_{n_1}(\phi) = \alpha_m(\phi)$, Lemma~\ref{lem:alpha-n} gives
$\phi(n_1^*mm^*n_1) = \phi(n_1^*n_1)\alpha_m(\phi)(mm^*) = \phi(n_1^*n_1)\phi(m^*m)$, so
$n_1^*m \not= 0$. So $n_1^*m \in A^\delta$ forces $g_1 = h$. Likewise $h = g_2$, so
$n_1^*n_2 \in A^\delta$.

Finally, $U_{n_1^*n_2}=U_{n_1^*m}U_{m^*n_2}$ by Lemma~\ref{lem:unitary}\eqref{it:u3}, so
$U_{n_1^*m},U_{m^*n_2}\in \mathcal{U}_0(D'_{A^\delta}/J_\phi)$ forces $U_{n_1^*n_2}\in
\mathcal{U}_0(D'_{A^\delta}/J_\phi)$, and thus $(n_1, \phi) \sim (n_2, \omega)$.
\end{proof}

\begin{rmk}
In the situation of Lemma~\ref{lem:equivalence-construction}, if $n \in N_\star(D)$ and
$\phi \in \osupp(n^*n)$, then $m := \phi(n^*n)^{-1/2} n$ satisfies $(n,\phi) \sim
(m,\phi)$ and $\phi(m^*m) = 1$. So we can and frequently will, without loss of
generality, choose representatives $(n,\phi)$ of the equivalence classes for $\sim$ such
that $\phi(n^*n) = 1$. If $\phi(m^*m) = \phi(n^*n) = 1$ and $\alpha_m$ and $\alpha_n$
agree on a neighbourhood $U \owns \phi$, then $U_{n^*m} = \pi_\phi(dn^*md)$ for any $d$
supported on $U$ with $\phi(d) = 1$.
\end{rmk}

We now construct our extended Weyl groupoid.

\begin{prp}\label{prp:groupoid-construction}
Let $A$ be a separable $C^*$-algebra, $\delta$ a coaction of a discrete group $\Gamma$ on
$A$, and $D$ an abelian $C^*$-subalgebra of $A^\delta$ containing an approximate unit for $A^\delta$. Let $\sim$ be the equivalence
relation of Lemma~\ref{lem:equivalence-construction}, and for $n \in \Nn_\star(D)$ and
$\phi \in \osupp(n^*n)$, let $[n,\phi]$ denote the equivalence class of $(n,\phi)$ under
$\sim$. Define
\[
\Hh(A,D,\delta) := \{[n, \phi] : n \in N_\star(D), \phi \in \osupp(n^*n)\}.
\]
There are maps
\begin{gather*}
    r, s \colon \Hh(A, D, \delta) \to \widehat{D},\\
       M \colon \Hh(A, D, \delta) \mathbin{_s\times_r} \Hh(A, D, \delta) \to \Hh(A, D, \delta),\text{ and}\\
       I \colon \Hh(A, D, \delta) \to \Hh(A, D, \delta)
\end{gather*}
such that
\begin{gather*}
    r([n, \phi] ) = \alpha_n(\phi), \quad s([n,\phi]) = \phi, \\
    \quad M([n, \phi],[m,\psi]) = [nm, \psi], \quad \text{and} \quad I([n,\phi]) = [n^*, \alpha_n(\phi)].
\end{gather*}
Moreover, $\Hh(A, D, \delta)$ is a groupoid under these operations, and there is a
cocycle $c_\delta : \Hh(A, D, \delta) \to \Gamma$ such that $c_\delta([n,\phi]) = g$ if
and only if $n \in A_g$.
\end{prp}

We need the following lemma for the proof of Proposition~\ref{prp:groupoid-construction}.

\begin{lem}\label{lem:trans}
Let $A$ be a separable $C^*$-algebra, $\delta$ a coaction of a discrete group $\Gamma$ on
$A$, and $D$ an abelian $C^*$-subalgebra of $A^\delta$ containing an approximate unit for $A^\delta$. Suppose $m\in N(D)$,
$\phi\in\osupp(m^*m)$, and $\phi(m^*m)=1$. Then there is an isomorphism
$\iota_m:D'_{A^\delta}/J_{\alpha_m(\phi)}\to D'_{A^\delta}/J_\phi$ such that
$\iota_m(\pi_{\alpha_m(\phi)}(a))=\pi_\phi(m^*am)$ for $a\in D'_{A^\delta}$.
\end{lem}
\begin{proof}
By Lemma~\ref{lem:alpha-n}, $a\in J_{\alpha_m(\phi)}\implies m^*am\in J_\phi$. Thus
$a\mapsto m^*am$ descends to a linear $^*$-preserving map
$\iota_m:D'_{A^\delta}/J_{\alpha_m(\phi)}\to D'_{A^\delta}/J_\phi$. For $a_1,a_2\in
D'_{A^\delta}$,
\begin{align*}
\pi_\phi(m^*a_1m)\pi_\phi(m^*a_2m)
&=\pi_\phi(m^*a_1mm^*a_2m)=\pi_\phi(m^*mm^*a_1a_2m)\\
&=\phi(m^*m)\pi_\phi(m^*a_1a_2m)=\pi_\phi(m^*a_1a_2m)
\end{align*}
so $\iota_m$ is multiplicative and hence a *-homomorphism. Symmetry gives a
*-homomorphism $\iota_{m^*}:D'_{A^\delta}/J_\phi\to D'_{A^\delta}/J_{\alpha_m(\phi)}$
such that $\iota_{m^*}(\pi_\phi(a))=\pi_{\alpha_m(\phi)}(mam^*)$ for $a\in
D'_{A^\delta}$. It is easy to check that $\iota_{m^*}$ is an inverse to $\iota_m$.
\end{proof}

\begin{proof}[Proof of Proposition~\ref{prp:groupoid-construction}]
If $(n,\phi) \sim (m,\psi)$, then $\phi = \psi$, and since $\alpha_n = \alpha_m$ on a
neighbourhood of $\phi$, we have $\alpha_n(\phi) = \alpha_m(\phi)$; so $r$ and $s$ are
well defined. Suppose that $[n, \phi] = [n', \phi]$ and $r([m, \psi] ) = s([n, \phi])$.
Then $nm,n'm\in N_\star(D)$. We claim that $(nm , \psi)\sim (n' m, \psi)$. Indeed, (R1)
is clear, and~(R2) is immediate because $(n, \phi)\sim (n', \phi)$ forces
$(nm)^*n'm=m^*n^*n'm\in A^\delta$. Take an open $U$ with $\phi \in U \subseteq
\osupp(n^*n) \cap ((n')^*n')$ and $\alpha_n|_U = \alpha_{n'}|_U$. Fix $d$ supported on
$U$ with $\phi(d) = 1$. Since $(n,\phi) \sim (n',\phi)$, we have $U^\phi_{n^*n'} \in
\mathcal{U}_0(D'_{A^{\delta}}/J_\phi)$. Now $V:=\alpha_m^{-1}(U\cap \osupp(mm^*))$ is a
neighbourhood of $\psi$, and Lemma~\ref{lem:alpha-n} gives $\alpha_{nm}|_V =
\alpha_{n'm}|_V$, giving~(R3). For~(R4), we may assume that $\psi(m^*m) = 1$. The
isomorphism $\iota_m$ of Lemma~\ref{lem:trans} satisfies
$U^\psi_{(nm)^*(n'm)}=\iota_m(U^\phi_{n^*n'})$, so $U^\phi_{n^*n'} \in
\mathcal{U}_0(D'_{A^{\delta}}/J_\phi)$ forces
$U^\psi_{(nm)^*(n'm)}\in\mathcal{U}_0(D'_{A^\delta}/J_\psi)$, giving~(R4). So $(nm ,
\psi)\sim (n' m, \psi)$. Therefore
\begin{equation}\label{eq:1stfactor}
    \text{if }[n,\phi] = [n',\phi]\text{ and } s([n, \phi] ) = r([m, \psi]),
     \text{ then } [nm,\psi] = [n'm,\psi].
\end{equation}

Now suppose that $[m,\psi] = [m',\psi]$ and $r([m,\psi])=s([n,\phi])$. We have $nm,nm'\in
N_\star(D)$. We claim that $(nm,\psi)\sim (nm',\psi)$. Both (R1)~and~(R2) are immediate
as above. The argument for~(R3) is very similar to that in the preceding paragraph.
For~(R4), we may assume that $\psi(m^*m) = \phi(n^*n) = 1$. Choose $V\owns \psi$ open
such that $V\subseteq\osupp((nm)^*nm)\cap \osupp((nm')^*nm')$, and $d\in D$ with
$\osupp(d)\subseteq V$ and $\psi(d)=1$. Then
\begin{align*}
U^\psi_{(nm)^*nm'}
&=\pi_\psi(d(nm)^*nm'd)=\pi_\psi(\alpha_m^\#(\alpha_{m^*}^\#(d)n^*n)m^*m'd)=\phi(n^*n)\pi_\psi(dm^*m'd)\\
&=\pi_\psi(dm^*m'd)=U^\psi_{m^*m'}\in\mathcal{U}_0(D'_{A^\delta}/J_\psi).
\end{align*}
Thus, $(nm,\psi)\sim (nm',\psi)$. Hence
\begin{equation}\label{eq:2ndfactor}
    \text{if }[m,\psi] = [m',\psi]\text{ and } r([n,\phi]) = s([m,\psi]),
     \text{ then }[nm,\psi] = [nm',\psi].
\end{equation}

By \eqref{eq:1stfactor}~and~\eqref{eq:2ndfactor}, if $[n, \phi] = [n', \phi]$, $[m, \psi]
= [m', \psi]$, and $r([m, \psi]) = s([n, \phi] )$, then $[nm, \psi] = [n'm, \psi] =
[n'm', \psi]$, so $M$ is well-defined.

To see that $I$ is well-defined, suppose that $[n,\phi] = [m,\phi]$. We claim that $(n^*,
\alpha_n(\phi))\sim (m^*, \alpha_{m}(\phi))$. Again (R1)~and~(R2) are clear, and~(R3) is
routine because $\alpha_{n^*} = \alpha_n^{-1}$ and similarly for $\alpha_{m^*}$.
For~(R4), we may assume that $\psi(m^*m) = \phi(n^*n) = 1$. For an open $U \owns \phi$
with $\alpha_n|_U = \alpha_m|_U$ and $d$ supported on $U$ with $\phi(d) = 1$,
\begin{align*} U^{\alpha_n(\phi)}_{nm^*} &=
\pi_{\alpha_n(\phi)}\big((ndn^*)nm^*(ndn^*)\big)\\ &= \iota_{n^*}(\pi_\phi(dm^*nd))
=\iota_{n^*}(U^\phi(m^*n)) \in \mathcal{U}_0(D'_{A^{\delta}}/J_{\alpha_n(\phi)}).
\end{align*}So~(R4) is also satisfied, and $(n^*, \psi)\sim (m^*, \psi)$. Thus $I$ is
well-defined.

The multiplication defined by $M$ is associative because multiplication in $A$ is
associative. By construction, we have $[n, \phi]^{-1}[n,\phi] = [n^*n,\phi] = [d, \phi]$
for any $d \in D$ with $\phi(d) > 0$. Similarly, $[n,\phi][n,\phi]^{-1} = [c,
\alpha_n(\phi)]$ for any $c \in D$ with $\alpha_n(\phi)(c) > 0$. Since $\pi_\phi(d) =
\phi(d)1_{D'_{A^\delta}/J_\phi}$ for $d \in D$, it is routine to check using the
definition of $M$ that
\[
[c, \alpha_m(\psi)][m,\psi] = [m,\psi] = [m,\psi][d, \psi]
\]
for any $c,d \in D$ such that $\alpha_m(\psi)(c) > 0$ and $\psi(d) > 0$. So $\Hh(A, D,
\delta)$ is a groupoid.

Finally, the formula $c_\delta([n,\phi]) = g$ if $n \in \Nn_g(D)$ is well defined
by~(R2), and it is multiplicative because $n \in A_g$ and $m \in A_h$ implies $nm \in
A_{g+h}$.
\end{proof}

We now show how to make $(\Hh(A, D, \delta), c_\delta)$ into a graded Hausdorff \'etale
groupoid.

\begin{thm}\label{thm:lche gpd}
Let $A$ be a separable $C^*$-algebra, $\delta$ a coaction of a countable discrete group
$\Gamma$ on $A$, and $D$ an abelian $C^*$-subalgebra of $A^\delta$ containing an approximate unit for $A^\delta$. Let $\Hh(A, D,
\delta)$ be the groupoid of Proposition~\ref{prp:groupoid-construction}. For $n \in
\Nn_\star(D)$ and an open set $X \subseteq \widehat{D}$ contained in $\osupp(n^*n)$ let
\[
    Z(n, X) := \{[n,\phi] : \phi \in X\} \subseteq \Hh(A, D, \delta).
\]
Then
\begin{equation}\label{eq:basis}
    \{Z(n, X) : n \in \Nn_\star(D),\ X \subseteq \widehat{D}\text{ is open and } X\subseteq\osupp(n^*n)\}
\end{equation}
constitutes a countable basis for a locally compact locally Hausdorff \'etale topology on
$\Hh(A, D, \delta)$, and $c_\delta$ is continuous with respect to this topology.
\end{thm}

To prove the theorem, we need some preliminary lemmas.

\begin{lem}\label{lem:classes agree}
Let $A$ be a separable $C^*$-algebra, $\delta$ a coaction of a countable discrete group
$\Gamma$ on $A$, and $D$ an abelian $C^*$-subalgebra of $A^\delta$ containing an approximate unit for $A^\delta$. Let $\Hh(A, D,
\delta)$ be the groupoid of Proposition~\ref{prp:groupoid-construction}. If $[n,\phi],
[m, \phi] \in \Hh(A, D, \delta)$ and $[n, \phi] = [m, \phi]$, then there is an open $U
\subseteq \widehat{D}$ with $\phi \in U \subseteq \osupp(m^*m) \cap \osupp(n^*n)$ such
that $[n, \psi] = [m, \psi]$ for all $\psi \in U$.
\end{lem}
\begin{proof}
Since $[n,\phi] = [m, \phi]$, there is an open $U'$ with $\phi \in U' \subseteq
\osupp(n^*n) \cap \osupp (m^*m)$ such that $\alpha_n \vert_{U'} = \alpha_m \vert_{U'}$.
By scaling $m$ and $n$ by appropriate elements of $D$, we may assume that $\psi(m^*m) =
\psi(n^*n) = 1$ for all $\psi \in U'$. Since $[n,\phi] = [m,\phi]$ we have $n^*m \in
A^\delta$, so~(R1), (R2)~and~(R3) are satisfied for $(n,\psi)$ and $(m,\psi)$ for each
$\psi \in U'$.

Since $U^\phi_{n^*m} = \pi_\phi(dn^*md)\in \mathcal{U}_0(D'_{A^\delta}/J_\phi)$, there is
a path of unitaries in $D_{A^\delta}'/J_\phi$ from $\pi_\phi(dn^*md)$ to
$1_{D'_{A^\delta}/J_\phi}$. Choose $0 = t(0) < t(1) < \cdots < t(k) = 1$ such that
$\|w_{t(j)} - w_{t(j+1)} \| < \frac{1}{4}$ for all $j$. Let $a_0 = dn^*md$ and $a_k=d$,
and choose $a_j \in \pi_\phi^{-1}(w_{t(j)})$, $0 < j < k$. Since $\psi \mapsto
\|\pi_\psi(a)\|$ is upper semicontinuous, there is an open $U \owns \phi$ such that
\[
\|\pi_\psi(a_j a_j^* - a_k)\| < \frac{1}{4},\hskip1.5ex
    \|\pi_\psi(a_j^* a_j - a_k)\| < \frac{1}{4},\hskip1.5ex
    \|\pi_\psi(a_j^*)\|<2,\hskip1.5ex\text{and}\hskip1.5ex
    \|\pi_\psi (a_j - a_{j+1})\| < \frac{1}{4}
\]
for all $\psi \in U$ and $j < k$. We claim that $(n,\psi)\sim (m,\psi)$ for $\psi \in U$.

Fix $\psi\in U$. We have already seen that (R1)--(R3) are satisfied for $(n,\psi)$ and
$(m,\psi)$. The first two properties of the $a_j$ ensure that $\pi_\psi (a_ja_j^*)$,
$\pi_\psi (a_ja_j^*)$, and $\pi_\psi(a_j)$ are invertible, and that $\|\pi_\psi
(a_ja_j^*)^{-1}\|<4/3$. So $\|\pi_\psi (a_j - a_{j+1})\| < 3/8
<\|\pi_\psi(a_j)^{-1}\|^{-1}$. So \cite[Proposition~2.1.11]{RLL} shows that each
$\pi_\psi(a_j) \mathbin{\sim_h} \pi_\psi(a_{j+1})$ in
$\big(D'_{A^\delta}/J_\psi\big)^{-1}$. Hence $U^\psi_{n^*m} \mathbin{\sim_h}
\pi_\psi(a_k)=1_{D'_{A^\delta}/J_\psi}$ in $\big(D'_{A^\delta}/J_\psi\big)^{-1}$, and
then \cite[Proposition~2.1.8]{RLL} gives $U^\psi_{n^*m} \mathbin{\sim_h}
1_{D'_{A^\delta}/J_\psi}$ in the unitary group of $D'_{A^\delta}/J_\psi$. So~(R4) is
satisfied, and hence $(n,\psi)\sim (m,\psi)$.
\end{proof}

\begin{lem}\label{lem:same alpha}
Let $A$ be a separable $C^*$-algebra, $\delta$ a coaction of a countable discrete group
$\Gamma$ on $A$, and $D$ an abelian $C^*$-subalgebra of $A^\delta$ containing an approximate unit for $A^\delta$. Suppose that $n,m \in
\Nn(D)$ satisfy $\|n-m\| < \|n\|/5$. Then $\alpha_n(\phi) = \alpha_m(\phi)$ for all $\phi
\in \widehat{D}$ such that $\phi(n^*n) > \|n\|^2/2$.
\end{lem}
\begin{proof}
Since $\alpha_n = \alpha_{n/\|n\|}$ and $\alpha_m = \alpha_{m/\|n\|}$, it suffices to
prove the result for $\|n\| = 1$. So, by assumption, $\|n-m\| < 1/5$ giving $\|m\| \le
6/5$. For any $d \in D$ with $\|d\| \le 1$, we have
\begin{equation}\label{eq:estimate}\textstyle
\|n^* d n - m^* d m\| \le \|n^*dn - n^*dm\| + \|n^*dm - m^*dm\|
    < \frac{1}{5} + \frac{6}{25} < \frac12.
\end{equation}
Fix $\phi \in \widehat{D}$ with $\phi(n^*n) > \frac12$, and $d \in D$ with $0 \le d \le
1$ and $\phi(d) = 1$. Then $\alpha_n(\phi)(n dn^*) = \phi(d) \alpha_n(\phi)(nn^*) =
\phi(n^*n) > \frac12$, so~\eqref{eq:estimate} gives $\alpha_n(\phi)(mdm^*) >
\alpha_n(\phi)(ndn^*) - \frac12 > 0$. Hence $0 < \alpha_n(\phi)(m dm^*) =
\alpha_m^{-1}(\alpha_n(\phi))(d) \alpha_n(\phi)(mm^*)$, giving
$\alpha_m^{-1}(\alpha_n(\phi))(d) \not= 0$. So if $0 \le d \le 1$ and $\phi(d) = 1$, then
$\alpha^{-1}_{m}(\alpha_n(\phi))(d) > 0$. So Urysohn's lemma forces
$\alpha^{-1}_{m}(\alpha_n(\phi)) = \phi$.
\end{proof}

\begin{proof}[Proof of Theorem~\ref{thm:lche gpd}]
To see that the $Z(n,X)$ are a basis for a topology, suppose that $[n,\phi] \in Z(n,X)
\cap Z(n',Y)$. Then $\phi \in X \cap Y$, and $[n,\phi] = [n',\phi]$. By
Lemma~\ref{lem:classes agree}, there is an open $W$ with $\phi \in W \subseteq X \cap Y$
such that $[n,\psi] = [n',\psi]$ for all $\psi \in W$. Hence $Z(n,W) \subseteq Z(n,X)
\cap Z(n',Y)$.

To see that this topology is second countable, we first claim that each $\Nn_g(D)$ has a
countable dense subset. For this, fix a dense sequence $a_i$ in $A_g$. For $i,j \in \NN$
fix $n_{i,j} \in B(a_i, 1/j) \cap \Nn_g(D)$ if this set is nonempty, and otherwise let
$n_{i,j} = 0 \in \Nn_g(D)$. Fix $n \in \Nn_g(D)$. Choose $i_k$ with $\|a_{i_k} - n\| <
\frac{1}{3k}$. Then $n_{i_k, 2k} \in B(a_{i_k}, 1/2k)$, forcing $\|n_{i_k, 2k} - n\| <
\frac{1}{k}$. So $n \in \overline{\{n_{i,j} : i,j \in \NN\}}$, proving the claim. Now
since $\Gamma$ is countable, $\Nn_\star(D)$ has a countable dense sequence
$(n_i)_{i=1}^\infty$. Choose a countable basis $\{U_j\}$ for $\widehat{D}$. We claim that
for $n\in N_\star(D)$, an open subset $X\subseteq\osupp(n^*n)$, and $\phi\in X$, there
are $i , j \in \NN$ such that $[n,\phi]\in Z(n_i,U_j)\subseteq Z(n,X)$. Since $[n,\phi] =
[nd,\phi]\in Z(nd,X\cap\osupp(d))\subseteq Z(n,X)$ for $d\in D$ with $0\le d\le 1$ and
$\phi(d)=1$, we may assume that $\phi(n^*n) = \|n\|^2$. Choose $j$ so that $\phi \in U_j
\subseteq X$ and $\psi(n^*n) > \|n\|^2/2$ for all $\psi \in U_j$. Fix a subsequence $(
n_{i_k} )_{k = 1}^\infty$ of $( n_i )_{i = 1}^\infty$ converging to $n$ with each
$n^*n_{i_k} \in A^\delta$  and $\| n_{ i_k } - n \| < \| n \| / 5$.  So,
Lemma~\ref{lem:same alpha} gives $\alpha_n|_{U_j} = \alpha_{n_{i_k}}|_{U_j}$. Since
$\sup_{\psi \in U_j}\|\pi_\psi( d_\psi n^*n_{i_k} d_\psi  ) - \pi_\psi( n^*n )\| \to 0$
where $d_\psi \in D$ with $\| d \| = 1$, $\osupp ( d_\psi ) \subseteq U_j$, and $\psi
(d_\psi ) = 1$, we have $\sup_{\psi \in U_j} \|U^\psi_{n^*n_{i_k}} - U^\psi_{n^*n}\| \to
0$. Since $U^\psi_{n^*n} = 1_{D'_{A^\delta}/J_\psi}$ and since unitaries in
$B(1_{D'_{A^\delta}/J_\psi}; 2)$ are homotopic to $1_{D'_{A^\delta}/J_\psi}$, for large
$k$, we have $U^\psi_{n^*n_{i_k}} \mathbin{\sim_h} 1_{D'_{A^\delta}/J_\psi}$ for $\psi
\in U_j$.  Thus, $[n, \phi] \in Z(n_{i_k}, U_j) \subseteq Z(n, X)$. So the $Z(n_i, U_j)$
form a countable basis for the topology.

To see that $\Hh(A, D, \delta)$ is locally Hausdorff and \'etale, fix a basic open set
$Z(n, X)$. The source map $[n,\psi] \mapsto \psi$ is a homeomorphism $h : Z(n,X) \to X$:
it is bijective by definition of $\sim$, continuous as $h^{-1}(Y) = Z(n,Y)$ for $Y
\subseteq X$, and open because each open subset of $Z(n,X)$ is a union of sets of the
form $Z(n,Y)$ with $Y \subseteq X$ open, and each $h(Z(n,Y)) = Y$ is open. Similarly, the
range map is a homeomorphism $Z(n,X) \to \alpha_n(X)$. Thus, since $\widehat{D}$ is
Hausdorff, $\Hh(A, D, \delta)$ is locally Hausdorff and \'etale.

The map $I$ is a homeomorphism because $I(Z(n,X)) = Z(n^*, \alpha_n(X))$ and $\alpha_n$
is a homeomorphism on $\osupp(n^*n)$. To see that $M$ is continuous, suppose that
$[n_i,\phi_i] \to [n,\phi]$, that $[m_i,\psi_i] \to [m,\psi]$, and that each $\phi_i =
\alpha_{m_i}(\psi_i)$. Then the preceding paragraph gives $s([n,\phi]) = r([m,\psi])$,
and then $M([n,\phi],[m,\psi]) = [nm,\psi]$. Fix an open $V$ with $\psi \in V \subseteq
\osupp(m^*m) \cap \alpha_m^{-1}(\osupp(n^*n))$. Then $Z(m,V) \owns [m,\psi]$ and $Z(n,
\alpha_m(V)) \owns [n,\phi]$ are open, giving $[n_i,\phi_i] \in Z(n, \alpha_m(V))$ and
$[m_i,\psi_i] \in Z(m,V)$ for large $i$. Since $Z(m,V)$ and $Z(n, \alpha_m(V))$ are
bisections, for large $i$ we have $[n_i,\phi_i] = [n, \phi_i]$ and $[m_i, \psi_i] = [m,
\psi_i]$, so $M([n_i,\phi_i],[m_i,\psi_i]) = [nm, \psi_i]$. In particular, $\psi_i \to
\psi$, and as $s$ is a homeomorphism on $Z(mn,V)$, we obtain $[nm, \psi_i] \to
[nm,\psi]$. So $M$ is continuous.

For local compactness, fix $\gamma \in \Hh(A, D, \delta)$ and an open $W \owns \gamma$.
Choose $n \in \Nn_\star(D)$ and $X$ open with $\gamma \in Z(n,X)$. Since $\widehat{D}$ is
locally compact, there is a compact neighbourhood $K$ of $s(\gamma)$ in $X$. Now $\{[n,
\phi] : \phi \in K\}$ is the inverse image of $K$ under the homeomorphism $s|_{Z(n,X)}$,
and hence a compact neighbourhood of $\gamma$. Finally, $c_\delta$ is continuous because
it is constant on basic open sets.
\end{proof}

Our results so far do not require that $\clsp \Nn_\star(D)$ is all of $A$; but our next
lemma indicates that this is the situation of greatest interest.

\begin{lem}
Let $\delta$ be a coaction of a discrete group $\Gamma$ on a separable $C^*$-algebra $A$,
and $D$ an abelian $C^*$-subalgebra of $A^\delta$ containing an approximate unit for $A^\delta$. Let $A_\Nn = \clsp \Nn_\star(D)
\subseteq A$. Then $A_\Nn$ is a $C^*$-algebra, $\delta_\Nn := \delta|_{A_\Nn}$ is a
coaction, $D$ is an abelian $C^*$-subalgebra of $A_\Nn^{\delta_\Nn}$ and contains an approximate unit for $A_\Nn^{\delta_\Nn}$, and $\Hh(A, D, \delta) \cong
\Hh(A_\Nn, D, \delta_\Nn)$.
\end{lem}
\begin{proof}
As $\Nn_\star(D)$ is closed under multiplication and adjoints, $A_\Nn$ is a
$C^*$-algebra. Since $D \subseteq \Nn_e(D)$, we have $D \subseteq A_\Nn$, and $D$ clearly
contains an approximate unit for $A^{\delta_\Nn}_\Nn$. Since $A_\Nn = \clsp \bigcup_g
A_g$, the restriction of $\delta_\Nn := \delta|_{A_\Nn}$ takes values in $A_\Nn \otimes
C^*_r(\G)$. It is nondegenerate because $D \subseteq A^{\delta_\Nn}_\Nn$ contains an
approximate unit. Finally, it is easy to see that $[n, \phi] \mapsto [n, \phi]$ is an isomorphism from $\Hh(A_\Nn, D, \delta_\Nn)$ to $\Hh(A, D, \delta)$.
\end{proof}

\section{The interior of the isotropy in an \texorpdfstring{\'etale}{etale}
groupoid}\label{sec:intiso}

This section contains some technical results that we need in order to prove our
reconstruction results in the next section. Our proof of the first,
Lemma~\ref{lem:bundle}, is largely due to Becky Armstrong; the key elements, in the more
general situation of twisted groupoid $C^*$-algebras, will appear in her PhD thesis.

\begin{lem}[Armstrong]\label{lem:bundle}
Let $\G$ be a locally compact Hausdorff \'etale groupoid such that $\Io{\G}$ is abelian.
Then $C^*(\Io{\G}) = C^*_r(\Io{\G})$. The inclusion $C_0(\go) \hookrightarrow
C^*(\Io{\G})$ makes $C^*(\Io{\G})$ into a $C_0(\go)$-algebra.  The fibre homomorphisms
$\pi_u : C^*(\Io{\G}) \to C^*(\Io{\G})_u$ have the property that $u \mapsto \|\pi_u(a)\|$
is continuous for $a \in C^*(\Io{\G})$. For each $u \in \go$ there is an isomorphism
$C^*(\Io{\G}))_u\cong C^*(\Io{\G}_u)$ that takes $\pi_u(d)$ to $d(u)1_{C^*(\Io{\G}_u)}$
for $d \in C_0(\go)$.
\end{lem}
\begin{proof}
Theorem~3.5 of \cite{Ren2015} shows that $C^*(\Io{\G}) = C^*_r(\Io{\G})$. For $f \in
C_c(\Io{\G})$, $d \in C_0(\go)$ and $\gamma \in \Io{\G}$, we have that $(df)(\gamma) =
d(r(\gamma))f(\gamma) = f(\gamma)d(s(\gamma)) = (fd)(\gamma)$. So $C_0(\go)$ is central
in $C^*(\Io{\G})$ by continuity, and any approximate unit for $C_0(\go)$ is an
approximate unit for $C^*(\Io{\G})$. Thus $C^*(\Io{\G})$ is a $C_0(\go)$-algebra.

Fix $u \in \go$. For each $\gamma \in \Io{G}_u$, choose $a_\gamma \in C_c(\Io{\G})$
supported on a bisection with $a_\gamma(\gamma) = 1$. Then $\gamma \mapsto
\pi_u(a_\gamma)$ is a unitary representation of $\Io{\G}_u$ in $C^*(\Io{\G})_u$, and thus
determines a homomorphism $\tilde\pi_u : C^*(\Io{\G}_u) \to C^*(\Io{\G})_u$.

The regular representation $\rho_u : C^*_r(\Io{\G}) \to \Bb(\ell^2(\Io{\G}_u))$
satisfies $\rho_u(f) = 0$ for $f \in C_0(\go\setminus\{u\})$, and hence descends to a
representation
\[
\tilde\rho_u : C^*(\Io{\G})_u \to \Bb(\ell^2(\Io{\G}_u)).
\]
The representation $\tilde\rho_u \circ \tilde{\pi}_u$ is precisely the regular
representation $C^*(\Io{\G}_u)$ and hence faithful since $\Io{\G}_u$ is abelian.
Identifying $C^*_r(\Io{\G}_u)$ with $C^*(\Io{\G}_u)$,
\[
\tilde{\pi}_u \circ \tilde\rho_u\big(\pi_u(a_\gamma)) = \pi_u(a_\gamma)
\]
for $\gamma \in \Io{\G}_u$, and so $\tilde{\pi}_u \circ \tilde\rho_u =
\id_{C^*(\Io{\G})_u}$. Since $\tilde\rho_u(\pi_u(d)) = d(u)1_{C^*(\Io{\G}_u)}$, this
$\tilde\rho_u : C^*(\Io{\G})_u \to C^*(\Io{\G}_u)$ is the desired isomorphism.

Fix $a \in C_c(\Io{\G})$ supported on a bisection $U$. Write $\sigma : s(U) \to U$ for
the inverse of the source map. Then $\|\pi_u(a)\| = |a(\sigma(u))|$, so $u \mapsto
\|\pi_u(a)\|$ is continuous. An $\frac{\varepsilon}{3}$-argument now shows that $u
\mapsto \|\pi_u(a)\|$ is continuous for all $a \in C^*(\Io{\G})$.
\end{proof}

The continuity of the map $u \mapsto \|\pi_u(a)\|$ above also follows from \cite[Corollary 5.6]{LR}. For our next lemma, recall from Section~\ref{sec:background} that there is a
norm-decreasing injection $a \mapsto f_a$ from $C^*_r(\G)$ to $C_0(\G)$ given by
$f_a(\gamma) = \big(\rho_{s(\gamma)}(a) e_{s(\gamma)} \mid e_\gamma\big)$. We show
that $\{a \in C^*_r(\G) : f_a\text{ is supported on }\Io{\G}\}$ is a $C_0(\go)$-algebra.

\begin{lem}\label{lem:fibres}
Let $\G$ be a locally compact Hausdorff \'etale groupoid such that $\Io{\G}$ is abelian.
Let $D := C_0(\go) \subseteq C^*_r(\G)$. Let $A := \{a \in C^*_r(\G) : \osupp(f_a)
\subseteq \Io{\G}\}$. Then $A$ is a $C_0(\go)$-algebra with respect to $D \hookrightarrow
A$. For each $u \in \go$, let $J_u := \overline{\{f \in D : f(u) = 0\}A}$, the ideal of
$A$ generated by $\{f \in D : f(u) = 0\}$, and let $A_u := A/J_u$. Then there is an
isomorphism $\tilde\Phi_u : A_u \to C^*(\Io{\G}_u)$ such that
\[
\tilde\Phi_u(f) = \sum_{\gamma \in \Io{\G}_u} f(\gamma) \lambda_\gamma
    \quad\text{for all $f \in C_c(G) \cap A$,}
\]
where $\lambda_\gamma$ is the image of $\gamma$ in the left regular representation of $\Io{\G}_u$ in $C^*(\Io{\G}_u)$.
\end{lem}
\begin{proof}
For $b \in C^*_r(\G)$ and $g \in D$, we have $f_{gb}(\gamma) = g(r(\gamma)) f_b(\gamma)$
and $f_{bg}(\gamma) = f_b(\gamma)g(s(\gamma))$, and so $D \hookrightarrow A$ is a central
inclusion. Since $D$ contains an approximate unit for $C^*_r(\G)$, we deduce that $A$ is
a $C_0(\go)$-algebra. So we fix $u \in \go$ and show that there is an isomorphism
$\tilde\Phi_u : A_u \to C^*(\Io{\G}_u)$.

Let $P \in \Bb(\ell^2(\G_u))$ be the orthogonal projection onto $\ell^2(\Io{\G}_u)$.
Define $\Phi_u : C^*_r(\G) \to \Bb(\ell^2(\Io{\G}_u))$ by $\Phi_u(a) := P \rho_u(a)
P$. Then for $f \in C_c(\G)$ we have
\[\textstyle
\Phi_u(f) = \sum_{\gamma \in \Io{\G}_u} f(\gamma) \lambda_\gamma \in C^*(\Io{\G}_u),
\]
and hence $\Phi_u(C^*_r(\G)) \subseteq C^*(\Io{\G}_u)$ by continuity.

If $a \in \overline{\{f \in D : f(u) = 0\}A}$, then $\Phi_u(a) = 0$. Hence $\Phi_u$
induces a map $\tilde\Phi_u : A_u \to C^*(\Io{\G}_u)$. Proposition~4.2 of
\cite{Renault80} shows that $f_{ab}(\gamma) = \sum_{\alpha\beta = \gamma} f_a(\alpha)
f_b(\beta) = (f_a * f_b)(\gamma)$ and $f_{a^*}(\gamma) = \overline{f_a(\gamma^{-1})} =
f^*_a(\gamma)$, so $\tilde\Phi_u$ is a $*$-homomorphism. This $\tilde\Phi_u$ is
surjective because its image contains the canonical generators of $C^*(\Io{\G}_u)$.

To see that $\tilde\Phi_u$ is injective, we first claim that $J_u = \{a \in A :
f_a|_{\Io{\G}_u} = 0\}$. For $\subseteq$, observe that  $\{a \in A : f_a|_{\Io{\G}_u} =
0\}$ is an ideal containing $\{f \in D : f(u) = 0\}$. For the reverse containment,
suppose that $f_a|_{\Io{\G}_u} = 0$ and $\|a\| \le 1$. Since $f_a|_{\Io{\G}_u} = 0$, we
have $\|\pi_u(a)\| = 0$. The map $\go \owns v \mapsto \|\pi_v(a)\|$ is upper
semicontinuous by \cite[Proposition~C.10(a)]{Williams:cp}. So for $n \in \NN$ the set
$U_n := \{v \in \go : \|\pi_v(a)\| < 1/n\}$ is an open neighbourhood of $u$, and $\go
\setminus U_n$ is compact. So there exists $g_n \in C_0(\go)_+$ with $g_n \le 1$, $g_n(u)
= 0$, and $g_n \equiv 1$ on $\go \setminus U_n$. Now $\|a - g_na\| < \frac{1}{n}$. Since
each $g_n a \in J_u$, we deduce that $a = \lim_n g_n a \in J_u$. This proves the claim.
Now suppose $b\in A\setminus J_u$. Then $f_b|_{\Io{\G}_u} \not= 0$, so $f_{b^*b}(u) \not=
0$. Hence $(\Phi_u(b^*b) \delta_u \mid \delta_u) = f_{b^*b}(u) \not= 0$, forcing
$\tilde\Phi_u(b^*b) \not= 0$. Thus $\tilde\Phi_u(b) \not= 0$ because $\tilde\Phi_u$ is a
homomorphism. Hence $\tilde{\Phi}_u$ is injective.
\end{proof}

For the next result, recall from \cite[Proposition~1.9]{Christopher} that if $\G$ is an
\'etale groupoid then the canonical inclusion $C_c(\Io{\G}) \hookrightarrow C_c(\G)$
extends to an injective homomorphism $\iota : C^*_r(\Io{\G}) \hookrightarrow C^*_r(G)$.

\begin{cor}\label{cor:M<->C*r}
Let $\G$ be a locally compact Hausdorff \'etale groupoid such that $\Io{\G}$ is abelian.
Let $A := \{a \in C^*_r(\G) : \osupp(f_a) \subseteq \Io{\G}\}$. Then $A =
C_0(\go)'_{C^*_r(\G)} = \iota(C^*(\Io{\G}))$.
\end{cor}
\begin{proof}
Fix $a \in A$. For $d \in C_0(\go)$, we have $f_d = d$ and
\cite[Proposition~4.2]{Renault80} gives $f_{ad}(\gamma) = (f_a * f_d)(\gamma) =
f_a(\gamma)d(s(\gamma))$ and similarly $f_{da}(\gamma) = d(r(\gamma)) f_a(\gamma)$. Since
$a \in A$, we have $f_a(\gamma) \not= 0$ only if $r(\gamma) = s(\gamma)$, so $f_{ad} =
f_{da}$. Since $b \mapsto f_b$ is injective, it follows that $A \subseteq
C_0(\go)'_{C^*_r(\G)}$. Now fix $a \not\in A$. Since $f_a$ is continuous, there is then
$\gamma \not\in \operatorname{Iso}(\G)$ such that $f_a(\gamma)\ne 0$. Fix $d \in
C_0(\go)$ with $d(r(\gamma')) = 1$ and $d(s(\gamma')) = 0$. Then $f_{da}(\gamma) =
d(r(\gamma))f_a(\gamma) = f_a(\gamma) \not= 0$, and $f_{ad}(\gamma) = f_a(\gamma)
d(s(\gamma)) = 0$. So $da \not= ad$. Hence $C_0(\go)'_{C^*_r(\G)} \subseteq A$, and thus
$C_0(\go)'_{C^*_r(\G)} = A$.

If $f \in C_c(\Io{\G})$ and $d \in C_0(\go)$, then $(df)(\gamma) = d(r(\gamma))f(\gamma)$
and $(fd)(\gamma) = f(\gamma)d(s(\gamma))$ for every $\gamma \in \G$. Since $\osupp(f)
\subseteq \Io{\G}$, we obtain $df = fd$. Thus $\iota(C^*(\Io{\G})) \subseteq A$.

Both $A$ and $\iota(C^*_r(\Io{\G}))$ are $C_0(\go)$-algebras with respect to the
inclusion of $D = C_0(\go)$ in both. Write $J_u$ for the ideal of $A$ generated by
$C_0(\go\setminus\{u\}) \subseteq D$ and $K_u$ for the ideal of $\iota(C^*(\Io{\G}))$
generated by $C_0(\go\setminus\{u\}) \subseteq D$, so $A_u = A/J_u$ and
$\iota(C^*(\Io{\G}))_u = \iota(C^*(\Io{\G}))/K_u$.

Lemma~\ref{lem:fibres} gives isomorphisms $\tilde\Phi_u^{-1} : \iota(C^*_r(\Io{\G}))_u
\to A_u$ such that $\tilde\Phi_u^{-1}(\iota(f) + K_u) = f + K_u$ for $f \in
C_c(\Io{\G})$. Thus $\{m_f : u \mapsto f + J_u \mid f \in C_c(\Io{\G})\}$ is a fibrewise
dense vector space of continuous sections of $\mathcal{A} := \bigsqcup_u A_u$ that are
the images of a fibrewise dense vector space of continuous sections $k_f$ of $\mathcal{I}
:= \bigsqcup_u \iota(C^*_r(\Io{\G}))_u$. So \cite[Proposition~1.6]{Fell} gives
$\mathcal{A} \cong \mathcal{I}$ as bundles, and hence there is an isomorphism
$\Gamma_0(\mathcal{A}) \cong \Gamma_0(\mathcal{I})$ carrying $k_f$ to $m_f$. Since $f
\mapsto k_f$ is an isomorphism $C^*_r(\Io{\G}) \to \Gamma_0(\mathcal{I})$ and $f \mapsto
m_f$ is an isomorphism $A \to \Gamma_0(\mathcal{A})$, we obtain an isomorphism
$\iota(C^*_r(\Io{\G})) \cong A$ extending $\id_{C_c(\Io{\G})}$. So $A =
\overline{C_c(\Io{\G})} = \iota(C^*_r(\Io{\G}))$.
\end{proof}

We take the opportunity to resolve a loose end from \cite{BNRSW}.

\begin{cor}[{cf. \cite[Theorem~4.3]{BNRSW}}]\label{cor:C*r masa}
Let $\G$ be a locally compact Hausdorff \'etale groupoid such that $\Io{\G}$ is abelian.
Then $\iota(C^*(\Io{\G})) \subseteq C^*_r(\G)$ is maximal abelian.
\end{cor}
\begin{proof}
Lemma~4.4 of \cite{BNRSW} says that $\iota(C^*(\Io{\G})) \subseteq C^*_r(\G)$ is maximal
abelian if $\{a \in C^*_r(\G) : \osupp(f_a) \subseteq \Io{\G}\} \subseteq
\iota(C^*(\Io{\G}))$, which follows from Corollary~\ref{cor:M<->C*r}.
\end{proof}

\section{Reconstruction of groupoids}\label{sec:reconstruction}

Let $\Gamma$ be a discrete group. If $\G$ is an \'etale groupoid, and $c : \G \to \Gamma$ is a
continuous cocycle, we call $(G, c)$ a $\Gamma$-graded groupoid. To state our main theorem, we
first show that $c$ induces a coaction on $C^*_r(\G)$. Recall that, for $g \in \Gamma$, we write
$\lambda_g \in \Bb(\ell^2(\Gamma))$ for the image of $g$ in the left regular representation of
$\Gamma$.

\begin{lem}\label{lem:coaction}
Let $\G$ be a locally compact Hausdorff \'etale groupoid. Suppose that $c : \G \to
\Gamma$ is a continuous cocycle. Then there is a coaction $\delta_c : C^*_r(\G)\to
C^*_r(\G)\otimes C^*_r(\Gamma)$ such that $\delta_c(f) = f\otimes \lambda_g$ whenever
$g\in\Gamma$ and $f\in C_c(\G)$ satisfy $\supp(f)\subseteq c^{-1}(g)$.
\end{lem}
\begin{proof}
Let $\Hh$ be the Hilbert $C^*(\go)$-module completion of $C_c(\G)$ under $\langle f,
g\rangle_{C_0(\go)} = (f^*g)|_{\go}$. For $g \in \Gamma$, write $C_c(\G)_g :=
C_c(c^{-1}(g)) \subseteq C_c(\G)$, and let $\Hh_g = \overline{C_c(\G)_g} \subseteq \Hh$.
The $\Hh_g$ are mutually orthogonal because $\go \subseteq c^{-1}(e)$, so a calculation
using inner product shows that there are isometries $V_h : \Hh \to \Hh \otimes
\ell^2(\Gamma)$ such that $V_h(\xi) = \xi \otimes e_{gh}$ for $\xi \in \Hh_g$. As in
\cite[Appendix~A]{aHKS} there is a faithful representation $\pi : C^*_r(\G) \to\Ll(\Hh)$
extending left multiplication. So $\bigoplus_h \Ad V_h \circ \pi : C^*_r(\G) \to \Ll(\Hh
\otimes \ell^2(\Gamma))$ is faithful.

A routine calculation shows that for $f \in C_c(\G)_g$ and $\xi \in C_c(\G)_k$ and $h,l
\in \Gamma$, we have $\bigoplus_h \big(\Ad V_h \circ \pi\big)(f)(\xi \otimes e_l) =
e_{kh,l} (\pi(f) \otimes \lambda_g)(\xi \otimes \delta_l)$. So $\delta_c := (\pi^{-1}
\otimes \id) \circ (\bigoplus_h \Ad V_h \circ \pi)$ satisfies $\delta_c(f) = f \otimes
\lambda_g$. It is routine to check that this is a coaction.
\end{proof}

\begin{thm}\label{thm:1}
Fix a discrete group $\Gamma$ and $\Gamma$-graded second-countable locally compact
Hausdorff \'etale groupoids $(\G_1,c_1), (\G_2,c_2)$ with $\Io{c_i^{-1}(\id_\Gamma)}$ torsion-free
and abelian.
\begin{enumerate}
\item Suppose that $\kappa : \G_2 \to \G_1$ is an isomorphism satisfying
    $c_1\circ\kappa=c_2$. Then there is an isomorphism $\phi : C^*_r(\G_1)\to
    C^*_r(\G_2)$ such that $\phi(f) = f \circ \kappa$ for $f \in C_c(\G_1)$. We have
    $\phi(C_0(\G_1^{(0)})) = C_0(\G_2^{(0)})$ and $\delta_{c_2}\circ\phi =
    (\phi\otimes\id)\circ\delta_{c_1}$.
\item Suppose that $\phi:C^*_r(\G_1)\to C^*_r(\G_2)$ is an isomorphism satisfying
    $\phi(C_0(\G_1^{(0)}))=C_0(\G_2^{(0)})$ and
    $\delta_{c_2}\circ\phi=(\phi\otimes\id)\circ\delta_{c_1}$. Then there is an
    isomorphism $\kappa : \G_2 \to \G_1$ such that $\kappa|_{\G_2^{(0)}}$ is the
    homeomorphism induced by $\phi|_{C_0(\G_1^{(0)})}$ and $c_1 \circ \kappa = c_2$.
\end{enumerate}
\end{thm}

The first step in proving Theorem~\ref{thm:1} is showing that $C^*_r(c^{-1}(\id_\Gamma)) \cong
C^*_r(G)^{\delta_c} \subseteq C^*_r(\G)$.

\begin{lem}\label{lem:neutral subalg}
Let $(\G, c)$ be a $\Gamma$-graded locally compact Hausdorff \'etale groupoid. The
canonical inclusion $\iota : C_c(c^{-1}(\id_\Gamma)) \to C_c(\G)$ extends to an isomorphism
$C^*_r(c^{-1}(\id_\Gamma)) \cong C^*_r(\G)^{\delta_c}$.
\end{lem}
\begin{proof}
Proposition~1.9 of \cite{Christopher} shows that $\iota$ extends to an injective homomorphism from
$C^*_r(c^{-1}(\id_\Gamma))$ to $C^*_r(\G)$. Clearly $\iota(C^*_r(c^{-1}(\id_\Gamma))) \subseteq
C^*_r(\G)^{\delta_c}$. For $g \in \Gamma$, if $f \in C_c(\G)_g$, then $\Phi^{\delta_c}(\iota(f)) =
\delta_{\id_\Gamma,g}(\iota(f))$, so $C^*_r(\G) = \clsp \bigcup_g C_c(\G)_g$. It follows that
$\iota(C^*_r(c^{-1}(\id_\Gamma))) = \Phi^{\delta_c}(C^*_r(\G)) = C^*_r(\G)^{\delta_c}$ by
linearity and continuity.
\end{proof}

We now show that graded groupoids determine triples as in Theorem~\ref{thm:lche gpd}.

\begin{lem}\label{lem:gpd sds}
Let $\Gamma$ be a locally compact group, let $\G$ be a second-countable locally compact
Hausdorff \'etale groupoid and let $c : \G \to \Gamma$ be a continuous cocycle. Then $C^*_r(\G)^{\delta_c}$ is separable $C^*$-algebra and $C_0(\go)$ is an abelian $C^*$-subalgebra of the
generalised fixed-point algebra $C^*_r(\G)^{\delta_c}$ and contains an approximate unit for $C^*_r(\G)^{\delta_c}$.
\end{lem}
\begin{proof}
Since $\G$ is second-countable, $C^*_r(\G)^{\delta_c}$ is separable. Clearly $C_0(\go)$
is abelian, contains an approximate unit for $C^*_r(\G)^{\delta_c}$, and is contained in
$C^*_r(\G)^{\delta_c}$.
\end{proof}

To prove Theorem~\ref{thm:1} we show that $\big(\Hh(C^*_r(\G), C_0(\go), \delta_c),
c_{\delta_c}\big) \cong (\G, c)$ (cf. \cite[Proposition~4.13(ii)]{Ren2008},
\cite[Proposition~4.8]{BCW}, \cite[Corollary~3.11]{ABHS} and \cite[Proposition~3.6]{CR2}). For
$u \in \go$, we write $\hat{u} : C_0(\go) \to \CC$ for evaluation at $u$.

\begin{prp}\label{prp:isomorphism}
Let $\Gamma$ be a discrete group, and $(\G, c)$ a $\Gamma$-graded second-countable
locally compact Hausdorff \'etale groupoid with $\Io{c^{-1}(\id_\Gamma)}$ torsion-free and
abelian. There is an isomorphism $\theta : (\G, c) \to (\Hh(C^*_r(\G), C_0(\go),
\delta_c),c_{\delta_c})$ such that for $\gamma \in \G$ and $n \in C_c(\G)$ supported on a
bisection $U \subseteq c^{-1}(c(\gamma))$ with $n(\gamma) > 0$, we have $\theta(\gamma) =
[n, \widehat{s(\gamma)}]$.
\end{prp}

To prove proposition~\ref{prp:isomorphism}, we need two lemmas. We implicitly identify
$C_0(\go)\widehat{\;}$ with $\go$ via $\hat{u} \mapsto u$. Henceforth in this section, we identify
$C_0(\go)'_{C^*_r(c^{-1}(\id_\Gamma))}$ with $C^*_r(\Io{G})$ using Corollary~\ref{cor:M<->C*r} and
Lemma~\ref{lem:bundle}. For $u \in \go$, we identify $\big(C_0(\go)'_{C^*_r(c^{-1}(e) )}\big)_u$
with $C^*_r(\Io{c^{-1}(\id_\Gamma)}_u)$, and $\pi_u : C_0(\go)'_{C^*_r(c^{-1}(\id_\Gamma))} \to
C_0(\go)'_{C^*_r(c^{-1}(\id_\Gamma))}/J_{\hat{u}}$ with the regular representation $\rho_u :
C^*_r(\Io{c^{-1}(\id_\Gamma)}) \to \Bb(\ell^2(\Io{c^{-1}(\id_\Gamma)}_u))$.

\begin{lem}[{cf. \cite[Proposition~4.8]{Ren2008}, \cite[Lemma~4.10(ii)]{BCW}}]\label{lem:fn<->alphan}
Let $G$ be a second-countable locally compact Hausdorff \'etale groupoid. Let $D :=
C_0(\go) \subseteq C^*_r(G)$. If $n \in \Nn_{C^*_r(G)}(D)$ and $f_n$ as in~\eqref{eq:Jean's jmap}
satisfies $f_n(\gamma) \not= 0$, then $r(\gamma) = \alpha_n(s(\gamma))$.
\end{lem}
\begin{proof}
Suppose for contradiction that $r(\gamma) \not= \alpha_n(s(\gamma))$. Since
$nn^*(r(\gamma)) \ge |f_n(\gamma)|^2 > 0$, there exist orthogonal $d, d' \in
C_c(\osupp(nn^*))$ such that $d(r(\gamma)) = 1 = d'(\alpha_n(s(\gamma)))$. So
\begin{align*}
0 &= (dd' nn^*)(r(\gamma))
    = (dn (d'\circ\alpha_n) n^*)(r(\gamma))\\
    &\ge d(r(\gamma)) |n\sqrt{d' \circ \alpha_n}(\gamma)|^2
    = d(r(\gamma)) \, \big|f_n(\gamma) \sqrt{d'(\alpha_n(s(\gamma)))}\big|^2
    = |f_n(\gamma)|^2 > 0.\qedhere
\end{align*}
\end{proof}

\begin{lem}\label{lem:bisecton alpha}
Let $\G$ be a second-countable locally compact Hausdorff \'etale groupoid. Let $D :=
C_0(\go) \subseteq C^*(\G)$. Suppose that $n \in C_c(\G)$ is supported on a bisection.
Then $n \in \Nn_{C^*(\G)}(D)$, and $\alpha_n(s(\gamma)) = r(\gamma)$ for $\gamma \in \osupp(n)$. If
$c : \G \to \Gamma$ is a grading, then $C^*_r(G) = \clsp \Nn_\star(D)$.
\end{lem}
\begin{proof}
That $n \in \Nn_{C^*(\G)}(D)$ and $\alpha_n(s(\gamma)) = r(\gamma)$ for $\gamma \in \osupp(n)$
follow from \cite[Proposition~4.8]{Ren2008}. The Stone--Weierstrass theorem gives
\[
C_c(c^{-1}(g)) = \lsp\{f \in C_c(c^{-1}(g)) : \supp(f)\text{ is contained in a
bisection}\}
\]
for $g\in\Gamma$. Since $C_c(\G) = \lsp\bigcup_{g \in \Gamma}
C_c(c^{-1}(g))$,
it follows that $C^*_r(\G) = \clsp \Nn_\star(C_0(\go))$.
\end{proof}

\begin{proof}[Proof of Proposition~\ref{prp:isomorphism}]
Since $c$ is continuous, $\Gamma$ is discrete, and $\G$ is \'etale, there exists a
bisection $U \subseteq c^{-1}(c(\gamma))$ containing $\Gamma$, and then an element $n \in
C_c(U) \subseteq C_c(\G)$ with $n(\gamma) = 1$. Fix $n,m \in C_c(\G)$ supported on
bisections contained in $c^{-1}(c(\gamma))$ with $n(\gamma) = m(\gamma) = 1$. Choose an
open $\gamma \in U \subseteq \osupp(m) \cap \osupp(n)$. Lemma~\ref{lem:bisecton alpha}
shows that $\alpha_n = \alpha_m$ on $s(U)$. Let $u := s(\gamma)$.  For $d \in C_0(s(U))$
with $d(u) = 1$, we have $\osupp(nd), \osupp(md) \subseteq U$, so $(d n^*m
d)|_{\Io{c^{-1}(e)}_u}$ is just the point-mass $\delta_u$. Hence $U_{n^*m} = \pi_u(d n^*m
d) = 1_{C^*(\Io{G}_u)}$. Since $\alpha_{nd} = \alpha_n = \alpha_m = \alpha_{md}$ on
$\osupp(d) \subseteq U$, we have $[n,u] = [m,u]$. So $\theta$ is well-defined.

To see that $\theta$ is injective, fix $\gamma \not= \eta \in \G$. Choose $n, m$
supported on open bisections containing $\gamma$ and $\eta$ respectively, so
$\theta(\gamma) = [n,\widehat{s(\gamma)}]$ and $\theta(\eta) = [m, \widehat{s(\eta)}]$.
If $s(\gamma) \not= s(\eta)$, then clearly $\theta(\gamma) \not= \theta(\eta)$, so
suppose that $s(\gamma) = s(\eta) =: u$. If $r(\gamma) \not= r(\eta)$, then
$\alpha_n(\hat{u}) = r(\gamma) \not= r(\eta) = \alpha_m(\hat{u})$ by
Lemma~\ref{lem:bisecton alpha}, so again $\theta(\gamma) \not= \theta(\eta)$. So suppose
that $r(\gamma) = r(\eta) =: v$. Fix $d \in C_0(\go)$ with $\hat{u}(d) = 1$, and put $w =
d n^*m d$. Then $w|_{\Io{c^{-1}(e)}_u} = \delta_{\gamma^{-1}\eta} \in
C_c(\Io{c^{-1}(\id_\Gamma)}_u)$. Hence $U_{n^*m} = \pi_u(w) = U_{\gamma^{-1}\eta} \in
C^*(\Io{c^{-1}(\id_\Gamma)}_u)$. Since $\Io{c^{-1}(\id_\Gamma)}_u$ is a discrete torsion-free abelian
group, \cite[Theorem~8.57]{HofmannMorris} gives $\Io{c^{-1}(\id_\Gamma)}_u \cong
\Uu(C^*(\Io{c^{-1}(\id_\Gamma)}_u))/\Uu_0(C^*(\Io{c^{-1}(\id_\Gamma)}_u))$ via the map $\gamma \mapsto
U_\gamma\, \Uu_0(C^*(\Io{c^{-1}(\id_\Gamma)}_u))$; so $u_{\gamma^{-1}\eta} \not\in
\Uu_0(C^*(\Io{c^{-1}(\id_\Gamma)}_u))$, and $\theta(\gamma) \not= \theta(\eta)$.

To see that $\theta$ is surjective, fix $[n,\hat{u}] \in \Hh(C^*_r(\G), C_0(\go),
\delta_c)$, and let $\hat{v} := \alpha_n(\hat{u})$. Regard $n$ as an element of $C_0(\G)$
using~\eqref{eq:Jean's jmap}. Since $0 < n^*n(u) = nn^*(\alpha_n(u)) = nn^*(v) =
\sum_{\gamma \in \G^v} |n(\gamma)|^2$, we have $n(\gamma) \not=0$ for some $\gamma \in
G^v$. So Lemma~\ref{lem:fn<->alphan} gives $\alpha_n(\widehat{s(\gamma)}) =
\widehat{r(\gamma)} = \hat{v} = \alpha_n(\hat{u})$. Since $\alpha_n$ is bijective,
$s(\gamma) = u$. Fix an open bisection $B \owns \gamma$ with $f_n$ nonzero on $B$. Fix $m
\in C_0(B)$ with $m$ identically $1$ on a neighbourhood of $\gamma$.
Lemma~\ref{lem:fn<->alphan} gives $\alpha_n = \alpha_m$ on $s(\osupp(m))$, so
Lemma~\ref{lem:alpha-n} yields $\alpha_{n^*m} = \id$ on a neighbourhood of $\hat{u}$.
Thus $n^*m$ is supported on $\Io{c^{-1}(\id_\Gamma)}$ by Lemma~\ref{lem:fn<->alphan}. By
\cite[Theorem~8.57]{HofmannMorris}, $\Io{c^{-1}(\id_\Gamma)}_u \cong
\Uu(C^*(\Io{c^{-1}(\id_\Gamma)}_u))/\Uu_0(C^*(\Io{c^{-1}(\id_\Gamma)}_u))$, so $U_{n^*m} \sim_h U_\eta$ for
some $\eta \in \Io{c^{-1}(\id_\Gamma)}_u$. Fix a bisection neighbourhood $W$ of $\eta^{-1}$ and $h
\in C_c(W)$ with $h(\eta^{-1}) = 1$. Then $\pi_u(w_{n^*mh}) = \pi_u(w_{n^*m})\pi_u(h) =
\pi_u(w_{n^*m}) \delta_{\eta^{-1}} \mathbin{\sim_h} 1_{C^*(\Io{c^{-1}(\id_\Gamma)}_u)}$. Since
$\alpha_{h} = \id = \alpha_{n^*m}$, we have $\alpha_n = \alpha_{mh}$ on a neighbourhood
of $u$, so $[n,\hat{u}] = [mh,\hat{u}] = \theta(\gamma\eta^{-1})$.

To see that $\theta$ is open, recall that $\G$ is a normal space, so has a basis of open
bisections $U$ with closure contained in precompact open bisections $V$. For such $U,V$,
fix $n \in C_c(V)$ with $n|_U = 1$. Then $\theta(\gamma) = [n, \widehat{s(\gamma)}]$ for
every $\gamma \in U$, so $\theta(U) = Z(n, U)$ is open.

Finally, to see that $\theta$ is continuous, fix $n \in N_g(D)$ and an open set $U
\subseteq \osupp(n^*n)$. Fix $u \in U$. Since $\theta$ is surjective, $[n,u] =
\theta(\gamma) = [f,u]$ for some $f \in C_c(G)$ supported on a bisection $B\subseteq
c^{-1}(g)$ containing $\gamma$. So Lemma~\ref{lem:classes agree} yields an open
neighbourhood $V$ of $u$ such that $[n,v]=[f,v]$ for all $v\in V$. Then $\gamma \in BV
\subseteq \theta^{-1}(Z(n,U))$.
\end{proof}

\begin{proof}[Proof of Theorem~\ref{thm:1}]
Statement~(1) is clear. For~(2) let $h : \G_2^{(0)} \to \G_1^{(0)}$ be induced by
$\phi|_{C_0(\go_1)}$, so $f(h(x)) = \phi(f)(x)$ for $f \in C_0(\go_1)$. The formula
$\phi^*([n,x]) := [\phi^{-1}(n), h(x)]$ defines a graded isomorphism $\phi^* :
\Hh(C^*_r(\G_2), C_0(\G_2^{(0)}), \delta_{c_2}) \to \Hh(C^*_r(\G_1), C_0(\G_1^{(0)}),
\delta_{c_1})$. Proposition~\ref{prp:isomorphism} gives graded isomorphisms $\theta_i :
\G_i \to \Hh(C^*_r(\G_i), C_0(\G_i^{(0)}), \delta_{c_i})$. So $\kappa := \theta_1^{-1}
\circ \phi^* \circ \theta_2 : \G_2 \to \G_1$ is a graded isomorphism. For $x \in
\G_2^{(0)}$ and $a \in C_0(\G_2^{(0)})$ with $a(x) > 0$,
Proposition~\ref{prp:isomorphism} gives $\kappa(x) = \theta_1^{-1} \circ \phi^*([a, x]) =
\theta^{-1}_1(a \circ h^{-1}, h(x)) = h(x)$.
\end{proof}

\section{Group actions}\label{sec:group actions}
In this section we consider transformation groupoids of continuous actions of countable
discrete groups on topological spaces. We characterise isomorphism of transformation
groupoids in terms of continuous orbit equivalence of the actions. If in each group, the
subgroup $\{\gamma : \text{there is an open $U \subseteq X$ that $\gamma$ fixes
pointwise}\}$ is torsion-free and abelian, Theorem~\ref{thm:1} yields a generalisation of
\cite[Theorem 1.2]{Li}.

Fix a countable discrete group $\Gamma$ acting on the right of a second-countable locally
compact Hausdorff space $X$. We write $x\gamma$ for the action of $\gamma$ on $x$. Define
\[
X \rtimes \Gamma := X\times\Gamma
\]
under the product topology, let $(X \rtimes \Gamma)^{(2)} := \{\big((x_1, \gamma_1),
(x_2, \gamma_2)\big) : x_2 = x_1\gamma_1\}$, and define $(x_1,\gamma_1)(x_1\gamma_1,
\gamma_2) = (x_1, \gamma_1\gamma_2)$, and $(x, \gamma)^{-1} = (x\gamma, \gamma^{-1})$.
Then $X \rtimes \Gamma$ is a second-countable locally compact Hausdorff \'etale groupoid.
Its unit space is $X \times \{\id_\Gamma\}$, which we identify with $X$, so $r(x,\gamma) = x$ and
$s(x,\gamma) = x\gamma$. We have $C^*_r(X \rtimes \Gamma) \cong C_0(X) \rtimes_r \Gamma$
via an isomorphism that carries $C_0 ( (X \rtimes \Gamma)^{(0)})$ to $C_0 (X) \subseteq
C_0 ( X ) \rtimes_r \Gamma$ (see \cite[Example~3.2.8]{SimsNotes}).

For $x\in X$, we write
\[
\Stab(x) := \{\gamma\in\Gamma : x\gamma=x\}
\]
for the \emph{stabiliser subgroup} of $x$ in $\Gamma$; observe that then $(X \rtimes
\Gamma)^x_x = \{x\} \times \Stab(x)$. We also consider the \emph{essential stabiliser
subgroup}
\[
\Stab^{\ess}(x) := \{\gamma \in \Gamma : \gamma \in \Stab(y)\text{ for all $y$ in some neighbourhood $U$ of $x$}\}.
\]
Observe that $\Io{X\rtimes\Gamma} = \bigcup_{x\in X}\{x\}\times \Stab^{\ess}(x)$. We say
$(X,\Gamma)$ is \emph{topologically free} if each $\Stab^{\ess}(x) = \{\id_\Gamma\}$; a
Baire-category argument shows that $(X, \Gamma)$ is topologically free if and only if
$\overline{\{x \in X : \Stab(x) = \{\id_\Gamma\}\}} = X$.

\begin{dfn}\label{dfn:coe groups}
Let $\Gamma \curvearrowright X$ and $\Lambda \curvearrowright Y$ be actions of countable
discrete groups on locally compact Hausdorff spaces. A \emph{continuous orbit
equivalence} $(h, \pi, \eta)$ from $(X,\Gamma)$ to $(Y,\Lambda)$ consists of a
homeomorphism $h : X \to Y$ and continuous maps $\phi : X\times \Gamma \to \Lambda$ and
$\eta : Y\times\Lambda \to \Gamma$ such that $h(x\gamma) = h(x)\phi(x,\gamma)$ for all
$x,\gamma$ and $h^{-1}(y\lambda)=h^{-1}(y)\eta(y,\lambda)$ for all $y, \lambda$. We call
$h$ the \emph{underlying homeomorphism} of $(h,\phi,\eta)$.
\end{dfn}

For topologically free systems, the intertwining condition appearing in
Definition~\ref{dfn:coe groups} has some important consequences.

\begin{lem}\label{lem:coe consequences}
Let $\Gamma \curvearrowright X$ and $\Lambda \curvearrowright Y$ be topologically free
actions of countable discrete groups on locally compact Hausdorff spaces. Let $(h, \phi,
\eta)$ be a continuous orbit equivalence from $(X, \Gamma)$ to $(Y, \Lambda)$. Fix $x \in
X$. We have
\[
    \phi(x, \gamma\gamma') = \phi(x,\gamma)\phi(x\gamma, \gamma')\quad\text{for all $\gamma,\gamma' \in \Gamma$,}
\]
and each $\theta_x := \phi(x, \cdot) : \Gamma \to \Lambda$ is a bijection that carries
$\id_\Gamma$ to $\id_\Lambda$ and restricts to bijections $\Stab(x) \to \Stab(h(x))$ and
$\Stab^{\ess}(x) \to \Stab^{\ess}(h(x))$.
\end{lem}
\begin{proof}
The first part is proved in the same way as \cite[Lemma~2.8]{Li}, the only difference is
that the actions considered in \cite{Li} are left actions and here we consider right
actions.

For the second, by \cite[Corollary~2.11]{Li}, for all $x \in X$ with $\Stab(x)$ trivial,
$\theta_x$ is a bijection.  For arbitrary $x \in X$, take $x_n \to x$ such that $\Stab
(x_n)$ is trivial.  Since $\phi$ is continuous and $\Lambda$ is a discrete group,
$\theta_{x_n} = \theta_x$ for large $n$.  So $\theta_x$ is a bijection.  That $\theta_x
(\id_\Gamma ) = \id_\Lambda$ follows from the first statement. The intertwining condition in
Definition~\ref{dfn:coe groups} and that $h$ is a homeomorphism gives $\theta_x(\Stab(x))
= \Stab(h(x))$; and $\theta_x(\Stab^{\ess}(x)) = \Stab^{\ess}(h(x))$ because both are
trivial.
\end{proof}

Lemma~\ref{lem:coe consequences} prompts the following definition.

\begin{dfn}
Let $\Gamma \curvearrowright X$ and $\Lambda \curvearrowright$ be actions of countable
discrete groups on locally compact Hausdorff spaces. Consider a map $\phi : X \times
\Gamma \to \Lambda$.
\begin{enumerate}
\item We call $\phi$ a \emph{cocycle} if $\phi(x, \gamma\gamma') =
    \phi(x,\gamma)\phi(x\gamma, \gamma')$ for all $x, \gamma, \gamma'$.
\item Let $h : X \to Y$ be a homeomorphism. We say that $(h,\phi)$ \emph{preserves
    stabilisers} if $\phi(x,\cdot)$ restricts to bijections $\Stab(x) \to
    \Stab(h(x))$, and that $(h,\phi)$ \emph{preserves essential stabilisers} if
    $\phi(x, \cdot)$ restricts to bijections $\Stab^{\ess}(x) \to
    \Stab^{\ess}(h(x))$.
\end{enumerate}
\end{dfn}

\begin{prp} \label{prp:group action}
Let $\Gamma \curvearrowright X$ and $\Lambda \curvearrowright Y$ be actions of countable
discrete groups on locally compact Hausdorff spaces. Suppose that $h : X \to Y$ is a
homeomorphism and $\phi : X\times \Gamma \to \Lambda$ is continuous. The following are
equivalent:
\begin{enumerate}
    \item there is an isomorphism $\Theta : X \rtimes \Gamma \to Y \rtimes \Lambda$
        such that $\Theta(x,\id_\Gamma)=(h(x),\id_\Lambda)$ and $\Theta(x,\gamma) =
        (h(x),\phi(x,\gamma))$ for all $x\in X$ and $\gamma\in\Gamma$;
    \item $\phi$ is a cocycle, $(h,\phi)$ preserves stabilisers, and there is a map
        $\eta : Y \times \Lambda \to \Gamma$ such that $(h,\phi,\eta)$ is a
        continuous orbit equivalence; and
    \item $\phi$ is a cocycle, $(h,\phi)$ preserves essential stabilisers, and there
        is a map $\eta : Y \times \Lambda \to \Gamma$ such that $(h,\phi,\eta)$ is a
        continuous orbit equivalence.
\end{enumerate}
\end{prp}
\begin{proof}
(1)$\,\implies\,$(2): Define $\eta : Y \times \Lambda \to \Gamma$ by $\Theta^{-1}(y,
\lambda) = (h^{-1}(y), \eta(y,\lambda))$. Then $(h,\phi,\eta)$ is a continuous orbit
equivalence and $\phi$, $\eta$ are cocycles. The pair $(h,\phi)$ preserves stabilisers
because $\Theta\big(\{(x,\gamma):\gamma\in\Gamma_x\}\big) =
\{(h(x),\lambda):\lambda\in\Lambda_{h(x)}\}$.

(2)$\,\implies\,$(3): It suffices to show that $\phi(x,\gamma) \in \Stab^{\ess}(h(x))$ if
and only if $\gamma \in \Stab^{\ess}(x)$. First suppose that $\gamma\in\Stab^{\ess}(x)$.
Fix an open neighbourhood $U \owns x$ such that $x'\gamma = x'$ for all $x'\in U$. Since
$\phi$ is continuous and $\Lambda$ is discrete, we can assume that $\phi(x', \gamma) =
\phi(x,\gamma) =: \lambda$ for all $x' \in U$. Then $U':=h(U)$ is an open neighbourhood
of $h(x)$ such that for $u \in U'$,
\[
y\lambda = h(h^{-1}(y)) \lambda = h(h^{-1}(y)) \phi(h^{-1}(y), \gamma)
    = h(h^{-1}(y)\gamma) = h(h^{-1}(y)) = y.
\]
Hence $\phi(x,\gamma) \in \Stab^{\ess}(h(x))$.

Now suppose that $\phi(x,\gamma) \in \Stab^{\ess}(h(x))$. There is an open $U' \owns
h(x)$ such that $\lambda := \phi(x,\gamma)$ satisfies $y\lambda=y$ for all $y\in U'$. By
continuity of $h$ and $\phi$, there is an open $U \owns x$ such that $h(U)\subseteq U'$
and $\phi(x', \gamma) = \lambda$ for all $x' \in U$. So $x'\gamma = x'$ for all $x'\in
U$, giving $\gamma \in \Stab^{\ess}(x)$.

$(3)\implies (1)$: Define $\Theta : X \rtimes \Gamma \to Y \rtimes \Lambda$ by
$\Theta(x,\gamma) = (h(x),\phi(x,\gamma))$. Then $\Theta$ is open, continuous
homomorphism. To see that $\Theta$ is injective, suppose that $\Theta(x_1,\gamma_1) =
\Theta(x_2,\gamma_2)$. Then $x_1 = x_2$, $\gamma_1\gamma_2^{-1}\in\Gamma_{x_1}$, and
$\Theta(x_1,\gamma_1\gamma_2^{-1}) \in (Y \rtimes \Lambda)^{(0)}$. Since $Y \rtimes
\Lambda$ is \'etale, $(Y \rtimes \Lambda)^{(0)}$ is open in $Y \rtimes \Lambda$. Hence
$(Y \rtimes \Lambda)^{(0)} \subseteq \Io{Y \rtimes \Lambda}$. Since
$\Theta^{-1}(\operatorname{Iso}(Y \rtimes \Lambda)) \subseteq \operatorname{Iso}(X
\rtimes \Gamma)$ and $\Theta$ is continuous, $\Theta^{-1}(\Io{Y \rtimes \Lambda})
\subseteq \Io{X \rtimes \Gamma}$. Hence $(x_1,\gamma_1\gamma_2^{-1})\in \Io{X \rtimes
\Gamma}$. Since $\phi(x_1, \cdot) : \Stab^{\ess}(x_1) \to \Stab^{\ess}(h(x_1))$ is
bijective, $\gamma_1 = \gamma_2$, so $(x_1,\gamma_1) = (x_2,\gamma_2)$.

For surjectivity, fix $(y,\lambda) \in Y \rtimes \Lambda$. Then $y\phi(h^{-1}(y),
\eta(y,\lambda)) = h(h^{-1}(y)\eta(y,\lambda)) = h( h^{-1} ( y \lambda ) ) =  y\lambda$,
so $\lambda\phi(h^{-1}(y), \eta(y,\lambda))^{-1} \in \Stab(y)$. By continuity of $\phi$
and $\eta$ there is an open neighbourhood $U \owns y$ such that $y' \phi(h^{-1}(y'),
\eta(y',\lambda)) = h(h^{-1}(y') \eta(y', \lambda)) = y'\lambda$ for all $y' \in U$. For
$y' \in U$ we have $\lambda'\phi(h^{-1}(y'), \eta(y',\lambda'))^{-1} \in \Stab(y')$.
Hence $\lambda\phi(h^{-1}(y), \eta(y,\lambda))^{-1} \in \Stab^{\ess}(y)$. Since
$\phi(h^{-1}(y), \cdot) : \Stab^{\ess}(h^{-1}(y)) \to \Stab^{\ess}(y)$ is bijective,
there exists $\gamma\in \Stab^{\ess}(h^{-1}(y))$ such that $\phi(h^{-1}(y),\gamma) =
\lambda\phi(h^{-1}(y), \eta(y,\lambda))^{-1}$. Thus $\Theta(h^{-1}(y), \gamma
\eta(y,\lambda)) = (y,\lambda)$.

Finally, we have $\phi(x,\id_\Gamma)=\phi(x,\id_\Gamma^2)=\phi(x,\id_\Gamma)^2$ from which it follows that $\phi(x,\id_\Gamma)=\id_\Lambda$ and thus $\Theta(x,\id_\Gamma)=(h(x),\id_\Lambda)$.
\end{proof}

This gives the following generalisation of Li's rigidity theorem \cite[Theorem 1.2]{Li}.

\begin{cor}
Let $\Gamma \curvearrowright X$ and $\Lambda \curvearrowright Y$ be actions of countable
discrete groups on second-countable locally compact Hausdorff spaces. Suppose that
$\Stab^{\ess}(x)$ and $\Stab^{\ess}(y)$ are torsion-free and abelian for all $x\in X$ and
all $y\in Y$, and that $h$ is a homeomorphism from $X$ to $Y$. The following are
equivalent:
\begin{enumerate}
    \item there exist cocycles $\phi : X \times \Gamma \to \Lambda$ and $\eta : Y
        \times \Lambda \to \Gamma$ such that $(h, \phi, \eta)$ is a continuous orbit
        equivalence from $(X,\Gamma)$ to $(Y,\Lambda)$ and $(h, \phi)$ and $(h^{-1},
        \eta)$ preserve essential stabilisers;
    \item there is an isomorphism $\Theta : X \rtimes \Gamma \to Y \rtimes \Lambda$
        such that $\Theta(x,\id_\Gamma)=(h(x),\id_\Lambda)$ for all $x\in X$; and
    \item there is an isomorphism $\phi:C_0(X)\times_r\Gamma\to
        C_0(Y)\times_r\Lambda$ such that $\phi(C_0(X))=C_0(Y)$ and $\phi(f)=f\circ
        h^{-1}$ for $f\in C_0(X)$.
\end{enumerate}
\end{cor}
\begin{proof}
This follows directly from Theorem~\ref{thm:1} and Proposition~\ref{prp:group action}.
\end{proof}

\begin{rmk}\label{rmk:KOQstuff}
Let $\Gamma \curvearrowright X$ and $\Gamma \curvearrowright Y$ be actions of a countable
discrete group on second-countable locally compact Hausdorff spaces, and consider the
reduced crossed-products $C_0(X) \times_r \Gamma$ and $C_0(Y) \times_r \Gamma$, with dual
coactions $\delta_X$ and $\delta_Y$ of $\Gamma$. As $(C_0(X) \times_r \Gamma)^{\delta_X}
= C_0(X)$ and likewise for $Y$, any equivariant isomorphism $C_0(X) \times_r \Gamma \to
C_0(Y) \times_r \Gamma$ restricts to an isomorphism $C_0(X) \to C_0(Y)$.
Theorem~\ref{thm:1} therefore shows that $C_0(X) \times_r \Gamma$ and $C_0(Y) \times_r
\Gamma$ are equivariantly isomorphic if and only if $X \rtimes \Gamma \cong Y \rtimes
\Gamma$ as graded groupoids. We have $X \rtimes \Gamma \cong Y \rtimes \Gamma$ as graded
groupoids if and only if there is a homeomorphism $h : X \to Y$ such that $h(x)\gamma =
h(x\gamma)$ for all $x \in X$ and $\gamma \in \Gamma$, so if and only if $(X, \Gamma)$
and $(Y, \Gamma)$ are topologically conjugate. So we recover \cite[Proposition~4.3]{KOQ}.
\end{rmk}

\section{Local homeomorphisms of locally compact Hausdorff spaces}\label{sec:Deaconu-Renault}

In this section we adapt the ideas of Section~\ref{sec:group actions} to actions of $\NN$
by local homeomorphisms. We first characterise an appropriate notion of continuous orbit
equivalence in terms of diagonal-preserving isomorphism of $C^*$-algebras. We then
characterise eventual conjugacy in terms of gauge-equivariant diagonal-preserving
isomorphisms.

A \emph{Deaconu--Renault system} is a pair $(X, \sigma)$ consisting of a locally compact
Hausdorff space $X$ and a local homeomorphism $\sigma : \dom(\sigma) \to \ran(\sigma)$
from an open set $\dom(\sigma) \subseteq X$ to an open set $\ran(\sigma) \subseteq X$.
Inductively define $\dom(\sigma^n) := \sigma^{-1}(\dom(\sigma^{n-1}))$, so each $\sigma^n
: \dom(\sigma^n) \to \ran(\sigma^n)$ is a local homeomorphism and $\sigma^m \circ
\sigma^n = \sigma^{m+n}$ on $\dom(\sigma^{m+n})$. We write $D_n := \dom(\sigma^n)$ and
$\sigma^0 := \id_X$. For $x\in X$ we define the \emph{stabiliser group} at $x$ by
\begin{equation*}
\Stab(x) := \{m-n:m,n\in\NN,\ x\in D_m\cap D_n,\text{ and }\sigma^n(x)=\sigma^m(x)\} \subseteq \ZZ,
\end{equation*}
and we define the \emph{essential stabiliser group} at $x$ by
\begin{multline*}
\Stab^{\ess}(x) := \{m-n:m,n\in\NN\text{ and there is an open neighbourhood }\\
    U\subseteq D_m\cap D_n\text{ of }x\text{ such that }\sigma^n|_U = \sigma^m|_U\} \subseteq
    \Stab(x).
\end{multline*}
With the convention that $\min(\emptyset) = \infty$, we define the \emph{minimal
stabiliser} of $x$ to be
\[
\Stab_{\min}(x) := \min\{n \in \Stab(x) : n \ge 1\},
\]
and the \emph{minimal essential stabiliser} at $x$ to be
\[
    \Stab^{\ess}_{\min}(x) := \min\{n \in \Stab^{\ess}(x) : n \ge 1\}.
\]

The Deaconu--Renault groupoid of $(X,\sigma)$ is
\[
\G = \G(X, \sigma) = \bigcup_{n,m \in \NN} \big\{(x, n-m , y) \in D_n \times\{n-m\} \times D_m : \sigma^n(x) = \sigma^m(y)\big\},
\]
under the topology with basic open sets $Z(U, n, m, V) := \{(x, n-m, y) : x \in U,\ y \in
V, \text{ and }\sigma^n(x) = \sigma^m(y)\}$ indexed by quadruples $(U, n, m, V)$ where
$n,m \in \NN$, $U \subseteq D_n$ and $V \subseteq D_m$ are open, and $\sigma^n|_U$ and
$\sigma^m|_V$ are homeomorphisms. Each $Z(U, n, m, V)$ can be written as $Z(U', n, m,
V')$ with $\sigma^n(U') = \sigma^m(V')$ (put $U' = U \cap (\sigma^n)^{-1}(V)$ and $V' = V
\cap (\sigma^m)^{-1}(U)$). This $\G$ is a locally compact Hausdorff \'etale groupoid,
with $\G^{(0)} = \{(x,0,x) : x \in X\}$ identified with $X$. It is also amenable by the
argument of \cite[Lemma~3.5]{SWiv}, so $C^*_r(\G) = C^*(\G)$. The isotropy subgroupoid of
$\G$ is $\{(x,n,x):x\in X,\ n\in\Stab(x)\}$, and the interior of the isotropy is
$\Io{\G}=\{(x,n,x):x\in X,\ n\in\Stab^{\ess}(x)\}$. So $\Io{\G}$ is torsion-free and
abelian. Taking $\Gamma = \{e\}$ and $c : \G \to \Gamma$ the trivial cocycle, we obtain a
(trivially) graded groupoid $\G$.

\subsection{Continuous orbit equivalence}
We show that stabiliser-preserving continuous orbit equivalence of Deaconu--Renault
systems characterises isomorphism of their groupoids.

\begin{dfn}
Let $(X, \sigma)$ and $(Y, \tau)$ be Deaconu--Renault systems. We say that $(X, \sigma)$
and $(Y, \tau)$ are \emph{continuous orbit equivalent} if there exist a homeomorphism $h
: X \to Y$  and continuous maps $k,l : \dom(\sigma) \to \NN$ and $k', l' : \dom(\tau) \to
\NN$ such that
\[
\tau^{l(x)}(h(x)) = \tau^{k(x)}(h(\sigma(x)))
    \quad\text{ and }\quad
\sigma^{l'(y)}(h^{-1}(y)) =
\sigma^{k'(y)}(h^{-1}(\tau(y)))
\]
for all $x, y$. We call $(h,l,k,l',k')$ a continuous orbit equivalence and we call $h$
the \emph{underlying homeomorphism}. We say that $(h,l,k,l',k')$ \emph{preserves
stabilisers} if $\Stab_{\min}(h(x))  < \infty \iff \Stab_{\min}(x) < \infty$, and
\begin{align*}
&\bigg|\sum_{n=0}^{\Stab_{\min}(x)-1} l(\sigma^n(x)) - k(\sigma^n(x)) \bigg| =
\Stab_{\min}(h(x))
\text{ and}\\
&\bigg|\sum_{n=0}^{\Stab_{\min}(y)-1} l'(\tau^n(y)) - k'(\tau^n(y)) \bigg| =
\Stab_{\min}(h^{-1}(y))
\end{align*}
whenever $\Stab(x), \Stab(y)$ are nontrivial, $\sigma^{\Stab_{\min}(x)}(x)=x$, and
$\tau^{\Stab_{\min}(y)}(y)=y$.

Likewise, we say that $(h,l,k,l',k')$ \emph{preserves essential stabilisers} if
$\Stab^{\ess}_{\min}(h(x)) < \infty\iff \Stab^{\ess}_{\min}(x) < \infty$, and
\begin{align*}
&\bigg|\sum_{n=0}^{\Stab^{\ess}_{\min}(x)-1} \left( l(\sigma^n(x)) - k(\sigma^n(x)) \right) \bigg| =
\Stab^{\ess}_{\min}(h(x))
\text{ and }\\
&\bigg|\sum_{n=0}^{\Stab^{\ess}_{\min}(y)-1}\left( l'(\tau^n(y)) - k'(\tau^n(y)) \right) \bigg| =
\Stab^{\ess}_{\min}(h^{-1}(y))
\end{align*}
whenever $\Stab^{\ess}_{\min}(x), \Stab^{\ess}_{\min}(y) < \infty$,
$\sigma^{\Stab^{\ess}_{\min}(x)}(x)=x$, and $\tau^{\Stab^{\ess}_{\min}(y)}(y)=y$.
\end{dfn}

Our definition of continuous orbit equivalence boils down to the usual notion for
homeomorphisms (see for instance \cite{BT}, \cite{GPS}, and \cite{Tomiyama}), and to
orbit equivalence of graphs if $(X, \sigma)$ and $(Y, \tau)$ are the shifts on their
boundary-path spaces (see \cite{BCW}). Our main theorem in this section generalises
\cite[Theorem~5.1]{BCW}, \cite[Theorem~2.3]{MM}, and \cite[Theorem~2]{Tomiyama}:

\begin{thm}\label{thm:DR oe<->gi}
Let $(X, \sigma)$ and $(Y, \tau)$ be Deaconu--Renault systems with $X, Y$ second
countable, and suppose that $h : X \to Y$ is a homeomorphism. Then the following are
equivalent:
\begin{enumerate}
\item\label{it:DR oe} there is a stabiliser-preserving continuous orbit equivalence
    from $(X, \sigma)$ to $(Y, \tau)$ with underlying homeomorphism $h$;
\item\label{it:DR G iso} there is a groupoid isomorphism $\Theta : \G(X, \sigma) \to
    \G(Y, \tau)$ such that $\Theta|_X = h$; and
\item\label{it:DR Cstar iso} there is an isomorphism $\phi : C^*(\G(X, \sigma)) \to
    C^*(\G(Y, \tau))$ such that $\phi(C_0(X)) = C_0(Y)$ with $\phi(f) = f\circ
    h^{-1}$ for $f \in C_0(Y)$.
\end{enumerate}
\end{thm}

To prove Theorem~\ref{thm:DR oe<->gi}, we need to relate isomorphism of Deaconu--Renault
groupoids to continuous orbit equivalence. Arklint, Eilers, and Ruiz \cite{AER} (see also
\cite{CW}) proved that isomorphism of graph groupoids (and hence diagonal-preserving
isomorphism of graph $C^*$-algebras) is characterised by continuous orbit equivalence of
the shift maps on their boundary path spaces with underlying homeomorphism $h$ satisfying
$\Stab^{\ess}(h(x)) = \{0\} \iff \Stab^{\ess}(x) = \{0\}$. The following is the analogous
result for Deaconu--Renault systems.

\begin{prp}\label{prp:DR}
Let $(X, \sigma)$ and $(Y, \tau)$ be Deaconu--Renault systems. Let $h : X\to Y$ be a
homeomorphism and let $l,k:\dom(\sigma)\to\NN$ and $l',k':\dom(\tau)\to\NN$ be continuous
maps such that $h(x)\in\dom(\tau^{l(x)})$ and $h(\sigma(x))\in\dom(\tau^{k(x)})$ for
$x\in\dom(\sigma)$, and $h^{-1}(y)\in\dom(\sigma^{l'(y)})$ and
$h^{-1}(\tau(y))\in\dom(\sigma^{k'(y)})$ for $y\in\dom(\tau)$. The following are
equivalent:
\begin{enumerate}
    \item there are groupoid isomorphisms $\Theta:G(X,\sigma)\to G(Y,\tau)$ and $\Theta':G(Y,\tau)\to G(X,\sigma)$ such that $\Theta|_X=h$, $\Theta'|_Y=h^{-1}$, $\Theta(x,1,\sigma(x))=(h(x),l(x)-k(x),h(\sigma(x)))$ for $x\in\dom(\sigma)$, and $\Theta'(y,1,\tau(y))=(h^{-1}(y),l'(y)-k'(y),h^{-1}(\tau(y)))$ for $y\in\dom(\tau)$;
    \item\label{it:DR esp} $(h,l,k,l',k')$ is an essential-stabiliser-preserving
    continuous orbit equivalence; and
    \item\label{it:DR sp} $(h,l,k,l',k')$ is a stabiliser-preserving continuous orbit equivalence.
\end{enumerate}
\end{prp}

Note: (\ref{it:DR esp})$\;\implies\;$(\ref{it:DR sp}) of Proposition~\ref{prp:DR} shows
that if $\sigma$ and $\tau$ are topologically free, then every continuous orbit
equivalence from $(X, \sigma)$ to $(Y, \tau)$ preserves stabilisers.

The rest of this subsection deals with the proofs of Proposition~\ref{prp:DR} and
Theorem~\ref{thm:DR oe<->gi}.

\begin{lem}\label{lem:lmap}
Let $(X, \sigma)$ be a Deaconu--Renault system. The function $l_X : \G(X, \sigma) \to
\NN$ given by $l_X(x, n, y) := \min\{l \in \NN : l \ge n\text{ and } \sigma^l(x) =
\sigma^{l-n}(y)\}$ is continuous.
\end{lem}
\begin{proof}
Suppose that $(x_i, n_i, y_i) \to (x,n,y) \in \G(X, \sigma)$. Then $n_i = n$ for large
$i$, so we can assume that $n_i = n$ for all $i$. We first show that $l_X(x_i, n, y_i)
\le l_X(x,n,y)$ for large $i$. To see this, fix a basic open neighbourhood $Z(U, p, p-n,
V)$ of $(x,n,y)$; so $\sigma^p(U) = \sigma^{p-n}(V)$, and $\sigma^p|_U$ and
$\sigma^{p-n}|_V$ are homeomorphisms. Since $(x_i, n, y_i) \to (x,n,y)$ we have $(x_i, n,
y_i) \in Z(U, p, p-n, V)$ for large $i$. So $\sigma^p(x_i) = \sigma^{p-n}(y_i)$ for large
$i$. Let $l := l_X(x,n,y)$. Then $l \le p$, say $p = l + q$. Hence
$\sigma^q(\sigma^l(x_i))
    = \sigma^p(x_i)
    = \sigma^{p-n}(y_i)
    = \sigma^q(\sigma^{l-n}(y_i))$
for large $i$. Since $\sigma^q$ is locally injective and $\lim_i \sigma^l(x_i) =
\sigma^l(x) = \sigma^{l-n}(y) = \lim_i \sigma^{l-n}(y_i)$, we deduce that $\sigma^l(x_i)
= \sigma^{l-n}(y_i)$ for large $i$. So $l_X(x_i, n, y_i) \le l$ for large $i$.

It now suffices to show that $l_X(x_i, n, y_i) \ge l$ for large $i$; equivalently, if
$l_X(x_i, n, y_i) = k$ for infinitely many $i$, then $k \ge l$. Suppose that $l(x_{i_j},
n, y_{i_j}) = k$ for all $j$. Then $k \ge n$ by definition of $l_X(x_i, n, y_i)$. Since
$\sigma^k(x_{i_j}) = \sigma^{k-n}(y_{i_j})$ for all $j$ and $x_{i_j} \to x$ and $y_{i_j}
\to y$, continuity forces $\sigma^k(x) = \sigma^{k-n}(y)$. So $k \ge l$.
\end{proof}

Given a Deaconu--Renault system $(X, \sigma)$, define $c_X : \G(X, \sigma) \to \ZZ$ by
$c_X(x, n, y) := n$.

\begin{lem}\label{lem:DR oe}
Let $(X, \sigma)$ and $(Y, \tau)$ be Deaconu--Renault systems. Suppose that $\Theta :
\G(X, \sigma) \to \G(Y, \tau)$ is an isomorphism of groupoids. Let $h : X \to Y$ be the
restriction of $\Theta$ to $\G(X, \sigma)^{(0)}$. For $p \in \NN$, the functions $l_p,
k_p : D_p \to \NN$ given by
\begin{align*}
l_p(x) &:= \min\{l \in \NN : \tau^l(h(x)) = \tau^{l - c_Y(\Theta(x, p,
    \sigma^p(x)))}(h(\sigma^p(x)))\},
        \quad\text{ and}\\
k_p(x) &:= l_p(x) - c_Y(\Theta(x, p, \sigma^p(x)))
\end{align*}
are continuous, and $\tau^{l_p(x)}(h(x)) = \tau^{k_p(x)}(h(\sigma^p(x)))$ for all $x \in
D_p$. For $p \in \NN$ and $x \in D_p$,
\[
\sum_{n=0}^{p-1}(l_1(\sigma^n(x)) - k_1(\sigma^n(x)))=l_p(x) - k_p(x) = c_Y(\Theta(x, p, \sigma^p(x))).
\]
\end{lem}
\begin{proof}
Since $l_p(x)=l_Y(\Theta(x,p,\sigma^p(x)))$ and $\Theta$ is continuous, $l_p$ is
continuous by Lemma~\ref{lem:lmap}. Now $k_p$ is continuous because $l_p$ and $c_Y$ are.
We have $\tau^{l_p(x)}(h(x)) = \tau^{k_p(x)}(h(\sigma^p(x)))$ and $l_p(x) - k_p(x) =
c_Y(\Theta(x, p, \sigma^p(x)))$ by definition of $k_p$. So $\sum_{n=0}^{p-1} \left(
l_1(\sigma^n(x)) - k_1(\sigma^n(x)) \right) =\sum_{n=0}^{p-1}c_Y(\Theta(\sigma^n(x), 1,
\sigma^{n+1}(x)))=c_Y(\Theta(x, p, \sigma^p(x)))$.
\end{proof}

\begin{lem}\label{lem:k,l on I}
Let $(X, \sigma)$ and $(Y, \tau)$ be Deaconu--Renault systems as above. Suppose that
$\Theta : \G(X, \sigma) \to \G(Y, \tau)$ is an isomorphism. Let $h, k_p, l_p$ be as in
Lemma~\ref{lem:DR oe} and let $x\in X$. Then $\Stab^{\ess}_{\min}(x) < \infty$ if and
only if $\Stab^{\ess}_{\min}(h(x)) < \infty$, and if $\Stab^{\ess}_{\min}(x) < \infty$
and $\sigma^{\Stab^{\ess}_{\min}(x)}(x) = x$, then $|l_{\Stab^{\ess}_{\min}(x)}(x) -
k_{\Stab^{\ess}_{\min}(x)}(x)| = \Stab^{\ess}_{\min}(h(x))$.
\end{lem}
\begin{proof}
For any Deaconu--Renault system $(Z, \beta)$ and $a \in Z$, $\Stab^{\ess}_{\min}(a) =
\infty$ if and only if $\Io{\G(Z, \beta)}_a = \{(a,0,a)\}$.  Thus $\Stab^{\ess}_{\min}(x)
= \infty \iff \Stab^{\ess}_{\min}(h(x)) = \infty$ as $\Theta$ restricts to an isomorphism
$\Io{\G(X, \sigma)} \to \Io{\G(Y, \tau)}$ mapping $\G(X, \sigma)^{(0)}$ onto $\G(Y,
\tau)^{(0)}$.

Suppose that $\Stab^{\ess}_{\min}(h(x)) < \infty$. Then $\Theta(x,
\Stab^{\ess}_{\min}(x), x) = (h(x), q, h(x))$ for some $q \in \ZZ$. Since $(x,
\Stab^{\ess}_{\min}(x), x)$ generates $\Io{\G(X, \sigma)}_x$, we deduce that $(h(x), q,
h(x))$ generates $\Io{\G(Y, \tau)}_{h(x)}$. Thus $q = \pm \Stab^{\ess}_{\min}(h(x))$. So
if $\sigma^{\Stab^{\ess}_{\min}(x)}(x)=x$, then
\[
|l_{\Stab^{\ess}_{\min}(x)}(x) - k_{\Stab^{\ess}_{\min}(x)}(x)|
    = |c_Y(\Theta(x,\Stab^{\ess}_{\min}(x),x))|
    = \Stab^{\ess}_{\min}(h(x)).\qedhere
\]
\end{proof}

Given a Deaconu--Renault system $(X, \sigma)$ and $l : \dom(\sigma) \to \NN$, we
inductively define $l_m : \dom(\sigma^m) \to \NN$, $m \ge 1$ by $l_1 = l$ and $l_{m+1}(x)
= l(x) + l_m(\sigma(x))$. For $m,n \ge 1$ we have
\begin{equation}\label{eq:lp formula}
l_m(x) = \sum^{m-1}_{i=0} l(\sigma^i(x))\quad\text{ and }\quad l_{m+n}(x) = l_m(x) + l_n(\sigma^m(x)).
\end{equation}

\begin{lem}\label{lem:oplus}
Let $(X, \sigma)$ and $(Y, \tau)$ be Deaconu--Renault systems and let $(h,l,k,l',k')$ be
a continuous orbit equivalence from $(X, \sigma)$ to $(Y, \tau)$. Then there is a
continuous cocycle $c_{(h,l,k,l',k')} : \G(X,\sigma) \to \ZZ$ such that
$c_{(h,l,k,l',k')}(x,m-n,x')= l_m (x) - k_m ( x ) - l_n (x') + k_n (x')$.
\end{lem}
\begin{proof}
Suppose that $x\in D_{m+1}$, $x'\in D_{n+1}$, and $\sigma^m(x)=\sigma^n(x')$. A
computation shows that $l_m (x) - k_m ( x ) - l_n (x') + k_n (x') = l_{m+1} (x) - k_{m+
1} ( x ) - l_{n+1} (x') + k_{n+1} (x')$.  Therefore, $c_{(h,l,k,l',k')} : \G(X,\sigma)
\to \ZZ$ is a well-defined map.  It is easy to check that this map is a cocycle. For
continuity, suppose that $\sigma^m(x)=\sigma^n(x')$. Fix open subneighbourhoods $U \owns
x$ and $V \owns x'$ of $D_m$ and $D_n$ such that $\sigma^m|_U$ and $\sigma^n|_V$ are
homeomorphisms, $l,k,\dots,l\circ\sigma^m,k\circ\sigma^m$ are constant on $U$, and
$l,k,\dots,l\circ\sigma^n,k\circ\sigma^n$ are constant on $V$. Then $c_{(h,l,k,l',k')}$
is constant on $Z(U,m,n,V)$.
\end{proof}

\begin{lem}\label{lem:Theta hom}
Let $(X, \sigma)$ and $(Y, \tau)$ be Deaconu--Renault systems, and let $(h, k, l, k',
l')$ be a continuous orbit equivalence from $(X, \sigma)$ to $(Y, \tau)$. Then there is a
continuous groupoid homomorphism $\Theta_{k,l} : \G(X, \sigma) \to \G(Y, \tau)$ such that
$\Theta_{k,l}(x, m-n, x') = (h(x), l_m(x) - k_m(x) - l_n(x') + k_n(x'), h(x'))$ whenever
$\sigma^m(x) = \sigma^n(x')$. For each $x \in X$ there is a group homomorphism $\pi_x :
\Stab(x) \to \Stab(h(x))$ such that
\begin{equation}\label{eq:pi}
\pi_x(m-n)=l_m(x) - k_m(x) - l_n(x) + k_n(x)\text{ whenever } \sigma^m(x)=\sigma^n(x).
\end{equation}
For $x \in X$ and $m,n \in \NN$, we have $\Stab(\sigma^m(x)) = \Stab(\sigma^n(x))$,
$\Stab(h(\sigma^m(x))) = \Stab(h(\sigma^n(x)))$ and $\pi_{\sigma^m(x)} =
\pi_{\sigma^n(x)}$.
\end{lem}
\begin{proof}
Lemma~\ref{lem:oplus} yields a continuous homomorphism $\Theta_{k,l} : \G(X, \sigma) \to
\G(Y, \tau)$ such that $\Theta_{k,l}(x, m-n, x') = (h(x), l_m(x) - k_m(x) - l_n(x') +
k_n(x'), h(x'))$ whenever $\sigma^m(x) = \sigma^n(x')$.

For $x \in X$ the map $\pi_x : \Stab(x) \to \Stab(h(x))$ defined by $(h(x), \pi_x(p),
h(x)) = \Theta(x, p, x)$ is a homomorphism satisfying~\eqref{eq:pi}. That each $\Stab(x)
= \Stab(\sigma(x))$ follows from the definition of $\Stab$, and then induction gives
$\Stab(\sigma^m(x)) = \Stab(\sigma^n(x))$ for all $x$. Since $h$ intertwines
$\sigma$-orbits and $\tau$-orbits, it follows immediately that $\Stab(h(\sigma^m(x))) =
\Stab(h(\sigma^n(x)))$ for all $x$. For the final statement, let $p := l_m(\sigma^n(x)) -
k_m(\sigma^n(x)) - l_n(\sigma^m(x)) + k_n(\sigma^m(x))$ and calculate:
\begin{align*}
(h&(\sigma^m(x)), \pi_{\sigma^m(x)}(q), h(\sigma^m(x)))
    = \Theta(\sigma^m(x), q, \sigma^m(x)) \\
    &= \Theta(\sigma^m(x), n-m, \sigma^n(x))\Theta(\sigma^n(x), q, \sigma^n(x))\Theta(\sigma^n(x), m-n, \sigma^m(x))\\
    &= (h(\sigma^m(x)), -p, h(\sigma^n(x))) (h( \sigma^n(x) ) , \pi_{\sigma^n(x)}(q), h( \sigma^n(x)) )  (h(\sigma^m(x)), p, h(\sigma^n(x)))\\
    &= (h(\sigma^m(x)), \pi_{\sigma^n(x)}(q), h(\sigma^m(x))).\qedhere
\end{align*}
\end{proof}

\begin{proof}[Proof of Proposition~\ref{prp:DR}]
$(1)\implies (2)$: Let $k_p,l_p$ be as in Lemma~\ref{lem:DR oe}. Then $l - k = l_1 - k_1$
on $\dom(\sigma)$. Likewise, if $k'_p,l'_p : \dom(\tau^p)\to\NN$ are the functions
obtained from Lemma~\ref{lem:DR oe} for $\Theta^{-1}$, then $l'-k'=l'_1-k'_1$ on
$\dom(\tau)$. So Lemmas \ref{lem:DR oe}~and~\ref{lem:k,l on I} show that $(h,l,k,l',k')$
is an essential-stabiliser-preserving continuous orbit equivalence.

For both (2)$\,\implies\,$(3) and (3)$\,\implies\,$(1), fix a continuous orbit
equivalence $(h, k, l, k', l')$ from $(X, \sigma)$ to $(Y, \tau)$. Let
\[
\Theta : \G(X, \sigma) \to \G(Y, \tau),\quad\text{ and }\quad
\Theta' : \G(Y, \tau) \to \G(X, \sigma)
\]
be the homomorphisms of Lemma~\ref{lem:Theta hom} for $(h, k, l, k', l')$ and for
$(h^{-1}, k', l', k, l)$ respectively, and for each $x \in X$, let
\[
\pi_x : \Stab(x) \to \Stab(h(x))
\]
be the homomorphism~\eqref{eq:pi}.

(2)$\,\implies\,$(3). Using that, by Lemma~\ref{lem:Theta hom}, $\Stab(\cdot)$,
$\Stab(h(\cdot))$ and $x \mapsto \pi_x$ are constant on orbits, and that $(h,l,k,l',k')$
preserves essential stabilisers, it is easy to check that $\pi_x(\Stab^{\ess}(x)) =
\Stab^{\ess}(h(x))$ for all $x$. Fix $x\in X$ and $n\in\Stab(h(x))$. Since $\Theta \circ
\Theta'$ is continuous and $\Theta(\Theta'(h(x),n,h(x)))=(h(x),m,h(x))$ for some
$m\in\Stab(h(x))$, there exist $p,q\in\NN$ with $p-q=n$, and open neighbourhoods $U,V$ of
$h(x)$ such that $\tau^p|_U$ and $\tau^q|_V$ are homeomorphisms, $\tau^p(U)=\tau^q(V)$,
and $\Theta(\Theta'(y,n,y'))=(y,m,y')$ for $y\in U$, $y'\in V$, with
$\tau^p(y)=\tau^q(y')$. So $(y,n-m,y)=(y,n,y')(y,m,y')^{-1}\in \G(Y, \tau)$ for all $y\in
U$, giving $n-m\in\Stab^{\ess}(h(x))$. Hence $\pi_x(r) = n-m$ for some
$r\in\Stab^{\ess}(x)$. Thus $\pi_x(r+s)=n$ where $s=c_X(\Theta'(h(x),n,h(x)))$. So
$\Stab_{\min}(h(x))<\infty \implies \Stab_{\min}(x)<\infty$, and symmetry gives the
reverse implication.

Suppose that $\Stab_{\min}(x) < \infty$ and that $\sigma^{\Stab_{\min}(x)}(x)=x$. Since
$\Stab_{\min}(h(x))$ generates $\Stab(h(x))$ and $\Stab_{\min}(x)$ generates $\Stab(x)$,
we have $l_{\Stab_{\min}(x)}(x) - k_{\Stab_{\min}(x)}(x) = \pi_x(\Stab_{\min}(x)) = \pm
\Stab_{\min}(h(x))$. Hence $(h,l,k,l',k')$ preserves stabilisers.

$(3)\implies (1)$: Let $\Theta:=\Theta_{k,l}:G(X,\sigma)\to G(Y,\tau)$ and $\Theta':=\Theta_{k',l'}:G(Y,\tau)\to G(X,\sigma)$ be as in Lemma~\ref{lem:Theta hom}. Then $\Theta$ and $\Theta'$ are continuous groupoid homomorphisms. We show that $\Theta$ is bijective and $\Theta^{-1}$ is continuous.
For injectivity, suppose $\Theta(x_1,n_1,x_1')=\Theta(x_2,n_2,x'_2)$. As $h$ is a
homeomorphism, $x_1=x_2$ and $x_1'=x_2'$. So $\Theta(x_1,n_1-n_2,x_1) =
\Theta(x_1,n_1,x_1')\Theta(x_1,n_2,x_1')^{-1} = (h(x_1),0,h(x_1))$, giving $\pi_x(n_1 -
n_2) = 0$. As $(h, k, l, k', l')$ preserves stabilisers and $\Stab(\cdot)$,
$\Stab(h(\cdot))$ and $x \mapsto \pi_x$ are constant on orbits, each $\pi_x : \Stab(x)
\to \Stab(h(x))$ is bijective. Thus $(x_1,n_1,x_1')=(x_2,n_2,x'_2)$.

For surjectivity, fix $(y,n,y')\in \G(Y,\tau)$. We have $\Theta(\Theta'(y,n,y')) =
(y,m,y')$ for some $m\in\ZZ$, so $n-m \in \Stab(y)$. Since $\pi_{h^{-1}(y)}$ is
bijective, $n-m = \pi_{h^{-1}(y)}(p)$ for some $p \in \Stab(h^{-1}(y))$. So
$\Theta\big(h^{-1}(y), p + c_X(\Theta'(y,n,y')), h^{-1}(y')\big)=(y,n,y')$.

To see that $\Theta^{-1}$ is continuous, suppose $(y_n,m_n,y'_n) \to (y,m,y')$ in
$\G(Y,\tau)$. Fix $p,q\in\NN$ and open $U \owns h^{-1}(y)$ and $V \owns h^{-1}(y')$ such
that $\sigma^p|_U$ and $\sigma^q|_V$ are homeomorphisms, $\sigma^p(h^{-1}(y)) =
\sigma^q(h^{-1}(y'))$, and $\Theta^{-1}(y,m,y') = (h^{-1}(y),p-q,h^{-1}(y'))$. Choose
open subneighbourhoods $U' \owns h^{-1}(y)$ and $V' \owns h^{-1}(y')$ of $U, V$ such that
$\Theta(x,p-q,x') = (h(x),m,h(x'))$ whenever $x\in U'$, $x'\in V'$, and $\sigma^p(x) =
\sigma^q(x')$. Fix $N$ such that $y_n\in h(U')$, $y'_n\in h(V')$, and $m_n=m$ for $n\ge
N$. Then $\Theta^{-1}(y_n,m_n,y'_n)\in Z(U',p,q,V')\subseteq Z(U,p,q,V)$ for $n\ge N$. So
$\Theta^{-1}(y_n,m_n,y'_n) \to \Theta^{-1}(y,m,y')$.

We thus have that $\Theta$ is an isomorphism. A similar argument shows that $\Theta'$ is also an isomorphism.
\end{proof}

\begin{proof}[Proof of Theorem~\ref{thm:DR oe<->gi}]
The equivalence \mbox{(\ref{it:DR oe})\;$\iff$\;(\ref{it:DR G iso})} follows from
Proposition~\ref{prp:DR} and Lemma~\ref{lem:DR oe}, and the equivalence \mbox{(\ref{it:DR
G iso})\;$\iff$\;(\ref{it:DR Cstar iso})} follows from Theorem~\ref{thm:1}.
\end{proof}

\subsection{Eventual conjugacy}

Here we generalise \cite[Theorem 4.1]{CR} by showing that the isomorphism of
$C^*$-algebras in Theorem~\ref{thm:DR oe<->gi} is gauge-equivariant if and only if the
groupoid isomorphism is cocycle-preserving, which is if and only if the continuous orbit
equivalence is an eventual conjugacy.

\begin{dfn}
Let $(X, \sigma)$ and $(Y, \tau)$ be Deaconu--Renault systems. We say that $(X, \sigma)$
and $(Y, \tau)$ are \emph{eventually conjugate} if there is a stabiliser-preserving
continuous orbit equivalence $(h,l,k,l',k')$ from $(X, \sigma)$ to $(Y, \tau)$ such that
$l(x) = k(x) + 1$ for all $x \in X$.
\end{dfn}

Given $(X, \sigma)$, there is an action $\gamma^X : \TT \to \Aut(C^*(\G(X, \sigma)))$
such that $\gamma^X_z(f)(x, n, x') = z^n f(x, n, x')$ for all $z \in \TT$, $(x, n, x')
\in \G(X, \sigma)$ and $f \in C_c(\G(X, \sigma))$.

\begin{thm}\label{thm:ev conj equivalences}
Let $(X, \sigma)$ and $(Y, \tau)$ be Deaconu--Renault systems and let $h : X \to Y$ be a
homeomorphism. Then the following are equivalent:
\begin{enumerate}
\item\label{it:DR evconj} there is an eventual conjugacy from $(X, \sigma)$ to $(Y,
    \tau)$ with underlying homeomorphism $h$;
\item\label{it:DR G c-iso} there is an isomorphism $\Theta : \G(X, \sigma) \to \G(Y,
    \tau)$ such that $\Theta|_X = h$ and $c_X = c_Y \circ \Theta$; and
\item\label{it:DR graded iso} there is an isomorphism $\phi : C^*(\G(X, \sigma)) \to
    C^*(\G(Y, \tau))$ such that $\phi(C_0(X)) = C_0(Y)$, with $\phi(f) = f\circ
    h^{-1}$ for $f \in C_0(X)$, and $\phi \circ \gamma^X_z = \gamma^Y_z \circ \phi$.
\end{enumerate}
\end{thm}
\begin{proof}
Theorem~\ref{thm:1} applied to the $\ZZ$-coactions dual to $\gamma^X$ and $\gamma^Y$
gives \mbox{(\ref{it:DR G c-iso})\,$\iff$\,(\ref{it:DR graded iso})}.

\mbox{(\ref{it:DR G c-iso})\,$\implies$\,(\ref{it:DR evconj})}. If $c_X = c_Y \circ
\Theta$, then the formula for $k_1$ in Lemma~\ref{lem:DR oe} gives $l_1(x) - k_1(x) = 1$,
so the continuous orbit equivalence constructed in the proof of \mbox{(\ref{it:DR G
iso})\;$\implies$\;(\ref{it:DR oe})} in Theorem~\ref{thm:DR oe<->gi} is an eventual
conjugacy.

\mbox{(\ref{it:DR evconj})\,$\implies$\,(\ref{it:DR G c-iso})}. Suppose that
$(h,l,k,l',k')$ is an eventual conjugacy from $(X, \sigma)$ to $(Y, \tau)$. The
formula~\eqref{eq:lp formula} gives $l_p(x) - k_p(x) = p$ for all $x \in X$. Thus,
$\Theta_{k,l}$ of Lemma~\ref{lem:Theta hom} satisfies $c_Y \circ \Theta_{k,l} = c_X$.  As
in the proof of \mbox{(\ref{it:DR oe})\;$\implies$\;(\ref{it:DR G iso})} in
Theorem~\ref{thm:DR oe<->gi}, $\Theta_{k,l}$ an isomorphism.
\end{proof}

\section{Homeomorphisms of compact Hausdorff spaces}\label{sec:homeomorphisms}
We now specialise to the case where $X$ is a compact Hausdorff space and $\sigma:X\to X$
is a homeomorphism. We combine the results and techniques developed in Sections
\ref{sec:group actions}~and~\ref{sec:Deaconu-Renault} to obtain a generalisation of Boyle
and Tomiyama's theorem \cite[Theorem 3.6]{BT}. If $\sigma : X \to X$ is a homeomorphism, then $\alpha :
\G(X,\sigma) \to X \rtimes \ZZ $, $\alpha(x,n,y)\mapsto (x,n)$ is an isomorphism, so
induced an isomorphism $\phi : C^*(\G(X,\sigma)) \cong C(X)\times_\sigma\ZZ$ with $\phi (
C(\G(X, \sigma)^{(0)}) ) = C(X)$.

Using Theorem~\ref{thm:DR oe<->gi}, we can prove the following generalisation of
\cite[Theorem 3.6]{BT} (and thus of \cite[Theorem 2.4]{GPS} and
\cite[Corollary]{Tomiyama}). Following \cite{BT}, we say that homeomorphisms $\sigma : X
\to X$ and $\tau : Y \to Y$ are \emph{flip conjugate} if there is a homeomorphism $h : X
\to Y$ such that either $h \circ \sigma = \tau \circ h$ or $h \circ \sigma =
\tau^{-1}\circ h$.

\begin{thm}\label{thm:bt}
Suppose that $\sigma : X \to X$ and $\tau : Y \to Y$ are homeomorphisms of
second-countable compact Hausdorff spaces. The following are equivalent:
\begin{enumerate}
    \item $\G(X,\sigma)$ and $\G(Y,\tau)$ are isomorphic;
    \item $C(X)\times_\sigma\ZZ \cong C(Y)\times_\tau\ZZ$ via an isomorphism that
        maps $C(X)$ to $C(Y)$; and
    \item there exist decompositions $X=X_1\sqcup X_2$ and $Y=Y_1\sqcup Y_2$ into
        disjoint open invariant sets such that $\sigma|X_1$ is conjugate to
        $\tau|Y_1$ and $\sigma|X_2$ is conjugate to $\tau^{-1}|Y_2$.
\end{enumerate}
If $\sigma$ and $\tau$ are topologically transitive or $X$ and $Y$ are connected, then
these conditions hold if and only $\sigma$ and $\tau$ are flip-conjugate.
\end{thm}

Our proof of $(1)\implies(3)$ closely follows \cite{BT}, and requires some preliminary
results. Take $X$, $Y$, $\sigma$, $\tau$ as in Theorem~\ref{thm:bt}, and an isomorphism
$\theta:\G(X,\sigma)\to \G(Y,\tau)$. Define $h : X \to Y$ by
$\theta(x,0,x)=(h(x),0,h(x))$. Let $c_X : G(X, \sigma) \to \ZZ$ and $c_Y : G(Y, \tau) \to
\ZZ$ be the canonical cocycles. Define $f : X \times \ZZ \to \ZZ$ by
$f(n,x):=c_Y(\theta(x,n,\sigma^n(x)))$. Then
\begin{equation}\label{eq:fstuff}
f(m+n,x) = f(m,x) + f(n,\sigma^m(x))\quad\text{ for all $m,n\in\ZZ$ and $x\in X$.}
\end{equation}
For $x\in X$, $f(\cdot,x)$ is a bijection of $\ZZ$ with inverse $n\mapsto
c_X(\theta^{-1}(h(x),n,\tau^n(h(x))))$.

For $m,n\in\ZZ$, we let $[m,n]:=\{k\in\ZZ:m\le k\le n\}$.

\begin{lem}\label{lem:M}
For each $M \in \NN$ there exists $\overline{M} \in \NN$ such that
\begin{equation}\label{eq:Mproperty}
[-M,M]\subseteq \{f(n,x):n\in [-\overline{M},\overline{M}]\}\quad\text{ for all $x \in X$.}
\end{equation}
\end{lem}
\begin{proof}
Let $x \in X$.  Fix $M_x \in \NN$ such that $[-M,M]\subseteq \{f(n,x):n\in [-M_x,M_x]\}$.
Continuity of the map $f(n,\cdot)$ for each $n$ implies $x$ has an open neighbourhood
$U_x$ such that $[-M,M]\subseteq \{f(n,x'):n\in [-M_x,M_x]\}$ for $x'\in U_x$.
Compactness of $X$ gives a finite $F \subseteq X$ such that $\bigcup_{x\in F}U_x=X$. So
$\overline{M}:=\max\{M_x:x\in F\}$ satisfies~\eqref{eq:Mproperty}.
\end{proof}

\begin{lem}\label{lem:N}
There is a positive integer $N$ such that
\begin{align*}
X_1 &:= \{x\in X : f(n,x)>0\text{ and }f(-n,x) < 0\text{ for }n> N\}\quad\text{ and}\\
X_2 &:= \{x\in X : f(n,x)<0\text{ and }f(-n,x) > 0\text{ for }n> N\}
\end{align*}
are clopen $\sigma$-invariant subsets such that $X=X_1\sqcup X_2$.
\end{lem}
\begin{proof}
Compactness of $X$ implies that $M:=\max\{|f(1, x)|:x\in X\}$ is finite.
Lemma~\ref{lem:M} gives $N \in \NN$ with $[-M,M]\subseteq \{f(n,x):n\in [-N,N]\}$ for all
$x$.

To see that $X = X_1 \cup X_2$, fix $x\in X$. Choose $m > N$. Since $n\mapsto f(n,x)$ is
bijective and $[-M,M]\subseteq \{f(n,x):n\in [-N,N]\}$, we have $|f(m,x)|>M$. Since
$|f(m+1,x)-f(m,x)|=|f(1,\sigma^m(x))|\le M$, it follows that $f(m+1,x)$ and $f(m,x)$ have
the same sign. Similarly, $f(-m-1,x)$ and $f(-m,x)$ have the same sign. Since $n\mapsto
f(n,x)$ is bijective, $x \in X_1 \cup X_2$. Hence $X=X_1\cup X_2$. Continuity of $f(n,
\cdot)$ imply $X_i$ are clopen.

We show that $\sigma (X_1) \subseteq X_1$ (a similar argument gives $\sigma(X_2)
\subseteq X_2$).  Fix $x\in X_1$. Choose $m > N$. Then $f(m+1, x) > M$. Since $|f(m+1,x)
- f(m,\sigma(x))| = |f(1,x)| \le M$, we have $\sigma(x) \not\in X_2$; so $\sigma(x) \in
X_1$.
\end{proof}

\begin{lem}\label{lem:ab}
There are continuous functions $a:X_1\to\ZZ$ and $b:X_2\to\ZZ$ such that
$f(1,x)=a(x)-a(\sigma(x))+1$ for $x\in X_1$, and $f(1,x) = b(x)-b(\sigma(x)) - 1$ for
$x\in X_2$.
\end{lem}
\begin{proof}
Define $n(x) := f(1, x)$ for all $x \in X$. Fix $x\in X_1$. Since $f(n,x)>0$ and
$f(-n,x)<0$ for $n>N$, both $|\{n \ge 0 : f(n,x) < 0\}|$ and $|\{m < 0 : f(m,x) \ge 0\}|$
are finite. Let
\begin{equation*}
a(x) := \big|\{m<0:f(m,x)\ge 0\}\big| - \big|\{n\ge 0:f(n,x)<0\}\big|.
\end{equation*}
Continuity of $f(n,\cdot)$ implies that $a$ is continuous.

Take $x\in X_1$. Using~\eqref{eq:fstuff} at the third equality, we calculate
\begin{align*}
a(x) + 1
    &= \big|\{m<0:f(m,x)\ge 0\}\big| - \big|\{n\ge 0:f(n,x)<0\}\big| + 1\\
    &=\big|\{m<1:f(m,x)\ge 0\}\big| - \big|\{n\ge 1:f(n,x)<0\}\big|\\
    &=\big|\{m<1:f(m-1,\sigma(x))\ge - n(x)\}\big|
        - \big|\{n\ge 1:f(n-1,\sigma(x))<- n(x)\}\big|\\
    &=\big|\{m<0:f(m,\sigma(x))\ge - n(x)\}\big| - \big|\{n\ge 0:f(n,\sigma(x))<- n(x)\}\big|.
\end{align*}
Suppose now that $n(x)\ge 0$. Then
\begin{align*}
\big|\{m<0:f(m,\sigma(x))\ge -n(x)\}\big| - \big|&\{m<0 : f(m,\sigma(x))\ge 0\}\big|\\
    &= \big|\{m<0:0>f(m,\sigma(x))\ge -n(x)\}\big|,\text{ and}\\
\big|\{n\ge 0:f(n,\sigma(x))< 0\}\big| - \big|&\{n\ge 0: f(n,\sigma(x))< -n(x)\}\big| \\
    &= \big|\{n\ge 0:0>f(n,\sigma(x))\ge -n(x)\}\big|.
\end{align*}
Since $\{f(m,\sigma(x)):m\in\ZZ\}=\ZZ$, we have
\begin{equation*}
\big|\{m<0:0>f(m,\sigma(x))\ge -n(x)\}\big| + \big|\{n\ge 0:0>f(n,\sigma(x))\ge -n(x)\}\big| = n(x).
\end{equation*}
Hence
\begin{align*}
a(x) &= \big|\{m<0:f(m,\sigma(x))\ge -n(x)\}\big| - \big|\{n\ge 0:f(n,\sigma(x))<-n(x)\}\big| - 1\\
    &= \big|\{m<0:f(m,\sigma(x))\ge 0\}\big| - \big|\{n\ge 0:f(n,\sigma(x))<0\}\big| + n(x) - 1\\
    &= a(\sigma(x)) + n(x) - 1,
\end{align*}
so $n(x) = a(x) - a(\sigma(x)) + 1$. A similar argument applies for $n(x) < 0$.

Similarly, $b(x) := \big|\{m<0:f(m,x)\le 0\}\big| - \big|\{n\ge 0:f(n,x)>0\}\big|$
defines a continuous function such that $n(x)=b(x)-b(\sigma(x))=1$ for $x\in X_2$.
\end{proof}

\begin{proof}[Proof of Theorem~\ref{thm:bt}]
The equivalence of (1) and (2) follows directly from Theorem~\ref{thm:DR oe<->gi}.

$(3)\implies (1)$: Suppose $h_1:X_1\to Y_1$ and $h_2:X_2\to Y_2$ are homeomorphisms such
that $h_1(\sigma(x))=\tau(h_1(x))$ for $x\in X_1$ and $h_2(\sigma(y))=\tau^{-1}(h_2(y))$
for $y\in X_2$. Then
\begin{equation*}
\theta(x,n,y)=
\begin{cases}
(h_1(x),n,h_1(y))&\text{if }x,y \in X_1,\\
(h_2(x),-n,h_2(y))&\text{if }x,y \in X_2,
\end{cases}
\end{equation*}
defines an isomorphism $\theta:\G(X,\sigma)\to \G(Y,\tau)$.

$(1)\implies (3)$: Let $X_1$ and $X_2$ be as in Lemma~\ref{lem:N} and $a$ and $b$ be as
in Lemma~\ref{lem:ab}. Let $Y_1:=h(X_1)$ and $Y_2:=h(X_2)$. Define $h_1:X_1\to Y_1$ by
$h_1(x)=\tau^{a(x)}(h(x))$ and $h_2:X_2\to Y_2$ by $h_2(x)=\tau^{b(x)}(h(x))$. Then
$h_1(\sigma(x))=\tau^{a(\sigma(x))}(h(\sigma(x)))=\tau^{a(x)-f(1,x)+1}(h(\sigma(x)))=\tau^{a(x)+1}(h(x))=\tau(h_1(x))$
for $x\in X_1$ because $\tau^{f(1,x)}(h(x))=h(\sigma(x))$, and
$h_2(\sigma(x))=\tau^{b(\sigma(x))}(h(\sigma(x)))=\tau^{b(x)-f(1,x)-1}(h(\sigma(x)))=\tau^{b(x)-1}(h(x))=\tau^{-1}(h_1(x))$
for $x\in X_2$ because $\tau^{f(1,x)}(h(x))=h(\sigma(x))$.
\end{proof}

\section{Equivariant Morita equivalence}\label{sec:Eq Me}

In this section we define equivalence of graded groupoids and equivariant Morita
equivalence of nested pairs of $C^*$-algebras. We show that given coactions $\delta_i :
A_i \to A_i \otimes C^*_r(\Gamma)$ and abelian $C^*$-subalgebras $D_i \subseteq A_i$ with $D_i$ containing an approximate unit for $A_i$, the
pairs $(A_i, D_i)$ are equivariantly Morita equivalent if and only if their extended Weyl
groupoids are graded equivalent.

Groupoids $\G_1$ and $\G_2$ are \emph{equivalent} if there is a topological space $Z$
carrying commuting free and proper actions of $\G_1$ and $\G_2$ on the left and right
respectively such that $r : Z \to \go_1$ and $s : Z \to \go_2$ induce homeomorphisms
$\G_1\backslash Z \cong \go_2$ and $Z/\G_2 \cong \go_1$. The associated \emph{linking
groupoid} $L(\G_1, \G_2)$ \cite{SimWil} is
\[
L := L(\G_1, \G_2) = \G_1 \sqcup Z \sqcup Z^{\op} \sqcup \G_2
\]
with $L^{(0)} = \go_1 \cup \go_2$, the obvious range and source maps, and multiplication
determined by multiplication in $\G_1$ and $\G_2$, the actions of the $\G_i$ on $Z$ and
$Z^{\op}$, and the maps ${_{\G_1}[\cdot, \cdot]} : \{(z, y^{\op}) \in Z \times Z^{\op} :
s(z) = r(y^{\op})\} \to \G_1$ and ${[\cdot, \cdot]_{\G_2}} : \{(y^{\op}, z) \in Z^{\op}
\times Z : s(y^{\op}) = r(z)\} \to \G_2$ determined by ${_{\G_1}[z,y^{\op}]} \cdot y = z$
and $y \cdot [y^{\op}, z]_{\G_2} = z$. Conversely, if $\G$ is a groupoid and $K_1, K_2$
are complementary $\G$-full open subsets of $\go$, then $Z := K_1 L K_2$ is a $K_1 \G
K_1$--$K_2 \G K_2$-equivalence under the actions given by multiplication in $\G$.

\begin{dfn}\label{dfn:equivalence}
Let $c_i : \G_i \to \Gamma$, $i = 1,2$ be gradings of locally compact groupoids. A
\emph{graded} $(\G_1, c_1)$--$(\G_2, c_2)$-equivalence consists of a
$\G_1$--$\G_2$-equivalence $Z$ and a continuous map $c_Z : Z \to \Gamma$ satisfying
$c_Z(\gamma \cdot z \cdot \eta) = c_1(\gamma)c_Z(z)c_2(\eta)$ for all $\gamma, z, \eta$.
\end{dfn}

\begin{lem}\label{lem:equiv equiv}
Let $\Gamma$ be a discrete group. Then graded equivalence as described in
Definition~\ref{dfn:equivalence} is an equivalence relation on $\Gamma$-graded groupoids.
\end{lem}
\begin{proof}
Suppose that $(\G_i, c_i)$ is a $\Gamma$-graded groupoid for $i = 1,2,3$ and that $(Z_i,
c_{Z_i})$ is a $(\G_i, c_i)$--$(\G_{i+1}, c_{i+1})$-equivalence for $i = 1,2$. Define
$\sim$ on $Z_1 \mathbin{_r\times_s} Z_2 := \{(z_1, z_2) \in Z_1 \times Z_2 : s(z_1) =
r(z_2)\}$ by $(z_1 \cdot \gamma, \gamma^{-1}\cdot z_2) \sim (z_1, z_2)$ for $\gamma \in
(\G_2)^{s(z_1)}$. By \cite[page~6]{MRW}, $Z_1 *_{\G_2} Z_2 := (Z_1 \mathbin{_r\times_s}
Z_2)/{\sim}$ is a $\G_1$--$\G_3$-equivalence with $\gamma_1\cdot[z_1, z_2] =
[\gamma_1\cdot z_1, z_2]$ and $[z_1, z_2]\cdot \gamma_3 = [z_1, z_2\cdot\gamma_3]$. For
$[z_1, z_2] \in Z_1 *_{\G_2} Z_2$ and $\gamma \in \G_2^{s(z_1)}$, we have
\[
c_{Z_1}(z_1 \cdot \gamma)c_{Z_2}(\gamma^{-1} \cdot z_2)
    = c_{Z_1}(z_1) c_2(\gamma)c_2(\gamma^{-1})c_{Z_2}(z_2)
    = c_{Z_1}(z_1) c_{Z_2}(z_2),
\]
so there is a map $\tilde{c} : Z_1 *_{\G_2} Z_2 \to \Gamma$ such that $\tilde{c}([z_1,
z_2]) = c_{Z_1}(z_1) c_{Z_2}(z_2)$. So for $\gamma_1 \in \G_1$, $[z_1, z_2] \in Z_1
*_{\G_2} Z_2$, and $\gamma_3 \in \G_3$ with $s(\gamma_1) = r([z_1, z_2])$ and $s([z_1,
z_2]) = r(\gamma_3)$, we have $\tilde{c}(\gamma_1 \cdot [z_1, z_2] \cdot \gamma_3) =
c_{Z_1}(\gamma_1 \cdot z_1) c_{Z_2}(z_2 \cdot \gamma_3) = c_1(\gamma_1)
\tilde{c}([z_1,z_2]) c_3(\gamma_3)$. So $(Z_1 *_{\G_2} Z_3, \tilde{c})$ is a $(\G_1,
c_1)$--$(\G_3, c_3)$-equivalence.
\end{proof}

\begin{lem}\label{lem:graded equiv}
Let $\Gamma$ be a discrete group and let $(\G_1, c_1)$, $(\G_2, c_2)$ be $\Gamma$-graded
locally compact Hausdorff \'etale groupoids with each $\Io{c_i^{-1}(\id_\Gamma)}$ torsion-free and
abelian. Suppose that $(Z, c_Z)$ is a graded equivalence from $(\G_1, c_1)$ to $(\G_2,
c_2)$. Let $\G = L(\G_1, \G_2)$, and define $c : \G \to \Gamma$ by $c|_{G_i} = c_i$,
$c|_Z = c_Z$ and $c(z^{\op}) = c_{Z}(z)^{-1}$ for $z \in Z$. Then $(\G, c)$ is a
$\Gamma$-graded groupoid, $\Io{c^{-1}(\id_\Gamma)}$ is torsion-free and abelian, and $(\go_i \G
\go_i, c) \cong (\G_i, c_i)$ for each $i$. Conversely, given a $\Gamma$-graded groupoid
$(\G, c)$ such that $\Io{c^{-1}(\id_\Gamma)}$ is torsion-free and abelian, and given complementary
open $\G$-full sets $K_1, K_2 \subseteq \go$ such that each $(K_i \G K_i, c) \cong (\G_i,
c_i)$, the pair $(K_1 \G K_2, c|_{K_1 \G K_2})$ is a $(\G_1, c_1)$--$(\G_2, c_2)$
equivalence under the left and right actions given by multiplication in $\G$.
\end{lem}
\begin{proof}
Each $c_Z(\gamma \cdot z \cdot \eta) = c_1(\gamma) c_Z(z) c_2(\eta)$ so $c$ is a cocycle; and
$\Io{G} \cap c^{-1}(\id_\Gamma) = (\Io{\G_1} \cap c_1^{-1}(\id_\Gamma)) \sqcup (\Io{\G_2} \cap
c_2^{-1}(\id_\Gamma))$ is abelian and torsion free. As $(\go_i \G\go_i, c) \cong (\G_i, c_i)$, the
first statement follows. For the second, we saw that $K_1 \G K_2$ is a $\G_1$--$\G_2$-equivalence;
and each $c_Z(\gamma \cdot z \cdot \eta) = c(\gamma z \eta) = c(\gamma) c(z) c(\eta) = c_1(\gamma)
c_Z(z) c_2(\eta)$.
\end{proof}

We now turn to Morita equivalence of pairs $(A,D)$ where $A$ is a $C^*$-algebra and $D\subseteq A$ is a $C^*$-subalgebra. As in
Section~\ref{sec:background}, if $X$ is an $A_1$--$A_2$ imprimitivity bimodule, then
$X^*$ is its adjoint module and the linking algebra $A = A_1 \oplus X \oplus X^* \oplus
A_2$ contains $A_1,A_2$ as complementary full corners. Writing $P_i$ for $1_{M(A_i)}\in
M(A)$, the \emph{multiplier module of $X$} is $M(X) := P_1 M(A) P_2$ (see \cite{ER}); so
$M(X^*) \cong P_2 M(A) P_1 = M(X)^*$, and the adjoint operation in $M(A)$ determines an
extension to multiplier modules of the map $\xi \mapsto \xi^*$ from $X$ to $X^*$.

\begin{dfn}\label{dfn:sds Me}
Suppose that for $i = 1,2$, $A_i$ is a $C^*$-algebra carrying a coaction $\delta_i$ of a
discrete group $\Gamma$, and $D_i \subseteq A_i^{\delta_i}$ is a $C^*$-subalgebra. We say that
$(A_1, D_1)$ and $(A_2, D_2)$ are \emph{equivariantly Morita equivalent} if there are an
$A_1$--$A_2$-imprimitivity bimodule $X$ and a linear map $\zeta: X
\to M(X\otimes C^*_r(\Gamma))$ such that: $\zeta$ is a right-Hilbert bimodule morphism in the 
sense that $\zeta(a \cdot \xi \cdot b) = 
\delta_1(a)\cdot\zeta(\xi)\cdot\delta_2(b)$ for all $\xi, a, b$; we have $(\zeta\otimes\id_\Gamma)\circ\zeta = (\id_X\otimes\delta_\Gamma)\circ\zeta$; 
and for each $g \in \Gamma$, the subspace $X_g :=
\{x \in X : \zeta(x) = x \otimes \lambda_g\}$ satisfies
\begin{equation}\label{eq:Xg span}
X_g = \clsp\{\xi \in X_g : {_{A_1}\langle \xi, \xi\cdot D_2\rangle} \subseteq D_1
    \text{ and }\langle D_1 \cdot \xi, \xi\rangle_{A_2} \subseteq D_2\}.
\end{equation}

If $\Gamma = \{0\}$, we say that $(A_1, D_1)$ and $(A_2, D_2)$ are \emph{Morita
equivalent}.
\end{dfn}

The following is well known.

\begin{lem}\label{lem:separable}
Let $A_1, A_2$ be separable $C^*$-algebras and let $X$ be an $A_1$--$A_2$-imprimitivity
bimodule. Then $X$ is separable.
\end{lem}
\begin{proof}
Since $\langle \cdot, \cdot\rangle_{A_2}$ is full, and $A_2$ is separable there are
sequences $(y_n), (z_n)$ in $X$ such that $(\langle y_n, z_n\rangle_{A_2})_n$ is dense in
$A_2$. Cohen factorisation and the imprimitivity condition give $X = \clsp\{x \cdot
\langle y_n, z_n\rangle_{A_2} : x \in X, n \in \NN\} = \clsp\{{_{A_1}\langle x,
y_n\rangle} \cdot z_n : x \in X, n \in \NN\} \subseteq \clsp\{A_1 \cdot z_n : n \in
\NN\}$, which is separable because $A_1$ is.
\end{proof}

We now state the main result of the section; the proof occupies the rest of the section.

\begin{thm}\label{thm:eqiv systems->eqiv groupoids}
Let $\Gamma$ be a discrete group, and let $A_1, A_2$ be separable $C^*$-algebras.
Suppose, for $i = 1,2$, that $\delta_i$ is a coaction of $\Gamma$ on $A_i$ and $D_i\subseteq A_i^{\delta_i}$ is an abelian $C^*$-subalgebra containing an approximate unit for $A_i^{\delta_i}$. Suppose that $(X, \zeta)$ is an equivariant
Morita equivalence between $(A_1, D_1)$ and $(A_2, D_2)$. Let $A$ be the linking algebra
of $X$, let $D := D_1 \oplus D_2 \subseteq A$ and let $\delta : A \to A \otimes
C^*_r(\Gamma)$ be the map that restricts to $\delta_i$ on each $A_i$, to $\zeta$ on $X$
and to $x^* \mapsto \zeta(x)^*$ on $X^*$. Then $\delta$ is a coaction and $D\subseteq A^\delta$ is an abelian $C^*$-subalgebra that contains an approximate unit for $A^\delta$. The sets $\widehat{D}_i$ are complementary full clopen subsets of
the unit space of $\Hh := \Hh(A, D, \delta)$, we have $(\widehat{D}_i\Hh\widehat{D}_i,
c_{\delta}) \cong (\Hh(A_i, D_i, \delta_i), c_{\delta_i})$ for $i = 1,2$, and
$(\widehat{D}_1 \Hh \widehat{D}_2, \delta)$ is an equivalence from $\big(\Hh(A_1, D_1,
\delta_1), c_{\delta_1}\big)$ to $\big(\Hh(A_2, D_2, \delta_2), c_{\delta_2}\big)$.
\end{thm}

\begin{lem}\label{lem:LXdiagonal}
For $i = 1,2$, let $A_i$ be a separable $C^*$-algebra and $D_i \subseteq A_i$ an abelian
$C^*$-subalgebra containing an approximate unit for $A_i$. Let $X$ be an
$A_1$--$A_2$-imprimitivity bimodule such that
\[
X = \clsp\{x \in X : \langle D_1\cdot x,x\rangle_{A_2} \subseteq D_2
        \text{ and } {_{A_1}}\langle x,x\cdot D_2\rangle \subseteq D_1\}.
\]
Let $A$ be the linking algebra of $X$. Then $D := D_1 \oplus D_2$ is an abelian
$C^*$-subalgebra of $A$ and contains an approximate unit for $A$. The spaces $\Nn_{A_i}(D_i)
\subseteq A_i$ and
\[
\{\xi \in X : {_{A_1}}\langle \xi, \xi\cdot D_2\rangle \subseteq D_1\text{ and }\langle D_1\cdot \xi, \xi\rangle_{A_2} \subseteq D_2\}
\]
are all contained in $\Nn_A(D)$.
\end{lem}
\begin{proof}
Fix approximate units $(u^i_j)_j$ for $A_i$ in $D_i$. Then $(u^1_j \oplus u^2_j)_j$ is an
approximate unit for $A$ in $D$, which is clearly an abelian $C^*$-subalgebra of $A$.
Clearly each $N_{A_i}(D_i) \subseteq N_A(D)$.

If ${_{A_1}}\langle \xi, \xi\cdot D_2\rangle \in D_1\text{ and }\langle D_1\cdot \xi,
\xi\rangle_{A_2} \in D_2$, then for $d_i \in D_i$,
\[
\begin{pmatrix} 0 & \xi\\ 0& 0\end{pmatrix}
\begin{pmatrix} d_1 & 0\\ 0& d_2\end{pmatrix}
\begin{pmatrix} 0 & \xi\\ 0& 0\end{pmatrix}^*
= \begin{pmatrix} _{A_1} \langle \xi, \xi \cdot d_2 \rangle & 0\\ 0&0 \end{pmatrix} \in D_1 \oplus D_2,
\]
and a similar computation shows that
\[
\begin{pmatrix} 0 & \xi\\ 0& 0\end{pmatrix}^*
\begin{pmatrix} d_1 & 0\\ 0& d_2\end{pmatrix} \begin{pmatrix} 0 &
\xi\\ 0& 0\end{pmatrix} = \begin{pmatrix} 0 & 0 \\ 0 & \langle d_1^* \cdot \xi , \xi \rangle_{A_2} \end{pmatrix} \in D.\qedhere
\]
\end{proof}

\begin{lem}\label{lem-linking-unital-quotient}
Let $\Gamma$ be a discrete group. For $i = 1,2$, let $\delta_i$ be a coaction of $\Gamma$
on a separable $C^*$-algebra $A_i$, and let $D_i \subseteq A_i^{\delta_i}$ be an abelian $C^*$-subalgebra containing an approximate unit for $A_i^{\delta_i}$. Suppose that $(X, \zeta)$ is an equivariant Morita equivalence between
$(A_1, D_1)$ and $(A_2, D_2)$. Let $A$ be the linking algebra of $X$, let $D = D_1 \oplus
D_2$, and define $\delta : A \to A \otimes C^*_r(\Gamma)$ by
\[
\delta \left(\begin{matrix} a_1 & \xi \\ \eta^* & a_2\end{matrix}\right)
    = \left(\begin{matrix} \delta_1(a_1) & \zeta(\xi) \\
    \zeta(\eta)^* & \delta_2(a_2)\end{matrix}\right).
\]
Then $\delta$ is a coaction of $\Gamma$ on $A$, $D$ is an abelian $C^*$-subalgebra of $A^\delta$ and contains an approximate unit for $A^\delta$, and $D'_{A^\delta} =
(D_1)'_{A_1^{\delta_1}} \oplus (D_2)'_{A_2^{\delta_2}}$. For $i = 1,2$, the pair $(P_i A P_i, P_i D_i)$ is equivariantly
isomorphic to $(A_i, D_i)$. We have $P_1 A P_2 = \clsp\{P_1 n P_2 : n \in
\Nn_\star(D)\}$.
\end{lem}
\begin{proof}
It is routine to check that $\delta$ is a coaction and that $D$ is an abelian $C^*$-subalgebra of $A^\delta$ and contains an approximate unit for $A^\delta$. Clearly $(D_1)'_{A_1^{\delta_1}}
\oplus (D_2)'_{A_2^{\delta_2}} \subseteq D'_{A^\delta}$. For the reverse, fix
$\big(\begin{smallmatrix} a_1 & \xi \\ \eta^* & a_2 \end{smallmatrix}\big) \in
D'_{A^\delta}$. Then
\[
\begin{pmatrix} a_1  \otimes 1_{C^*_r(\Gamma)} & \xi \otimes 1_{C^*_r(\Gamma)} \\
                \eta^* \otimes 1_{C^*_r(\Gamma)} & a_2 \otimes 1_{C^*_r(\Gamma)} \end{pmatrix}
    = \begin{pmatrix} a_1 & \xi \\ \eta^* & a_2 \end{pmatrix} \otimes 1_{C^*_r(\Gamma)}
    = \delta\begin{pmatrix} a_1 & \xi \\ \eta^* & a_2 \end{pmatrix}
    = \begin{pmatrix} \delta_1(a_1) & \zeta(\xi) \\ \zeta(\eta)^* & \delta_2(a_2) \end{pmatrix},
\]
so each $a_i \in A^{\delta_i}$, and for $d_i \in D_i$,
\[
\begin{pmatrix}
a_1 d_1 & \xi \cdot d_2 \\
\eta^* \cdot d_1 & a_2 d_2
\end{pmatrix}
=
\begin{pmatrix}
a_1 & \xi \\
\eta^* & a_2
\end{pmatrix}
\begin{pmatrix}
d_1 & 0 \\
0 & d_2
\end{pmatrix}
=
\begin{pmatrix}
d_1 & 0 \\
0 & d_2
\end{pmatrix}
\begin{pmatrix}
a_1 & \xi \\
\eta^* & a_2
\end{pmatrix}
=
\begin{pmatrix}
d_1a_1 & d_1\cdot \xi \\
d_2 \cdot \eta^* & d_2a_2
\end{pmatrix},
\]
so each $a_i \in (D_i)'_{A_i^{\delta_i}}$. Moreover, $d_1 \cdot \xi = 0$ for all $d_1 \in
D_1$, so $\xi = 0$ by Cohen factorisation, and similarly $\eta^* = 0$. So
$\big(\begin{smallmatrix} a_1 & \xi \\ \eta^* & a_2
\end{smallmatrix}\big)  = \big(\begin{smallmatrix} a_1 & 0 \\ 0 & a_2
\end{smallmatrix}\big) \in (D_1)'_{A_1^{\delta_1}} \oplus (D_2)'_{A_2^{\delta_2}}$.

Each $(P_iA P_i, P_i D_i, \delta|_{P_i A P_i}) \cong (A_i, D_i, \delta_i)$ by
construction, and Lemma~\ref{lem:LXdiagonal} gives
\[
\{\xi \in X_g : {_{A_1}}\langle \xi, \xi\cdot D_2\rangle \subseteq D_1\text{ and }\langle D_1\cdot \xi, \xi\rangle_{A_2} \subseteq D_2\}
    \subseteq \Nn_\star(D)\quad\text{ for all $g \in G$.}
\]
Each $X_g \subseteq X = P_1 A P_2$, so $P_1 A P_2 = \clsp\{P_1 n P_2 : n \in
\Nn_\star(D)\}$ by~\eqref{eq:Xg span}.
\end{proof}

We call the system $(A, D, \delta)$ of Lemma~\ref{lem-linking-unital-quotient} the
\emph{linking system} for $(X, \zeta)$. The elements $P_i := 1_{M(D_i)} = 1_{M(A_i)}$ are
complementary full multiplier projections of $A$ and $(P_i A P_i, P_i D, \delta|_{P_i A
P_i}) \cong (A_i, D_i, \delta_i)$ for $i = 1,2$. So the $(A_i, D_i, \delta_i)$ are
complementary full subsystems of their linking system. The converse requires additional
hypotheses.

\begin{lem}\label{lem:linking->Me}
Let $A$ be a separable $C^*$-algebra, $\delta$ a coaction of a discrete group $\Gamma$ on
$A$, and $D\subseteq A$ an abelian $C^*$-subalgebra containing an approximate unit for $A^\delta$. Suppose that $P_1, P_2 \in M(D)$
are complementary full projections. Then for each $i$, we have that $\delta_i := \delta|_{P_i A P_i}$ is a coaction of $\Gamma$ on $A_i := P_i A P_i$ and $D_i := P_i D$ is an abelian $C^*$-algebra of $A_i^{\delta_i}$ and contains an approximate unit for $A_i^{\delta_i}$. If $P_1 A P_2 = \clsp\{P_1 n P_2 : n \in \Nn_\star(D)\}$, then $(P_1
A P_2, \delta|_{P_1 A P_2})$ is an equivariant Morita equivalence from $(A_1, D_1)$ to
$(A_2, D_2)$.
\end{lem}
\begin{proof}
Let $A_i^{\delta_i}:=P_iA^\delta P_i$.
Fix an approximate unit $(u_j)_j$ for $A$ in $D$. It is obvious that $D_i$ is a $C^*$-subalgebra of $A_i^{\delta_i}$, that $(P_i u_j)_j$ is an approximate unit for $A_i^{\delta_i}$ in $D_i$, and that $\delta_i$ is a $*$-homomorphism. Since $P_i \in M(D_i) \subseteq M(A_i^{\delta_i})$, the extension $\tilde{\delta}$ of
$\delta$ to $M(A^{\delta_i}_i)$ satisfies $\tilde{\delta}(P_i) = P_i \otimes
1_{C^*_r(\Gamma)}$, so for $g \in \Gamma$ and $a \in A_g$, we have $\delta(P_i a P_i) =
\tilde{\delta}(P_i)(a \otimes \lambda_g)\tilde{\delta}(P_i) = P_i a P_i \otimes \lambda_g \in A_i
\otimes C^*_r(\Gamma)$. It follows that $\delta_i$ is a coaction of $\Gamma$ on $A_i$ with generalised fixed-point algebra $A_i^{\delta_i}$.

It is well known that $P_1 A P_2$ is an
$A_1$--$A_2$-imprimitivity bimodule. By definition of the inner products, $P_1 A P_2 =
\clsp\{P_1 n P_2 : n \in \Nn_\star(D)\}$ if and only if
\[
(P_1 A P_2)_g = \clsp\{\xi \in (P_1 A P_2)_g : {_{A_1}\langle \xi, \xi\cdot D_2\rangle} \subseteq D_1
    \text{ and }\langle D_1 \cdot \xi, \xi\rangle_{A_2} \subseteq D_2\}
\]
for each $g \in \Gamma$. That $\delta|_{P_1 A P_2}(a \cdot \xi \cdot b) = \delta_1(a)\cdot \delta|_{P_1 A P_2}(\xi)\cdot\delta_2(b)$ for all $\xi,a,b$ is just the homomorphism property of $\delta$. That $(\delta|_{P_1 A P_2} \otimes\id_\Gamma) \circ \delta|_{P_1 A P_2} =
(\id_{P_1 A P_2}\otimes\delta_\Gamma) \circ \delta|_{P_1 A P_2}$ follows from the
coaction identity for $\delta$.
\end{proof}

\begin{cor}\label{cor:ehn linking<->Me}
Let $\Gamma$ be a discrete group.  For $i = 1,2$, let $\delta_i$ be a coaction of
$\Gamma$ on a separable $C^*$-algebra $A_i$, and let $D_i \subseteq A_i^{\delta_i}$ be an abelian $C^*$-subalgebra containing an approximate unit for $A_i^{\delta_i}$ such that $\clsp\Nn_\star(D_i) = A_i$. Then $(A_1, D_1)$ and
$(A_2, D_2)$ are equivariantly Morita equivalent if and only if there exist a separable
$C^*$-algebra $A$, a coaction $\delta$ of $\Gamma$ on $A$, an abelian $C^*$-subalgebra $D\subseteq A^\delta$ that contains an approximate unit for $A^\delta$ such that $\clsp \Nn_\star(D) = A$, a pair of complementary full
projections $P_1, P_2\in M(D)$, and isomorphisms $\phi_i : P_i A P_i \to A_i$ such that
\[
    \phi_i(P_i D P_i) = D_i\qquad\text{and}\qquad \delta_i \circ \phi_i = (\phi_i \otimes \id_{C^*_r(\Gamma)}) \circ \delta|_{P_i A P_i}\quad\text{ for $i = 1,2$.}
\]
\end{cor}
\begin{proof}
We have $P_2 A P_1 = (P_1 A P_2)^* \subseteq \Nn_\star(D)$, and each $P_i A P_i \subseteq
A \subseteq \clsp\Nn_\star(D)$. So Lemma~\ref{lem-linking-unital-quotient} gives ``only
if." Lemma~\ref{lem:linking->Me} gives ``if" because $A = \clsp \Nn_\star(D)$ forces $P_i
AP_i = \clsp\{P_i n P_i : n \in \Nn_\star(D)\} \subseteq \Nn_\star(D_i)$.
\end{proof}

\begin{lem}\label{lem:normalizer-linking-alg}
Let $\Gamma$ be a discrete group.  For $i = 1,2$, let $\delta_i$ be a coaction of
$\Gamma$ on a separable $C^*$-algebra $A_i$, and let $D_i \subseteq A_i^{\delta_i}$ be an abelian $C^*$-subalgebra containing an approximate unit for $A_i^{\delta_i}$. Suppose that $(X, \zeta)$ is an equivariant $(A_1,
D_1)$--$(A_2, D_2)$-equivalence. Let $(A, D, \delta)$ be the linking
system. Fix $g \in \Gamma$ and $n = \big(\begin{smallmatrix} a_1 & \xi \\
\eta^* & a_2 \end{smallmatrix}\big) \in \Nn_g(D)$. Then $a_1^*\cdot \xi = \xi \cdot a_2^*
= a_1  \cdot \eta = \eta \cdot a_2 = 0$. Moreover, $a_1 \in \Nn_g(D_1)$, $a_2 \in
\Nn_g(D_2)$, and $\langle D_1 \cdot \xi,  \xi \rangle_{A_2} \cup \langle D_1 \cdot \eta,
\eta \rangle_{A_2} \subseteq D_2$, and ${_{A_1}\langle \xi \cdot D_2, \xi \rangle} \cup
{_{A_1}\langle \eta \cdot D_2, \eta \rangle} \subseteq D_{2}$.
\end{lem}
\begin{proof}
Let $d_1 \in D_1$ and let $d_2 \in D_2$.  Then
\begin{align*}
n^* ( d_1 \oplus 0 ) n
    &= \begin{pmatrix}
        a_1^* d_1 a_1 & a_1^* d_1 \cdot \xi \\
        \xi^* \cdot d_1a_1 & \langle d_1^* \cdot \xi, \xi \rangle_{A_2}
    \end{pmatrix},
&n^* ( 0  \oplus d_2 ) n
    &= \begin{pmatrix}
        {_{A_1}\langle \eta , \eta\cdot d_2\rangle} & \eta\cdot d_2a_2 \\
        a_2^* d_2 \cdot \eta^* & a_2^* d_2 a_2
    \end{pmatrix},\\
n ( d_1 \oplus 0 ) n^*
    &= \begin{pmatrix}
        a_1 d_1 a_1^* & a_1 d_1 \cdot \eta \\
        \eta^*\cdot d_1a_1^* & \langle d_1^* \eta, \eta \rangle_{A_2}
    \end{pmatrix},\text{ and}
&n ( 0 \oplus d_2 ) n^*
    &= \begin{pmatrix}
        {_{A_1}\langle \xi, \xi \cdot d_2\rangle} & \xi \cdot d_2a_2^*  \\
        a_2 d_2 \cdot \xi^* & a_2 d_2 a_2^*
    \end{pmatrix}.
\end{align*}

That $n^* ( d_1 \oplus 0 ) n, n ( d_1 \oplus 0 )n^* \in D$, gives $a_1 \in \Nn(D_1)$ and
$(a_1^*d_1)\cdot \xi = (\xi_ \cdot d_2) \cdot a_2^* = (a_1 d_1 ) \cdot \eta = (\eta \cdot
d_2 ) \cdot a_2 = 0$. Taking the limit as the $d_i$ range over approximate identities for
the $A_i$ proves the first assertion. Since $n \in \Nn_g(D)$ we have
\[
\begin{pmatrix}
    a_1 \otimes \lambda_g & \xi \otimes \lambda_g \\
    \eta^* \otimes \lambda_g & a_2 \otimes \lambda_g \\
\end{pmatrix}
    = n \otimes \lambda_g
    = \delta(n)
    = \begin{pmatrix}
        \delta_1(a_1) & \zeta(\xi) \\
        \zeta(\eta)^* & \delta_2(a_2)
    \end{pmatrix},
\]
so $a_i \in \Nn_g(D_i)$. Since $n^* (d_1 \oplus 0) n, n(d_1 \oplus 0)n^* \in D_1 \oplus
D_2$, we have $\langle d_1^* \cdot \xi, \xi \rangle_{A_2}, \langle d_1^* \eta , \eta
\rangle_{A_2} \in D_2$ and ${_{A_1}\langle \eta , \eta\cdot d_2\rangle}, {_{A_1}\langle
\xi, \xi\cdot d_2 \rangle} \in D_1$, which proves the second assertion.
\end{proof}

In what follows, we sometimes deal with Weyl groupoids for multiple triples $(A, D,
\delta)$. If $\Hh = \Hh(A, D, \delta)$ is a Weyl groupoid, and if $n \in \Nn_\star(D)
\subseteq A$ and $\phi \in \osupp(n^*n) \subseteq \widehat{D}$, we write $[n, \phi]_\Hh$
for the corresponding element of $\Hh$.

\begin{lem}\label{lem:normalizer-comp}
Let $\Gamma$ be a discrete group.  For $i = 1,2$, let $\delta_i$ be a coaction of
$\Gamma$ on a separable $C^*$-algebra $A_i$, and let $D_i \subseteq A_i^{\delta_i}$ be an abelian $C^*$-subalgebra containing an approximate unit for $A_i^{\delta_i}$. Suppose that $(X, \zeta)$ is an equivariant $(A_1,
D_1)$--$(A_2, D_2)$-equivalence. Let $(A, D, \delta)$ be the linking system. Let $\Hh =
\Hh(A, D, \delta)$. Suppose that
\[
n = \begin{pmatrix}n_1 & \xi \\\eta^* & n_2\end{pmatrix}\in N_g(D)
    \quad\text{ where $n_1 \in A_1$, $n_2 \in A_2$, and $\xi,\eta \in X$.}
\]
Fix $i \in \{1,2\}$, and suppose that $\phi\in\osupp(n^*n)$ satisfies $[n,\phi]_\Hh \in
\widehat{D_i} \Hh \widehat{D_i}$. Then $\phi\in\osupp(n_i^*n_i)$, and $[n,\psi]_\Hh =
[n_i,\psi]_\Hh$ for all $\psi\in\osupp(n_i^*n_i)$.
\end{lem}
\begin{proof}
By symmetry, it suffices to prove the result for $i = 1$. By
Lemma~\ref{lem:normalizer-linking-alg}, we have $n_1 \in \Nn_g(D_1)$ and $n_2 \in
\Nn_g(D_2)$. Since $[n, \phi]_\Hh \in \widehat{D_1} \Hh \widehat{D_1}$, we have $\phi,
\alpha_n(\phi) \in \widehat{D}_1$.

To see that $\phi\in\osupp(n_1^*n_1)$, fix $d_1 \in D_1$ with $\alpha_n(\phi)(d_1) = 1$.
Then
\begin{align*}
0 < \phi(n^*n)
    &= \alpha_n(\phi)(d_1)\phi(n^*n)
    = \phi(n^* d_1 n)\\
    &= \phi\begin{pmatrix}
        n_1^*d_1n_1 & n_1^*d_1\cdot\xi \\
        \xi^* \cdot d_1 n_1 & \langle\xi, d_1\cdot \xi\rangle_{A_2}
    \end{pmatrix}
    = \phi(n_1^*d_1n_1)
    = \alpha_{n_1}(\phi)(d_1) \phi(n_1^*n_1).
\end{align*}

Since Lemma~\ref{lem:normalizer-linking-alg} gives $n_1 \in \Nn_g(D_1)$, we have $n_1^*n
\in A^\delta$. Let $U:=\osupp(n_1^*n_1)\subseteq\osupp(n^*n)$ and $\psi\in U$. Then $\psi
\in \widehat{D}_1$, and therefore $\psi$ vanishes on $D_2$. So for $d =
\big(\begin{smallmatrix} d_1 & 0 \\ 0 & d_2\end{smallmatrix}\big) \in \widehat{D}_+$ with
$\alpha_{n_1}(\psi)(d) = 1$,
\begin{align*}
\alpha_n(\psi)(d)\psi(n^*n)
    &= \psi(n^*d n)
    = \psi\begin{pmatrix}
        n_1^*d_1 n_1 + {_{A_1}\langle \eta\cdot d_2, \eta\rangle} & a_1^*d_1\cdot \xi + \eta \cdot d_2 a_2\\
        \xi^* \cdot d_1 a_1 + a_2^* d_2 \cdot \eta & \langle\xi, d_1\cdot\xi\rangle_{A_2} + n^*_2d_2n_2
    \end{pmatrix}\\
    &= \psi(n_1^*d_1n_1 + {_{A_1}\langle \eta\cdot d_2, \eta\rangle})
    \ge \psi(n_1^*d_1n_1)
    = \alpha_{n_1}(\psi)(d)\psi(n_1^*n_1)
    > 0.
\end{align*}
Since $\psi(n^*n)=\psi(n_1^*n_1)$ and $D_+$ separates points in $\widehat{D}$, this gives
$\alpha_n(\psi) = \alpha_{n_1}(\psi)$, so $\alpha_{n_1}$ and $\alpha_n$ agree on $U$. To
see that $[n, \psi]_\Hh = [n_1, \psi]_\Hh$ it remains only to establish~(R4). We have
$n^* n_1 = \big(\begin{smallmatrix} n_1^* n_1 & 0 \\ \xi^* n_1 & 0\end{smallmatrix}\big)
= \big(\begin{smallmatrix} n_1^* n_1 & 0 \\ 0 & 0 \end{smallmatrix}\big)$. Thus for $d
\in D$ such that $\psi(d) = 1$ and $\osupp(d) \subseteq U$, we have $\pi_\psi
\big((\psi(n^* n)^{-1/2} \psi (n_1^* n_1))^{-1/2}(d n^* n_1 d) \big) =
1_{D'_{A^\delta}/J_\phi}$. Hence, $[n, \psi] = [n_1, \psi]$.
\end{proof}

\begin{proof}[Proof of Theorem~\ref{thm:eqiv systems->eqiv groupoids}]
We have that $\delta$ is a coaction and $D\subseteq A^\delta$ is an abelian $C^*$-subalgebra containing an approximate unit for $A^\delta$ by
Lemma~\ref{lem-linking-unital-quotient}. Clearly the $\widehat{D}_i$ are complementary
open subsets of $\widehat{D}$. We show that $\widehat{D}_1$ is full ($\widehat{D}_2$ is
full by a symmetric argument). Fix $\phi \in \widehat{D}_2$. Since $X$ is full in $A_2$,
there exists $\xi \in X$ such that $\langle \xi, \xi\rangle_{A_2} \in D_2$ and
$\phi(\langle \xi, \xi\rangle_{A_2}) \not= 0$. Lemma~\ref{lem-linking-unital-quotient}
shows that $\delta$ is a coaction, so as in Section~\ref{sec:background}, we have $A =
\clsp\bigcup_{g \in \Gamma} A_g$. Let $P_i = 1_{M(D_i)} \in M(A)$. Then $X = P_1 A P_2 =
\clsp\bigcup_{g \in \Gamma} P_1 A_g P_2 = \clsp\bigcup_{g \in \Gamma} X_g$.
Hence~\eqref{eq:Xg span} shows that
\[\textstyle
X = \clsp \bigcup_{g \in \Gamma} \{\eta \in X_g : {_{A_1}\langle \eta, \eta\cdot
    D_2\rangle} \subseteq D_1\text{ and } \langle D_1 \cdot\eta, \eta\rangle_{A_2} \subseteq D_2\}.
\]
In particular, we can approximate $\xi$ by an element $\sum_j \eta_j$ with each $\eta_j \in
\Nn_\star(D)$ satisfying $\eta_j^* D_2 \eta_j \subseteq D_1$. Since $\phi(\langle \xi,
\xi\rangle_{A_2}) \not= 0$, we may assume that there exist $j,k$ such that $\phi(\langle\eta_j,
\eta_k\rangle_{A_2}) \not= 0$. Using that $X$ is an imprimitivity bimodule at the last
step, we calculate:
\begin{align*}
0 &\not= |\phi(\langle\eta_j, \eta_k\rangle_{A_2})|^2
    = \phi(\langle\eta_j, \eta_k\rangle_{A_2}\langle \eta_k, \eta_j\rangle_{A_2})\\
    &= \phi\big(\big\langle \langle\eta_j, \eta_k \cdot \langle \eta_k, \eta_j\rangle_{A_2}\big\rangle_{A_2}\big)
    = \phi\big(\big\langle \eta_j, {_{A_1}\langle\eta_k, \eta_k\rangle}\cdot \eta_j\rangle_{A_2}\big\rangle_{A_2}\big).
\end{align*}
Rewriting this in terms of multiplication in $A$, we obtain $\phi(\eta_j^* \eta_k
\eta_k^*\eta_j) \not= 0$, and therefore $\eta := \eta_j$ satisfies
$\phi(\eta\eta^*)\alpha_{\eta}(\phi)(\eta_k \eta^*_k) \not= 0$. In particular, $\eta \in
\Nn_\star(D)$ and $\phi \in \dom(\alpha_\eta)$. Since $\eta D_2 \eta^* \subseteq D_1$, we
have $r([\eta, \phi]) = \alpha_\eta(\phi) \in \widehat{D}_1$. So $\phi = s([\eta,\phi])
\in s(\widehat{D}_1 \Hh)$.

We must show that each $\widehat{D}_i\Hh \widehat{D}_i \cong \Hh(A_i, D_i, \delta_i)$. It
suffices to do this for $i = 1$. Let $j_1 : A_1 \to A$ be the inclusion map. For $\phi
\in \widehat{D}_1$, let $\overline{\phi}$ be the extension of $\phi$ to $D$ given by
$\overline{\phi}(d_1)=\phi(d_1)$ for $d_1\in D_1$ and $\overline{\phi}(d_2)=0$ for $d_2\in
D$.

We claim that there is a map $\Theta \colon \mathcal{H}(A_1, D_1, \delta_1) \to
\widehat{D}_1\mathcal{H}(A, D, \delta) \widehat{D}_1$ such that
\[
    \Theta([n, \phi]_{\Hh_1}) = [j_1(n), \overline{\phi}]_\Hh \quad\text{ for all $[n, \phi] \in \Hh(A, D, \delta)$.}
\]

For this, suppose that $[n,\phi]_{\Hh_1} = [m,\psi]_{\Hh_1}$. Then $\phi = \psi$ by
definition of $\sim$ in $(A_1, D_1, \delta_1)$, so $(j_1(n),\overline{\phi})$ and
$(j_1(m), \overline{\psi})$ satisfy~(R1). Since $n^*m \in A_1^{\delta_1}$, we have $n,m
\in \Nn_g(D_1)$ for some $g$. Hence $\delta(j_1(n)) = j_1(n) \otimes \lambda_g$ and
$\delta(j_1(m)) = j_1(m) \otimes \lambda_g$, so $j_1(n)^*j_1(m) \in A^\delta$,
giving~(R2). Since $(n,\phi), (m,\psi)$ satisfy~(R3), there is a neighbourhood $U
\subseteq \widehat{D}_1$ of $\phi$ such that $U\subseteq\osupp(n^*n)\cap\osupp(m^*m)$,
and $\alpha_n|_U = \alpha_m|_U$. Since $\widehat{D}_1$ is open in $\widehat{D}$, the set
$\overline{U}:=\{\overline{\psi}:\psi\in U\}$ is a neighbourhood of $\overline{\phi}$ in
$\widehat{D}$. Lemma~\ref{lem:normalizer-comp} gives $\alpha_{j_1(n)}|_{\overline{U}} =
\alpha_n|_{\overline{U}} = \alpha_m|_{\overline{U}} = \alpha_{j_1(m)}|_{\overline{U}}$.
So $(j_1(n),\overline{\phi})$ and $(j_1(m), \overline{\psi})$ satisfy~(R3).
Lemma~\ref{lem-linking-unital-quotient} gives $D'_{A^\delta} = (D_1)'_{A_1^{\delta_1}}
\oplus (D_2)'_{A_2^{\delta_2}}$. Since $J_{\overline{\phi}} \lhd D'_{A^\delta}$ is
generated by $\{d \in D : \phi(d) = 0\}$, the corresponding ideal $J_{1,\phi} \lhd
(D_1)'_{A_1^{\delta_1}}$ satisfies $J_{\overline{\phi}} = j_1(J_{1,\phi}) \oplus
(D_2)'_{A_2^{\delta_2}}$. Hence the projection map $p_1 : D'_{A^{\delta}} \to
(D_1)'_{A_1^{\delta_1}}$ descends to an isomorphism $\tilde{p}_1 :
D'_{A^\delta}/J_{\overline{\phi}} \cong (D_1)'_{A_1^{\delta_1}}/J_{1,\phi}$. By
definition of $U^\phi_{n^*m}$ and $U^{\overline{\phi}}_{j_n(n)^*j_1(m)}$ as in
Notation~\ref{ntn:Un*m},
\begin{equation}\label{eq:Us consistent}
    \tilde{p}_1(U^{\overline{\phi}}_{j_n(n)^*j_1(m)}) = U^\phi_{n^*m}.
\end{equation}
Since $U^\phi_{n^*m} \sim_h 1$, we deduce that $U^{\overline{\phi}}_{j_n(n)^*j_1(m)}
\sim_h 1$. So $[j_1(n), \overline{\phi}] = [j_1(m), \overline{\psi}]$, and $\Theta$ is
well defined. For any $[n, \phi]_{\Hh_1} \in \Hh_1$, we have $s(\Theta([n,\phi])) =
\overline{\phi} \in \widehat{D}_1$, and $r(\Theta([n,\phi])) =
\alpha_{j_1(n)}(\overline{\phi}) = \overline{\alpha_n(\phi)} \in \widehat{D}_1$ by
Lemma~\ref{lem:normalizer-comp}. So $\Theta(\Hh(A_1, D_1, \delta_1)) \subseteq
\widehat{D}_1 \Hh \widehat{D}_1$.

To see that $\Theta$ is surjective, fix $[n, \phi]_\Hh \in
\widehat{D}_1\Hh \widehat{D}_1$. Write $n = \big(\begin{smallmatrix} n_1 & \xi \\
\eta^* & n_2\end{smallmatrix}\big)$. Then Lemma~\ref{lem:normalizer-comp} gives $[n,
\phi]_\Hh = [j_1(n_1), \overline{\phi}]_\Hh = \Theta([n_1, \phi]_{\Hh_1})$. For
injectivity, suppose that $\Theta([n,\phi]_{\Hh_1}) = \Theta([m,\psi]_{\Hh_1})$. Then
$\phi = \psi$, and $n^*m \in A_1^{\delta_1}$ because $j_1(n^*m) \in A^\delta$ and $\delta
\circ j_1 = (j_1 \otimes \id_{C^*_r(\Gamma)}) \circ \delta_1$.
Lemma~\ref{lem:normalizer-comp} gives $\alpha_{j_1(n)} = \alpha_n$ and $\alpha_{j_1(m)} =
\alpha_m$. So $\alpha_{n}$ and $\alpha_{m}$ agree on a neighbourhood of $\phi$. Since
$U^{\overline{\phi}}_{j_1(n)^*j_1(m)} \in \Uu_0(D'_{A^\delta}/J_\phi)$,
Equation~\eqref{eq:Us consistent} gives $U^\phi_{n^*m} \in
\Uu_0\big((D_1)'_{A_1^{\delta_1}}/J_{1,\phi}\big)$. So $[n,\phi]_{\Hh_1} =
[m,\psi]_{\Hh_1}$.

For $[n,\phi]_{\Hh_1} \in \Hh_1$, we have $s(\Theta([n,\phi]_{\Hh_1})) = s([j_1(n),
\overline{\phi}]_\Hh) = \overline{\phi} = \overline{s([n,\phi]_{\Hh_1})}$, so
Lemma~\ref{lem:normalizer-comp} gives $r(\Theta([n,\phi]_{\Hh_1})) =
\alpha_{j_1(n)}(\overline{\phi}) = \overline{\alpha_n(\phi)} = \overline{r([n,
\phi]_{\Hh_1})}$. Furthermore, if $[n,\phi]_{\Hh_1}$ and $[m,\psi]_{\Hh_1}$ are
composable, then since $j_1$ is a homomorphism,
\[
\Theta([n,\phi]_{\Hh_1})\Theta([m,\psi]_{\Hh_1})
    = [j_1(n),\overline{\phi}]_\Hh [j_1(m),\overline{\psi}]_\Hh
    = [j_1(nm), \overline{\psi}]_\Hh
    = \Theta([n,\phi]_{\Hh_1}[m,\psi]_{\Hh_1}).
\]
Thus $\Theta$ is a groupoid homomorphism.

We must show that $\Theta$ is a homeomorphism. For $n \in \Nn_\star(D_1)$ and $U
\subseteq \osupp(n^*n) \subseteq \widehat{D}_1$ open, $\Theta(Z(n, U)) = Z(j_1(n),
\overline{U})$. Hence $\Theta$ is an open map. Now fix a basic open set $Z(n, U) \cap
\widehat{D}_1\Hh\widehat{D}_1 \subseteq \widehat{D}_1\Hh\widehat{D}_1$. Write $n =
\big(\begin{smallmatrix} n_1 & \xi \\ \eta^* & n_2\end{smallmatrix}\big)$ and let
$U':=\{\phi\in\widehat{D_1}:\overline{\phi}\in U\}$. It then follows from
Lemma~\ref{lem:normalizer-comp} that if $[n,\phi]_\Hh\in Z(n, U) \cap
\widehat{D}_1\Hh\widehat{D}_1$, then $\Theta([n_1,\phi_1])=[j_1(n_1),
\overline{\phi_1}]=[n,\phi]$ where $\phi_1$ is the element of $\widehat{D_1}$ such that
$\overline{\phi_1}=\phi$. Hence $\Theta^{-1}(Z(n, U) \cap \widehat{D}_1\Hh\widehat{D}_1)
= Z(n_1,U')$ is open.
\end{proof}

\section{Equivalence of graded groupoids and equivariant Morita equivalence}\label{sec:E
vs Me}

In this section we show how equivalence of graded groupoids relates to equivariant Morita
equivalence of pairs $(A, D)$. Our main result in this direction is the following.

\begin{thm}\label{thm:2}
Let $\Gamma$ be a discrete group and let $(\G_1,c_1), (\G_2,c_2)$ be $\Gamma$-graded
second-countable locally compact Hausdorff \'etale groupoids, and suppose that each
$\Io{c_i^{-1}(\id_\Gamma)}$ is torsion-free and abelian. The following are equivalent:
\begin{enumerate}
\item\label{it:Gs equiv} the graded groupoids $(\G_1, c_1)$ and $(G_2, c_2)$ are
    graded equivalent;
\item\label{it:Gs linked} there exist a second-countable locally compact Hausdorff
    \'etale groupoid $G$, a grading $c$ of $\G$ by $\Gamma$ such that
    $\Io{c^{-1}(\id_\Gamma)}$ is torsion-free and abelian, a pair of complementary open
    $\G$-full subsets $K_1, K_2\subseteq \go$, and isomorphisms $\kappa_1 : K_1\G K_1
    \to \G_1$ and $\kappa_2 : K_2 \G K_2 \to \G_2$ such that $c_i \circ \kappa_i =
    c|_{K_i\G K_i}$;
\item\label{it:SDSes equiv} $(C^*_r(\G_1), C_0(\G_1^{(0)}))$ and $(C^*_r(\G_2),
    C_0(\G_2^{(0)}))$ are equivariantly Morita equivalent; and
\item\label{it:SDSes linked} there exist a separable $C^*$-algebra $A$, a coaction
    $\delta$ of $\Gamma$ on $A$, an abelian $C^*$-subalgebra $D\subseteq A^\delta$ containing an approximate unit for $A^\delta$ such
    that $\clsp \Nn_\star(D) = A$, a pair of complementary full projections $P_1,
    P_2\in M(D) \subseteq M(A)$, and isomorphisms $\phi_i : P_i A P_i \to
    C^*_r(\G_i)$ such that $\phi_i(P_i D P_i) = D_i$ and $\delta_i \circ \phi_i =
    (\phi_i \otimes \id_{C^*_r(\Gamma)}) \circ \delta|_{P_i A P_i}$.
\end{enumerate}
\end{thm}
\begin{proof}
Lemma~\ref{lem:graded equiv} gives \mbox{(\ref{it:Gs equiv})\,$\iff$\,(\ref{it:Gs
linked})}, Corollary~\ref{cor:ehn linking<->Me} gives \mbox{(\ref{it:SDSes
equiv})\,$\iff$\,(\ref{it:SDSes linked})}, and Theorem~\ref{thm:eqiv systems->eqiv
groupoids} gives \mbox{(\ref{it:SDSes equiv})\;$\implies$\;(\ref{it:Gs equiv})}. For
\mbox{(\ref{it:Gs linked})\;$\implies$\;(\ref{it:SDSes linked})}, let $(\G, c)$, $K_i$
and $\kappa_i$ be as in~(\ref{it:Gs linked}). Lemma~\ref{lem:gpd sds} shows that
$C_0(\go)$ is an abelian $C^*$-subalgebra of $C_r^*(G)^{\delta_c}$ and contains an approximate unit for $C_r^*(G)^{\delta_c}$. Theorem~4.1 of
\cite{SimWil} shows that $P_i := 1_{K_i} \in M(C_0(\go))$ defines complementary full
projections, and the inclusions $C_c(\G_i) \hookrightarrow C_c(\G)$ induce isomorphisms
$\phi_i : P_i C^*_r(\G) P_i \to C^*_r(\G_i)$ that carry $P_i C_0(\go)$ to $C_0(\go_i)$.
We have $P_1 C^*_r(\G) P_2 = \overline{C_c(\go_1 \G \go_2)} \subseteq C^*_r(\G)$. Since
the $\G_i$ are \'etale, the Haar system of \cite[Lemma~2.2]{SimWil} consists of counting
measures, so $\G$ is also \'etale. Hence Lemma~\ref{lem:bisecton alpha} gives $P_1
C^*_r(\G) P_2 = \clsp\{P_1 n P_2 : n \in \Nn_\star(C_0(\go))\}$. Fix $i \in \{1,2\}$. For
$g \in \Gamma$ and $f \in C_c(\G_i)_g$, $\supp(\phi_i(f)) = \supp(f) \subseteq
c_i^{-1}(g)$, and so $\delta_i(\phi_i(f))
    = \phi_i(f) \otimes \lambda_g
    = (\phi_i \otimes \id_{C^*_r(\Gamma)})(f \otimes \lambda_g)
    = (\phi_i \otimes \id_{C^*_r(\Gamma)})(\delta(f))$.
Since $P_i C^*_r(\G) P_i = \clsp \bigcup_g C_c(\G_i)_g$, we obtain $\delta_i \circ \phi_i
= (\phi_i \otimes \id_{C^*_r(\Gamma)}) \circ \delta|_{P_i A P_i}$.
\end{proof}

In some situations the following stronger form of equivalence than the one considered in
Theorem~\ref{thm:2}, is interesting (see Corollaries \ref{cor:strong equiv
ample}~and~\ref{cor:two-sided conjugacy}). If $\Gamma$ is the trivial group, then
Theorems \ref{thm:2}~and~\ref{thm:3} both reduce to Theorem~\ref{thm:no-coaction 2}.

\begin{thm}\label{thm:3}
Let $\Gamma$ be a locally compact group, let $\G_1$ and $\G_2$ be second-countable
locally compact Hausdorff \'etale groupoids, and let $c_i : \G_i \to \Gamma$ be
continuous cocycles such that $\Io{c_i^{-1}(\id_\Gamma)}$ is torsion-free and abelian. The
following are equivalent:
\begin{enumerate}
\item\label{it:Gs strong equiv} there is a graded $(\G_1, c_1)$--$(\G_2,
    c_2)$-equivalence $(Z, c_Z)$ such that $c_Z^{-1}(\id_\Gamma)$ is a
    $c_1^{-1}(e)$--$c_2^{-1}(\id_\Gamma)$-equivalence;
\item\label{it:Gs strong linked} there exist a second-countable locally compact
    Hausdorff \'etale groupoid $\G$, a grading $c$ of $\G$ by $\Gamma$ such that
    $\Io{c^{-1}(\id_\Gamma)}$ is torsion-free and abelian, a pair of complementary open
    $c^{-1}(\id_\Gamma)$-full subsets $K_1, K_2\subseteq \go$, and isomorphisms $\kappa_1 :
    K_1 \G K_1 \to \G_1$ and $\kappa_2 : K_2 \G K_2 \to \G_2$ such that $c_i \circ
    \kappa_i = c|_{K_i \G K_i}$;
\item\label{it:SDSes strong equiv} there exists an equivariant $(C^*_r(\G_1),
    C_0(\G_1^{(0)}))$--$(C^*_r(\G_2), C_0(\G_2^{(0)}))$-imprimitivity bimodule $(X,
    \zeta)$ such that $X_{\id_\Gamma}$ is a
    $C^*_r(\G_1)^{\delta_{c_1}}$--$C^*_r(\G_2)^{\delta_{c_2}}$-imprimitivity
    bimodule; and
\item\label{it:SDSes strong linked} there exist a separable $C^*$-algebra $A$, a
    coaction $\delta$ of $\Gamma$ on $A$, an abelian $C^*$-subalgebra $D\subseteq A^\delta$ containing an approximate unit for $A^\delta$ such that $\clsp \Nn_\star(D) = A$, a pair of complementary
    $A^\delta$-full projections $P_1, P_2\in M(D)$, and isomorphisms $\phi_i : P_i A
    P_i \to C^*_r(\G_i)$ such that $\phi_i(P_i D P_i) = D_i$ and $\delta_i \circ
    \phi_i = (\phi_i \otimes \id_{C^*_r(\Gamma)}) \circ \delta|_{P_i A P_i}$.
\end{enumerate}
\end{thm}
\begin{proof}
$(\ref{it:Gs strong equiv})\implies (\ref{it:Gs strong linked})$. By
Lemma~\ref{lem:graded equiv}, $G := L(G_1, G_2)$ and $c : G\to\Gamma$ given by $c|_{G_i}
= c_i$, $c|_Z = c_Z$ and $c(z^{\op}) = c_{Z}(z)^{-1}$ for $z \in Z$ constitute a graded
groupoid with $\Io{c^{-1}(\id_\Gamma)} \cong \Io{c_1^{-1}(\id_\Gamma)} \sqcup \Io{c_2^{-1}(\id_\Gamma)}$
torsion-free and abelian, and each $(\go_i \G \go_i, c) \cong (\G_i, c_i)$. Since
$c_Z^{-1}(\id_\Gamma)$ is a $c_1^{-1}(\id_\Gamma)$--$c_2^{-1}(\id_\Gamma)$-equivalence, each $\G_i^{(0)}$ is
$c^{-1}(\id_\Gamma)$-full.

The proof of Corollary~\ref{cor:ehn linking<->Me} gives \mbox{(\ref{it:SDSes strong
equiv})\,$\iff$\,(\ref{it:SDSes strong linked})}, and Theorem~\ref{thm:eqiv systems->eqiv
groupoids} gives \mbox{(\ref{it:SDSes strong equiv})\;$\implies$\;(\ref{it:Gs strong
equiv})}.

For \mbox{(\ref{it:Gs strong linked})\;$\implies$\;(\ref{it:SDSes strong linked})}, the
proof of \mbox{(\ref{it:Gs linked})\;$\implies$\;(\ref{it:SDSes linked})} in
Theorem~\ref{thm:2} shows that $P_i := 1_{K_i}$ defines full multiplier projections in
$M(C^*_r(\G))$ such that $P_i C^*_r(\G) P_i$ is equivariantly isomorphic to $C^*(\G_i)$.
To see that $P_1$ is $C^*_r(\G)^{\delta_c}$-full, we show that $C_c(\G^{(0)}) \subseteq
C^*_r(\G)^{\delta_c} P_1 C^*_r(\G)^{\delta_c}$. It suffices to show that for each $x \in \go$ there exists
$a \in C^*_r(G)^{\delta_c} P_1 C^*_r(\G)^{\delta_c} \cap C_0(\go)$ such that $a(x) > 0$. This is clear for
$x \in K_1$, so fix $x \in K_2$. Since $K_1$ is $c^{-1}(\id_\Gamma)$-full there is an open
bisection $U \subseteq c^{-1}(\id_\Gamma)$ with $x \in r(U)$ and $s(U) \subseteq K_1$. Write $U
\cap r^{-1}(x) = \{\gamma\}$. Fix $f \in C_c(U)$ with $f(\gamma) = 1$. Then $ff^* = f P
f^* \in C_0(K_2) \cap C^*_r(G)^{\delta_c} P_1 C^*_r(\G)^{\delta_c}$, and $ff^*(x) = 1$. Symmetry shows that
$P_2$ is also full.
\end{proof}

Next we specialise to ample graded groupoids. Recall that $\R$ is the discrete groupoid
$\NN\times\NN$, and $\mathcal{C}$ denotes the canonical diagonal in $C^*(\R) \cong \Kk$. Given a
grading $c : \G \to \Gamma$ of a groupoid $\G$, we define $\bar{c} : \G\times\R \to
\Gamma$ by $\bar{c}(\eta_1,\eta_2)=c(\eta_1)$.

As in \cite{CRS}, we say $\G_1$ and $\G_2$ are \emph{weakly Kakutani $c_1$--$c_2$
equivalent} if there are open $c_i^{-1}(\id_\Gamma)$-full subsets $X_i\subseteq \G_i^{(0)}$ and an
isomorphism $\kappa: X_1 \G_1 X_1 \to X_2 \G_2 X_2$ such that $c_2\circ\kappa = c_1$ on
$X_1 \G_1 X_1$. As in \cite{Matui}, if $X_1, X_2$ are clopen, we say $\G_1$ and $\G_2$
are \emph{Kakutani $c_1$--$c_2$ equivalent}. Theorem~\ref{thm:3} combined with the
results of \cite{CRS} gives the following.

\begin{cor}\label{cor:strong equiv ample}
Let $\Gamma$ be a discrete group, and let $(\G_1,c_1), (\G_2,c_2)$ be second-countable
$\Gamma$-graded ample Hausdorff groupoids such that each $\Io{c_i^{-1}(\id_\Gamma)}$ is
torsion-free and abelian. The following are equivalent:
\begin{enumerate}
\item there exists a graded $(\G_1, c_1)$--$(\G_2, c_2)$-equivalence $(Z, c_Z)$ such that
    $c_Z^{-1}(\id_\Gamma)$ is a $c_1^{-1}(\id_\Gamma)$--$c_2^{-1}(\id_\Gamma)$-equivalence;
\item there exist a second-countable ample Hausdorff groupoid $\G$, a grading $c$ of
    $\G$ by $\Gamma$ such that $\Io{c^{-1}(\id_\Gamma)}$ is torsion-free and abelian, a pair
    of complementary clopen $c^{-1}(\id_\Gamma)$-full subsets $X_1,X_2\subseteq \go$, and
    isomorphisms $\kappa_i: X_i \G X_i \to \G_i$ such that $X_1\cup X_2=\go$, and
    each $c_i \circ \kappa_i|_{X_i \G X_i} = c|_{X_i \G X_i}$;
\item there is an isomorphism $\kappa:\G_1\times\R\to\G_2\times\R$ such that
    $\bar{c}_2\circ\kappa=\bar{c}_1$;
\item $\G_1$ and $\G_2$ are Kakutani $c_1$--$c_2$ equivalent;
\item $\G_1$ and $\G_2$ are weakly Kakutani $c_1$--$c_2$ equivalent;
\item there is an equivariant $(C^*_r(\G_1), C_0(\G_1^{(0)}))$--$(C^*_r(\G_2),
    C_0(\G_2^{(0)}))$-imprimitivity bimodule $(X,\zeta)$ such that $X_{\id_\Gamma}$ is a
    $C^*_r(c_1^{-1}(\id_\Gamma))$--$C^*_r(c_2^{-1}(\id_\Gamma))$-imprimitivity bimodule;
\item there exist a separable $C^*$-algebra $A$, a coaction $\delta$ of $\Gamma$ on
    $A$, an abelian $C^*$-subalgebra $D\subseteq A^\delta$ containing an approximate unit for $A^\delta$, a pair of complementary
    $A^\delta$-full projections $P_1, P_2\in M(D)$, and isomorphisms $\phi_i :
    P_1AP_1\to C^*_r(\G_i)$ such that $P_1 A P_2 = \clsp\{P_1 n P_2 : n \in
    \Nn_\star(D)\}$ and such that $\phi_i(P_i D P_i) = D_i$ and $\delta_i \circ
    \phi_i = (\phi_i \otimes \id_{C^*_r(\Gamma)}) \circ \delta|_{P_i A P_i}$;
\item there is an isomorphism $\phi:C^*_r(\G_1) \otimes \Kk \to
    C^*_r(\G_2)\otimes\Kk$ satisfying $\phi(C_0(\G_1^{(0)}) \otimes \mathcal{C}) =
    C_0(\G_2^{(0)})\otimes \mathcal{C}$ and $(\delta_{c_2}\otimes\id_\Kk)\circ\phi =
    (\phi\otimes\id_{C^*_r(\Gamma)})\circ(\delta_{c_1}\otimes\id_\Kk)$;
\item there are $C^*_r(c_i^{-1}(e))$-full projections $p_i\in M(C_0(\G_1^{(0)}))$ and an
    isomorphism $\phi : p_1C^*_r(\G_1)p_1 \to p_2C^*_r(\G_2)p_2$ such that
    $\phi(p_1C_0(\G_1^{(0)})) = p_2C_0(\G_1^{(0)})$, and $\delta_{c_2} \circ \phi =
    (\phi\otimes\id_{C^*_r(\Gamma)}) \circ \delta_{c_1}$ on $p_1C^*_r(\G_1)p_1$; and
\item there are $C^*_r(c_i^{-1}(\id_\Gamma))$-full ideals $I_i \subseteq C_0(\G_i^{(0)})$ and an
    isomorphism of hereditary subalgebras $\phi : I_1C^*_r(\G_1)I_1 \to I_2C^*_r(\G_2)I_2$ such
    that $\phi(I_1)=I_2$ and $\delta_{c_2} \circ \phi = (\phi\otimes\id_{C^*_r(\Gamma)}) \circ
    \delta_{c_1}$ on $I_1C^*_r(\G_1)I_1$.
\end{enumerate}
\end{cor}
\begin{proof}
Theorem~\ref{thm:3} gives (1)$\,\iff\,$(2)$\,\iff\,$(6)$\,\iff\,$(7). For equivalence of
(1)~and~(3)--(5), we just summarise the modifications needed to the arguments of
\cite{CRS}. The proof of (3)$\,\implies\,$(1) follows the second paragraph of the proof
of \cite[Theorem~2.1]{CRS}: if $\kappa:\G_1\times\R\to\G_2\times\R$ is an isomorphism
satisfying $\bar{c}_2\circ\kappa=\bar{c}_1$, then each $\G_i \times (\{1\} \times \NN)$
is a $(\G_i, c_i)$--$(\G_i \times \R, \overline{c}_i)$-equivalence, so
Lemma~\ref{lem:equiv equiv} shows that $(\G_1, c_1)$ and $(\G_2, c_2)$ are equivalent.
The argument of \cite[Theorem~2.1]{CRS} gives (1)$\,\implies\,$(3) once we prove that
given $(\G, c)$ and a clopen $c^{-1}(\id_\Gamma)$-full $K \subseteq \G^{(0)}$ there is an
isomorphism $\G \times \R \cong K \G K \times \R$ that is equivariant for $\overline{c}$
and $\overline{c}|_{K\G K \times \R}$. For this, apply \cite[Lemma~2.4]{CRS} to
$c^{-1}(\id_\Gamma)$ and $K$ to obtain a bisection $Y \subseteq c^{-1}(\id_\Gamma) \times \R$ with range $K
\times \NN$ and source $\G^{(0)} \times \NN$; then conjugation by this $Y$ implements a
graded isomorphism $\G \times \R \cong K\G K \times \R$.

For (1)$\;\iff\;$(4)$\;\iff\;$(5), we follow the proof of \cite[Theorem~3.2]{CRS}. The
implications (4)$\,\implies\,$(5)$\,\implies\,$(1) follow the first paragraph of that
proof using Lemma~\ref{lem:equiv equiv} for (5)$\,\implies\,$(1). For
(1)$\,\implies\,$(4), we follow the proof of \cite[Theorem~3.2]{CRS} observing that since
$r(c_Z^{-1}(\id_\Gamma)) = \G_1^{(0)}$ and $s(c_Z^{-1}(\id_\Gamma)) = \G_2^{(0)}$, we can choose the bisections $V_i$ to
belong to $c_Z^{-1}(\id_\Gamma)$. Thus the bisection $Y$ obtained in the penultimate paragraph of the
proof satisfies $Y \subseteq c^{-1}(\id_\Gamma)$. As in the proof of \cite[Theorem~3.2]{CRS},
conjugation by $Y$ is an isomorphism $r(Y) \G_1 r(Y) \cong s(Y) \G_2 s(Y)$, and it is
graded because $Y \subseteq c^{-1}(\id_\Gamma)$.

Theorem~\ref{thm:1} gives (3)\;$\iff\;$(8), so it suffices to prove
(4)$\,\implies\,$(9)$\,\implies\,$(10)$\,\implies\,$(5).
For (4)$\,\implies\,$(9), suppose that $\G_1$ and $\G_2$ are Kakutani
$c_1$--$c_2$-equivalent with respect to $X_i \subseteq \G^{(0)}_i$ and $\kappa :
X_1\G_1X_1 \to X_2\G_2 X_2$. Then $p_i := 1_{X_i} \in M(C^*(\G_i))$ defines
$C^*_r(\G_i)$-full projections such that $p_i C^*_r(\G_i) p_i \cong C^*_r(X_i \G_i X_i)$,
and $\kappa$ induces a $\delta_{c_1}$--$\delta_{c_2}$-equivariant isomorphism $C^*(X_i
\G_i X_i) \to C^*(X_i \G_i X_i)$.

For (9)$\,\implies\,$(10), take $I_i = p_iC_0(\G_i^{(0)})$. For (10)$\,\implies\,$(5),
fix $I_i \subseteq C_0(\G_i^{(0)})$ and $\phi$ as in~(10). Each $I_i = C_0(U_i)$ for some
open $U_i \subseteq \G_i^{(0)}$. Since each $I_i$ is $C^*_r(c_i^{-1}(\id_\Gamma))$-full, each
$U_i$ is $c_i^{-1}(\id_\Gamma)$-full by \cite[Theorem~4.3.3]{SimsNotes}. As in the proof of
\cite[Theorem~3.4.4]{SimsNotes}, restriction of functions induces isomorphisms $R_i :
I_iC^*_r(\G_i)I_i \to C^*_r(U_i \G_i U_i)$. Now $\psi := R_2 \circ \phi \circ R^{-1}_1 :
C^*_r(U_1 \G_1 U_1) \to C^*_r(U_2 \G_2 U_2)$ is an equivariant
isomorphism mapping $C_0(U_1)$ onto $C_0(U_1)$. Hence (2)$\,\implies\,$(1) in
Theorem~\ref{thm:1} yields a graded isomorphism $\kappa : U_1 \G_1 U_1 \to U_2 \G_2 U_2$.
\end{proof}

\subsection{Stable continuous orbit equivalence}

Let $X$ be a locally compact Hausdorff space, and $\sigma : X \to X$ a local
homeomorphism. Let $\widetilde{X} := X \times \NN$ with the product topology and define a
(surjective) local homeomorphism $\tilde{\sigma} : \widetilde{X} \to \widetilde{X}$ by
$\tilde{\sigma}(x,0) = (\sigma(x), 0)$ and $\tilde{\sigma}(x, n+1) = (x,n)$. We call
$(\tilde{X}, \tilde{\sigma})$ the \emph{stabilisation} of $(X, \sigma)$. Then
$\G(\widetilde{X}, \tilde\sigma) \cong \G(X, \sigma) \times \R$ via $((x,m), p, (y,n))
\mapsto \big((x, p-m+n,y), (m,n)\big)$. We have the following partial generalisation of
\cite[Corollary~6.3]{CEOR}.

\begin{cor}
Let $\sigma : X \to X$ and $\tau : Y \to Y$, be local homeomorphisms of second-countable
locally compact totally disconnected Hausdorff spaces. The following are equivalent:
    \begin{enumerate}
    \item there is a stabiliser-preserving continuous orbit equivalence from
        $(\widetilde{X}, \tilde\sigma)$ to $(\widetilde{Y}, \tilde{\tau})$;
    \item $G(X, \sigma)$ and $G(Y, \tau)$ are equivalent;
    \item $(C^*(\G(X, \sigma)), C_0(X))$ and $(C^*(\G(Y, \tau)), C_0(Y))$ are Morita
        equivalent; and
    \item there is an isomorphism $C^*(\G(X, \sigma)) \otimes \Kk \to C^*(\G(Y,
        \tau)) \otimes \Kk$ that carries $C_0(X) \otimes \mathcal{C}$ to $C_0(Y) \otimes
        \mathcal{C}$.
    \end{enumerate}
\end{cor}
\begin{proof}
Theorem~\ref{thm:DR oe<->gi} shows that~(1) holds if and only if
$\G(\widetilde{X},\tilde\sigma) \cong \G(\widetilde{Y}, \tilde\tau)$, so the discussion
preceding the corollary shows that~(1) holds if and only if $\G(X, \sigma) \times \R
\cong \G(Y, \tau) \times \R$. Now the result follows from
(1)$\;\iff\;$(3)$\;\iff\;$(6)$\;\iff\;$(8) in Corollary~\ref{cor:strong equiv ample}
applied to the trivial cocycles on $\G_1 = \G(X, \sigma)$ and $\G_2 = \G(Y, \tau)$.
\end{proof}

Let $\sigma:X\to X$ be a surjective local homeomorphism of a compact Hausdorff space. Let
$\TS{X} := \{\xi \in X^{\ZZ}:\sigma(\xi_n)=\xi_{n+1}\text{ for every }n\in\ZZ\}$, and
define $\TS{\sigma}:\TS{X}\to\TS{X}$ by $\TS{\sigma}(\xi)_n = \sigma(\xi_n)$. We call
$\sigma$ \emph{expansive} if there is a metrisation $(X,d)$ of $X$ and an $\epsilon>0$
such that $\sup_n d(\sigma^n(x),\sigma^n(x')) < \epsilon \implies x = x'$. We call
$\epsilon$ an \emph{expansive constant for $(X, d, \sigma)$}. The following generalises
\cite[Theorem~5.1]{CR}.

\begin{cor}\label{cor:two-sided conjugacy}
Let $X, Y$ be second-countable locally compact totally disconnected Hausdorff spaces and
let $\sigma : \dom(\sigma) \to \ran(\sigma)$ and $\tau : \dom(\tau) \to \ran(\tau)$ be
local homeomorphisms between open subsets of $X,Y$. The following are equivalent:
    \begin{enumerate}
    \item there are continuous, open maps $f:X\to Y$ and $f':Y\to X$, and continuous
        maps $a:X\to\NN$, $k:\dom(\sigma)\to\NN$, $a':Y\to\NN$, and
        $k':\dom(\tau)\to\NN$ such that $\sigma^{a(x)}(f'(f(x)))=\sigma^{a(x)}(x)$
        for $x\in X$, $\tau^{k(x)}(f(\sigma(x)))=\tau^{k(x)+1}(f(x))$ for $x\in
        \dom(\sigma)$, $\tau^{a'(y)}(f(f'(y)))=\tau^{a'(y)}(y)$ for $y\in Y$, and
        $\sigma^{k'(y)}(f'(\tau(y)))=\sigma^{k'(y)+1}(f'(y))$ for $y\in \dom(\tau)$;
    \item there is a graded $(G(X,\sigma),c_X)$--$(G(Y,\tau),c_Y)$-equivalence $(Z,
        c_Z)$ such that $c_Z^{-1}(0)$ is a $c_X^{-1}(0)$--$c_Y^{-1}(0)$-equivalence;
    \item there is an isomorphism $\kappa :G(X,\sigma)\times\R\to G(Y,\tau)\times\R$
        such that $\TS{c}_Y\circ\kappa=\TS{c}_X$;
    \item there is a $\TT$-equivariant Morita equivalence between $(C^*(\G(X,
        \sigma)), C_0(X))$ and $(C^*(\G(Y, \tau)), C_0(Y))$ with respect to the gauge
        actions $\gamma^X$ and $\gamma^Y$ whose fixed-point submodule is a $C^*(\G(X,
        \sigma)$--$C^*(\G(Y, \tau)$-imprimitivity bimodule; and
    \item there is an isomorphism $C^*(\G(X, \sigma)) \otimes \Kk(\NN) \to C^*(\G(Y,
        \tau)) \otimes \Kk(\NN)$ that carries $C_0(X) \otimes \mathcal{C}$ to $C_0(Y) \otimes
        \mathcal{C}$, and intertwines the actions $\gamma^X \otimes \id$ and $\gamma^Y
        \otimes \id$.
    \end{enumerate}
    If $X$ and $Y$ are compact, $\dom(\sigma)=\ran(\sigma)=X$, and
    $\dom(\tau)=\ran(\tau)=Y$, then each of the above five conditions implies
    \begin{itemize}
        \item[(6)] $(\TS{X},\TS{\sigma})$ and $(\TS{Y},\TS{\tau})$ are conjugate.
    \end{itemize}
    If in addition $\sigma$ and $\tau$ are expansive, then all six conditions are
    equivalent.
\end{cor}

\begin{rmk}
If $(X,d)$ is a totally disconnected metric space and $\sigma:X\to X$ is a surjective
expansive local homeomorphism, then $(X,\sigma)$ is conjugate to the edge shift of a
finite graph with no sinks or sources (see \cite[Section~1]{BW} and
\cite[Theorem~1]{IT}), so the equivalence of~(2), (5)~and~(6) follows from
Theorem~\ref{thm:ev conj equivalences} and \cite[Theorem 5.1]{CR}.
\end{rmk}

\begin{proof}[Proof of Corollary~\ref{cor:two-sided conjugacy}]
Equivalence of (2)--(5) follows from equivalence of statements~(1), (3), (6)~and~(8) in
Corollary~\ref{cor:strong equiv ample}.

$(1)\implies (2)$. For each $n\in\NN$, fix a countable family $I_n$ of mutually disjoint
compact open subsets of $X$ such that $\sigma^n|_U$ is a homeomorphism for each $U\in
I_n$ and $\dom(\sigma^n)=\bigcup_{U\in I_n}U$. Fix an injection $\iota : \bigcup_n \{n\}
\times I_n \to \NN$. For $x \in X$, let $U_x$ be the element of $I_{a(x)}$ containing
$x$, and let $b(x) = \iota(a(x),U_x)$. Let $Z:=\{(f(x),b(x)): x\in X\} \subseteq Y \times
\NN$. Since $a$ is continuous and $f$ is open, $Z$ is open.

Fix $y\in Y$. Since $\tau^{a'(y)}(f(f'(y)))=\tau^{a'(y)}(y)$, we see that
$(y,0,f(f'(y)))\in G(Y,\tau)$. Thus $Z$ is $\TS{c}_Y^{-1}(0)$-full in
$(G(Y,\tau)\times\R)^{(0)}$. Corollary~\ref{cor:strong equiv ample} yields a graded
$G(Y,\tau)$--$Z (G(Y,\tau)\times\R) Z$-equivalence $(W, c_W)$ such that $c^{-1}_W(0)$ is
a $c_Y^{-1}(0)$--$\TS{c}_Y^{-1}(0)$-equivalence.

Since $\sigma^{a(x)}(f'(f(x)))=\sigma^{a(x)}(x)$ for $x\in X$, the formula $h(x) :=
(f(x),b(x))$ gives a homeomorphism $h:X\to Z$. Since each $\tau^{k(x)}(f(\sigma(x))) =
\tau^{k(x)+1}(f(x))$, we have $(f(x),n,f(x'))\in G(Y,\tau)$ for all $(x,n,x')\in
G(X,\sigma)$. There is an injective homomorphism $\phi : G(X,\sigma) \to
G(Y,\tau)\times\R$ given by $\phi(x,n,x') := ((f(x),n,f(x')),(b(x),b(x')))$. We have
$\phi(G(X,\sigma)) \subseteq Z(G(Y,\tau)\times\R)Z$. Since $h$ is a homeomorphism and
$\sigma^{k'(y)}(f'(\tau(y)))=\sigma^{k'(y)+1}(f'(y))$ for $y\in Y$, we see that
$\phi(G(X,\sigma))= Z(G(Y,\tau)\times\R)Z$, giving~(2).

$(3)\implies (1)$. Write $\pi_X : G(X,\sigma)\times\R \to G(X,\sigma)$ and $\pi_Y :
G(Y,\tau)\times\R \to G(Y,\tau)$ for the projection maps. Define $f : X\to Y$ and $f' : Y
\to X$ by $\pi_Y(\kappa( (x,0,x),(0,0))) = (f(x),0,f(x))$ and
$\pi_X(\kappa^{-1}((y,0,y),(0,0))) = (f'(y),0,f'(y))$. Clearly $f$ and $f'$ are
continuous and open. Define $l_X:G(X,\sigma)\to\NN$ and $l_Y:G(Y,\tau)\to\NN$ as in
Lemma~\ref{lem:lmap}. Fix $x\in X$. Then $\kappa( (x,0,x),(0,0))=( (f(x),0,f(x)),(n,n))$
for some $n\in\NN$. Since $\TS{c}_Y\circ\kappa=\TS{c}_X$, it follows that
$\pi_X(\kappa^{-1}( (f(x),0,f(x)),(n,0)))=(x,0,f'(f(x)))$. Let $a(x):=l_X(x,0,f'(f(x)))$.
Then $a : X \to \NN$ is continuous, and $\sigma^{a(x)}(f'(f(x)))=\sigma^{a(x)}(x)$ for
all $x$. Similarly, $a' : Y\to\NN$ defined by $a'(y):=l_Y(y,0,f(f'(y)))$ is continuous,
and $\tau^{a'(y)}(f(f'(y)))=\tau^{a'(y)}(y)$ for all $y$.

Fix $x\in\dom(\sigma)$. Then $\pi_Y(\kappa((x,1,\sigma(x)),(0,0))) =
(f(x),1,f(\sigma(x)))$ because $\TS{c}_Y\circ\kappa=\TS{c}_X$. Let
$k(x):=l_Y(f(x),1,f(\sigma(x)))$. So $k :\dom(\sigma) \to \NN$ is continuous, and
$\tau^{k(x)}(f(\sigma(x)))=\tau^{k(x)+1}(f(x))$ for $x\in \dom(\sigma)$. Similarly,
$k'(y) := l_Y(f'(y),1,f'(\tau(y)))$ defines a continuous $k' : \dom(\tau) \to \NN$ such
that $\sigma^{k'(y)}(f'(\tau(y))) = \sigma^{k'(y)+1}(f'(y))$.

$(1)\implies (6)$. Since $X$ is compact, $m := \sup_{x \in X} k(x)$ is finite. We have
$\tau^m(f(\sigma(x))) = \tau^{m-k(x)}(\tau^{k(x)+1}(f(x))) = \tau^{m+1}(f(x))$ for all
$x$. So $\phi:=\tau^m\circ f : X \to Y$ is continuous and satisfies $\phi(\sigma(x)) =
\tau^m(f(\sigma(x))) = \tau^{m+1}(f(x)) = \tau(\phi(x))$ for all $x$. Thus
$\TS{\phi}(\xi) := (\phi(\xi_n))_{n\in\ZZ}$ defines a continuous map $\TS{\phi} : \TS{X}
\to \TS{Y}$. By definition, $\TS{\tau}\circ\TS{\phi}=\TS{\phi}\circ\TS{\sigma}$. We will
show that $\TS{\phi}$ is bijective and hence a conjugacy.

For injectivity, suppose that $\TS{\phi}(\xi) = \TS{\phi}(\xi')$. Then each
$\tau^m(f(\xi_n)) = \phi(\xi_n) = \phi(\xi'_n) = \tau^m(f(\xi'_n))$. Let $p :=
\max\{\sup_{z \in X} a(z), \sup_{y \in Y} k'(y)\}$. Then $\sigma^p(x) =
\sigma^p(f'(f(x))$ for $x\in X$, and $\sigma^{p+j}(f'(y)) = \sigma^{p}(f'(\tau^j(y)))$
for $y\in Y$, $j \in \NN$. So for $n \in \ZZ$,
\[
\xi_{n+p+m} = \sigma^{p+m}(\xi_n) = \sigma^{p}(f'(\tau^m(f(\xi_n))))
    = \sigma^{p}(f'(\tau^m(f(\xi'_n)))) = \sigma^{p+m}(\xi'_n) = \xi'_{n+p+m}.
\]
Hence $\xi = \xi'$, and we deduce that $\TS{\phi}$ is injective.

For surjectivity, fix $\eta\in\TS{Y}$. Let $q := \sup_{z \in Y} \max\{a'(z), k'(z)\}$,
and put $\eta' := \TS{\tau}^{-m-q}(\eta)$. For $y \in Y$, $\sigma^q(f'(\tau(y))) =
\sigma^{q-k'(y)}(\sigma^{k'(y)}(f'(\tau(y)))) = \sigma^{q+1}(f'(y))$. Thus
$(\sigma^q(f'(\eta'_n)))_{n\in\ZZ}\in\TS{X}$. Since $\tau^{m+j}(f(x)) =
\tau^m(f(\sigma^j(x)))$ for all $x,j$, and since $\tau^q(f(f'(y)))=\tau^q(y)$ for $y\in
Y$,
\[
\phi(\sigma^q(f'(\eta'_n)))=\tau^m(f(\sigma^q(f'(\eta'_n))))
    =\tau^{m+q}(f(f'(\eta'_n)))=\tau^{m+q}(\eta'_n)
\]
for all $n \in \ZZ$. So $\TS{\phi}((\sigma^q(f'(\eta_n)))_{n\in\ZZ}) =
\TS{\tau}^{m+q}(\eta') = \eta$.

$(6)\implies (1)$. Suppose that $(X, d)$ and $(Y, d')$ are metrisations of $X, Y$ and
that $\epsilon, \epsilon'$ are expansive constants for $(X, D, \sigma)$ and $(Y, d',
\tau)$. Suppose that $\psi:\TS{X}\to\TS{Y}$ is a conjugacy. Fix $\delta>0$ such that
$\TS{d}(\xi, \xi') < \delta\implies \TS{d}(\psi(\xi),\psi(\xi')) < \epsilon$. Let $M :=
\sup_{x,x' \in X} d(x,x') < \infty$, and fix $N$ with $2^{-N+1}M < \delta$. Then $\TS{d}(
\xi, \xi')<\delta$ whenever $\xi_n=\xi'_n$ for all $n> -N$. So given $\xi, \xi' \in
\TS{X}$, and putting $\eta = \psi(\xi)$ and $\eta' = \psi(\xi')$, if $\xi_n = \xi'_n$ for
$n> -N$, then $d'(\eta_m,\eta'_m) \le \TS{d}'(\TS{\tau}^m(\eta), \TS{\tau}^m(\eta')) <
\epsilon$ for all $m$, giving $\eta_0 = \eta'_0$. So we can define $f:X\to Y$ by
$\psi(\xi) = (f(\xi_{n-N}))_{n\in\ZZ}$, and then $f\circ\sigma=\tau\circ f$.

To see that $f$ is continuous, fix $x\in X$ and $\gamma>0$. Choose $\xi\in\TS{X}$ such
that $x=\xi_{-N}$. Choose an open neighbourhood $U$ of $\xi$ in $\TS{X}$ such that
$\TS{d}'(\psi(\xi), \psi(\xi') < \gamma$ for every $\xi' \in U$. Since $\sigma$ is
surjective and open, there is an open $V \owns x$ such that $V \subseteq \{\xi'_{-N} :
\xi' \in U\}$. So $d'(f(x),f(x')) \le \sup_{\xi' \in U} \TS{d}'(\psi(\xi), \psi(\xi')) <
\gamma$ for all $x' \in V$.

To see that $f$ is open, fix $V \subseteq X$ open. Then $U := \{\xi\in\TS{X} :
\xi_{-N}\in V\}$ is open, so $\psi(U)$ is open. The coordinate projections $\TS{Y} \to Y$
are open maps, so $f(V)$ is open.

So $f$ is continuous and open, $\tau\circ f = f\circ \sigma$, and $\psi(\xi) =
(f(\xi_{n-N}))_{n\in\ZZ}$ for $\xi\in\TS{X}$. Symmetry gives a continuous open $f' :
(Y,\tau) \to (X,\sigma)$ with $\sigma \circ f' = f' \circ \tau$ and $\psi^{-1}(\eta) =
(f'(\eta_{n-N'}))_{n\in\ZZ}$ for $\eta\in\TS{Y}$. We have $f\circ f'=\sigma^{N+N'}$ and
$f\circ f'=\tau^{N+N'}$, giving~(1).
\end{proof}

\begin{rmk}
We deduce an interesting ``stable isomorphism implies isomorphism" statement. Let
$\sigma:X\to X$ and $\tau:Y\to Y$ be homeomorphisms of second-countable totally
disconnected compact Hausdorff spaces. If $(C(X)\times_\sigma\ZZ) \otimes \mathcal{K}$
and $(C(Y)\times_\tau\ZZ) \otimes \mathcal{K}$ are $(\gamma^X \otimes \id)$--$(\gamma^Y
\otimes \id)$-equivariantly isomorphic, then there is a
$\gamma^X$--$\gamma^Y$-equivariant (and diagonal-preserving by Remark~\ref{rmk:KOQstuff})
isomorphism $C(X)\times_\sigma\ZZ \to C(Y)\times_\tau\ZZ$.

To prove this, first observe that $(X,\sigma) \cong (\TS{X},\TS{\sigma})$ and $(Y,\tau)
\cong (\TS{Y},\TS{\tau})$ and then apply \mbox{(5)$\,\implies\,$(6)} of
Corollary~\ref{cor:two-sided conjugacy} to see that if there is a
\emph{diagonal-preserving} equivariant isomorphism $(C(X)\times_\sigma\ZZ) \otimes
\mathcal{K} \cong (C(Y)\times_\tau\ZZ) \otimes \mathcal{K}$, then there is an equivariant
isomorphism $C(X)\times_\sigma\ZZ \cong C(Y)\times_\tau\ZZ$. Next note that
$((C(X)\times_\sigma\ZZ) \otimes \Kk)^{\gamma^X \otimes \id} = C(X) \otimes \Kk$ and
likewise for $(Y, \tau)$; so any equivariant isomorphism $\phi : C(X)\times_\sigma\ZZ \to
C(Y)\times_\tau\ZZ$ carries $C(X) \otimes \mathcal{C}$ to a maximal abelian subalgebra of $C(Y)
\otimes \Kk$; that is, to $C(Y) \otimes D$ for some maximal abelian $D \subseteq \Kk$.
Fix a unitary $U \in M(\Kk)$ that conjugates $D$ to $\mathcal{D}$. Then $\phi' := (\id \otimes
\Ad_U) \circ \phi$ is a diagonal-preserving equivariant isomorphism.

This result could also be obtained without recourse to groupoids using the techniques of
\cite[Proposition~4.3]{KOQ} (it is at least implicitly contained in that result).
\end{rmk}

\end{document}